\documentclass[leqno]{article}
\usepackage{authblk}
\usepackage{amsmath,amsthm,amssymb,latexsym,graphicx, mathtools, enumerate,bm,cite,mathrsfs,todonotes,esint}

\usepackage{graphicx}

\usepackage[margin=1in]{geometry}

\usepackage{hyperref}
\hypersetup{
    colorlinks=true,
    allcolors=blue,
}

\newcommand{\normm}[1]{{\left\vert\kern-0.25ex\left\vert\kern-0.25ex\left\vert #1 
    \right\vert\kern-0.25ex\right\vert\kern-0.25ex\right\vert}}

\newcommand{\EE}{{\mathbb E}}
\newcommand{\NN}{{\mathbb N}}

\newcommand{\RR}{{\mathbb R}}
\newcommand{\ZZ}{{\mathbb Z}}

\newcommand{\abs}[1]{ \left| #1 \right|}

\newcommand{\cadlag}{c\`{a}dl\`{a}g}

\newcommand{\dif}{\mathrm d}
\newcommand{\dist}{\on{dist}}
\newcommand{\eps}{\varepsilon}

\newcommand{\norm}[1]{ \left\| #1 \right\| }
\newcommand{\nor}[2]{\left\|#1\right\|_{#2}}
\newcommand{\oline}[1]{\overline{#1}}
\newcommand{\oo}{\infty}

\newcommand{\I}{\mathrm{I}}
\newcommand{\II}{\mathrm{II}}
\newcommand{\III}{\mathrm{III}}

\DeclareMathOperator*{\esssup}{ess\,sup}

\newcommand{\mcl}{\mathcal}
\newcommand{\mbb}{\mathbb}
\newcommand{\mbf}{\mathbf}
\newcommand{\on}{\operatorname}

\newtheorem{lemma}{Lemma}
\newtheorem{proposition}{Proposition}
\newtheorem{theorem}{Theorem}
\newtheorem{corollary}{Corollary}
\newtheorem{definition}{Definition}
\newtheorem{remark}{Remark}

\numberwithin{equation}{section}
\numberwithin{lemma}{section}
\numberwithin{proposition}{section}
\numberwithin{theorem}{section}
\numberwithin{corollary}{section}
\numberwithin{definition}{section}
\numberwithin{remark}{section}

\begin{document}

\title{Besov rough path analysis} 
\author{Peter K. Friz}
\affil{TU Berlin and WIAS Berlin\\
\href{mailto:friz@math.tu-berlin.de}{\nolinkurl{friz@math.tu-berlin.de}}}
\author{Benjamin Seeger
}
\affil{Universit\'e Paris-Dauphine and Coll\`ege de France \\ 
\href{mailto:seeger@ceremade.dauphine.fr}{\nolinkurl{seeger@ceremade.dauphine.fr}}}

\author{With an appendix by Pavel Zorin-Kranich}
\affil{University of Bonn\\
\href{mailto:pzorin@uni-bonn.de}{\nolinkurl{pzorin@uni-bonn.de}} }

\maketitle

\begin{abstract}
Rough path analysis is developed in the full Besov scale. This extends, and essentially concludes, an investigation started by {\em [Pr\"omel--Trabs,  Rough differential equations driven by signals in {B}esov spaces. J. Diff. Equ. 2016]}, further studied in a series of papers by Liu, Pr\"omel and Teichmann. A new Besov sewing lemma, a real-analysis result of interest in its own right, plays a key role, and the flexibility in the choice of Besov parameters allows for the treatment of equations not available in the H\"older or variation settings. Important classes of stochastic processes fit in the present framework.

\end{abstract}

\tableofcontents

\section{Introduction}

Rough path theory gives meaning to differential equations of the form
\begin{equation}
    dY_t = f_0 (Y_t) d t + \sum_{i=1}^n f_i(Y_t) d X^i_t,   \label{equ:intro1}
\end{equation}
where $X$ is an $n$-dimensional path of regularity $\alpha$. When $\alpha > 1/2$, the equation can be understood as a Young integral equation, which covers both the $\alpha$-H\"older and $p$-variation setting, where $p = 1/\alpha$. When $\alpha \le 1/2$, or $p\ge 2$, it was understood by T. Lyons \cite{lyonsrough98} that $X$ needs to be enhanced with additional information to restore well-posedness of the problem. The resulting rough path interpretation reads
\begin{equation}
    dY_t = f_0 (Y_t) d t +f (Y_t) d {\mbf X},  \label{equ:intro2}
\end{equation}
where the object ${\mbf X}$ should be thought of as the original path $(X^1,\ldots,X^n)$ enhanced with sufficient extra information, typically interpreted as iterated integrals, to regain analytic well-posedness. 
The H\"older case with $\alpha > 1/3$ is found e.g. in \cite{friz2020course}, the general ``geometric'' case (that is, the enhanced object $\mbf X$ satisfies a first-order calculus) with $\alpha > 0$, both in the H\"older and variation cases, is found in \cite{friz2010multidimensional}, and previous works by Lyons and coworkers focused on the continuous $p$-variation setting.\footnote{The case of discontinuous $p$-variation rough paths, as required for stochastic processes with jumps, is more recent \cite{friz2018differential}.} The non-geometric setting can be analyzed using branched structures \cite{gubinelli2010ramification}, but see also \cite{hairer2015geometric} for a reduction to the geometric case.

This work is devoted to Besov rough path analysis, that is, we consider driving signals belonging to the Besov space $B^s_{pq}$ with $0 < s < 1$, $1/s < p \le \oo$, and $0 < q \le \oo$.
At least morally, the H\"older and variation settings appear as end-points in a Besov-scale of rough path spaces, as indicated in Figure \ref{fig:Tao}, in which regularity $s \in [0,1]$ is plotted against inverse integrability $1/p \in [0,1]$.
When $p = q = \oo$, $B^s_{\oo,\oo} = C^s$, the space of $s$-H\"older continuous paths. On the other hand, we have the embeddings $B^{1/p}_{p,1} \subset V^p \subset B^{1/p}_{p,\oo}$ (see Proposition \ref{P:variationembedding} for more precise statements). For $1/p < s \le \oo$, $B^s_{pq}$ embeds continuously into $C^{s-1/p}$, and we prove a generalization of this result to paths taking values in a general metric space (Proposition \ref{P:Besovembedding}). Rough analysis begins to come into play when $s \le 1/2$, and requires an enhanced state space $\mbf B^s_{pq}$, depending on the level of roughness $[1/s]$ (see Definition \ref{D:BesovRP}).

In view of these embeddings, the sheer task of solving \eqref{equ:intro2} driven by Besov rough paths can be accomplished by embedding rough Besov spaces into rough H\"older or variation rough path spaces (see \cite{friz2006embedding,friz2018rough}, also Section \ref{sec:HoelderEmb}.) 
However, in the first case, due to the loss of H\"older regularity for finite $p$, this requires a strong regularity requirement (a large lower bound for $s$), while the second approach (which, in some sense, is taken in \cite{friz2018rough,liu2020sobolev}) provides no estimates with Besov (rough-path) metrics. 

One of the main contributions of this paper is the well-posedness of RDEs driven by Besov rough paths, with local Lipschitz estimates of the solution (a.k.a. Lyons-It\^o) map in the correct Besov spaces (see Theorems \ref{T:Youngeq}, \ref{T:YoungItoMap}, \ref{T:RDE}, \ref{T:IL}, and \ref{T:RDElevelN}):

\begin{theorem} \label{thm:mainIntro} If $0 < \alpha < 1$, $1/\alpha < p \le \oo$, and $0 < q \le \oo$, then, under natural regularity assumptions on the vector fields, there is a unique solution flow to \eqref{equ:intro2} such that the Lyons-It\^o map ${\mbf X} \mapsto Y$ is locally Lipschitz continuous in the full $B^\alpha_{p,q}$-Besov scale.
\end{theorem}

Moreover, we succeed in solving \eqref{equ:intro2} for Besov driving signals for which the Besov-variation embedding is too crude to allow for the use of the variation techniques, even in the Young regime. For example, we may treat $X \in B^{1/2}_{pq}$, $q \le 2 < p$, with Young integration (Theorem \ref{T:Youngeq}). On the other hand, the variation embedding \cite[Proposition 4.1(3)]{chong2020characterization} gives $X \in \mcl V^2$, which falls outside of the Young regime in the variation setting, and indeed, there exists $X \in B^{1/2}_{pq}$ with $q \le 2 < p$ such that $X$ does not belong to $\mcl V^r$ for any $r \in [1,2)$ (see Proposition \ref{P:strictvariation}).


%
%
%

\begin{figure}
\centering
\includegraphics[scale=0.7]{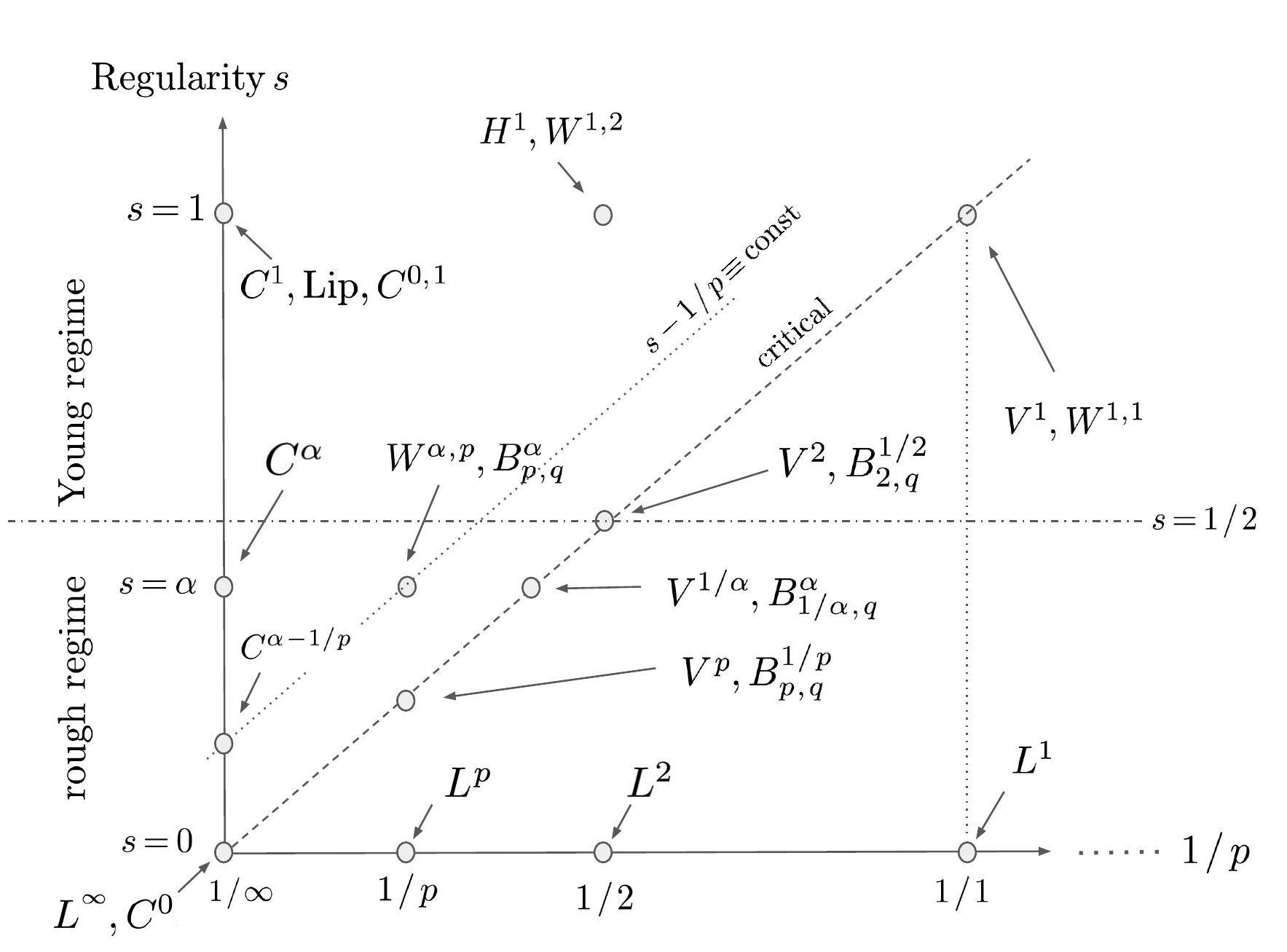}
\caption{{\em Type diagram for rough path spaces.} Note monotonicity with respect to both the regularity and - in view of the compact domain $[0,T]$ - integrability parameters. The dotted line helps to visualize the {\em Besov--H\"older embedding}; the dashed line hints at a close connection between variation and {\em critical} Besov spaces (cf. Section \ref{sec:criticalBesov}).}
\label{fig:Tao}
\end{figure}


Rough differential equations in the Besov scale with $\alpha > 1/3$, $p > 3$, and $q \ge 1$ were studied via paracontrolled distributions in \cite{proemel2016rough}. (The authors comment in detail on the difficulties of this approach for general $\alpha > 0$).
In \cite{friz2018rough}, results are obtained for RDEs in a Besov--Nikolskii type scale, although the interpolation of non-linear rough paths spaces of H\"older and variation used therein fails to yield the precise estimates in the Nikolski scale $N^{\alpha,p} \equiv B^\alpha_{p,\infty}$. Further progress on RDEs in the Besov--Sobolev scale $B^\alpha_{p,p} = W^{\alpha,p}$, notably existence and uniqueness, is made in \cite{liu2020sobolev}, but estimates of the solution map in terms of the correct Besov-Sobolev norm appeared as beyond reach of that paper's approach \cite[Remark 5.3]{liu2020sobolev}.
Theorem \ref{thm:mainIntro} essentially completes this line of investigation. 

\subsection{Sewing in the Besov scale} Our method differs from all the aforementioned works, in that we incorporate the Besov scale from the very beginning of the analysis, at the level of {\em sewing}.
We recall that sewing \cite{feyel2006curvilinear, gubinelli2004controlling,feyel2008noncommutative} is a real analysis method that not only underlies rough integration, but has become a most versatile tool in the field of rough analysis. It yields a generalized integration map
\[
	\left( A_{st} \right)_{0 \le s \le t \le T} \mapsto \left( \int_s^t A_{r,r+dr} \right)_{0 \le s \le t \le T}
\]
for appropriate two-parameter maps $A$, and also provides a precise estimate for the error
\[
	\int_s^t A_{r,r+dr} - A_{st}.
\]
A central contribution of this paper (Theorems \ref{T:Besovsewing}, \ref{T:Besovsewingendpoint}, and \ref{T:Besovsewingcontinuous}) is to generalize the sewing result to allow for two-parameter maps $A$ measured with Besov type regularity (see Definition \ref{D:Besovtwoparameter}): the map $A$ is taken to satisfy
\begin{equation}\label{introBesovA}
	\left[ \int_0^T \left(\frac{ \sup_{0 \le h \le \tau} \sup_{0 \le \theta \le 1} \nor{A_{\cdot,\cdot + h} - A_{\cdot,\cdot + \theta h} - A_{\cdot + \theta h, \cdot + h} }{L^p([0,T-h]} }{\tau^\gamma} \right)^q \frac{d \tau}{\tau} \right]^{1/q} < \oo
\end{equation}
for sufficient parameters $\gamma,p,q > 0$ (with the obvious modification if $q = \oo$).

When $p = q = \oo$, the requirement \eqref{introBesovA} reads as the standard H\"older type condition
\[
	\sup_{0 \le s < u < t \le T} \frac{ |A_{st} - A_{su} - A_{ut} |}{|t-s|^\gamma} < \oo,
\]
and the sewing procedure can be carried out as long as $\gamma > 1$. In the Besov setting, we similarly require that $\gamma > 1$, in addition to $\gamma > \frac{1}{p}$. This secondary condition is more than just technical, and it turns out, in the context of solving \eqref{equ:intro2}, to rule out exactly those regimes of Besov regularity that allow for \emph{jump discontinuities}. Indeed, this is to be expected, in view of the analysis in \cite{friz2018differential}, which explains that the presence of jumps requires additional augmented information.

We note, however, that more flexibility is allowed in the regularity parameter $\gamma$ if the secondary integration parameter $q$ is tuned sufficiently small (thus strengthening the condition \eqref{introBesovA}). In particular, if $0 < q \le 1 \wedge p$, then the sewing map can be constructed even if $\gamma = 1 \vee \frac{1}{p}$. One interesting application is the refinement in interpreting multiplication in Besov spaces. In particular, in Theorem \ref{T:YoungBesov} below, the Young integral
\[
	(f,g) \mapsto \int_0^\cdot f_r dg_r
\]
may be defined for $f$ and $g$ belonging to critical Besov spaces containing possibly discontinuous functions, for example, $f, g \in B^{1/2}_{2,2}$ (see Remark \ref{R:Youngintegrals}).

A first major difficulty faced in adapting sewing to the Besov setting is the a priori impossibility of point evaluation, while sewing, which is based on the Riemann-sum type limit
\begin{equation}\label{pointwiseRiemann}
	\int_s^t A_{r,r+dr} := \lim_{\norm{P} \to 0} \sum_{[u,v] \in P} A_{uv}, \quad P \text{ a partition of } [s,t],
\end{equation}
has an inherently pointwise character. We work around this challenge by retooling the viewpoint of Riemann sum: for a fixed partition $P$ of the unit interval $[0,1]$, we instead consider the map
\[
	(s,t) \mapsto \sum_{[u,v] \in P} A_{s + u(t-s), s + v(t-s)}
\]
whose limit in the generalized Besov metric, as $\norm{P} \to 0$, may be more easily analyzed.

Another obstacle to overcome is that Besov spaces do not (a priori) form an algebra, which complicates the task of using the Besov sewing result to construct a unique fixed point to solve RDEs. More precisely, the technique for forming solutions of \eqref{equ:intro2} is to apply the Besov sewing procedure to maps of the form
\[
	A_{st} := f_0(Y_s) (t-s) + \sum_{i=1}^n f_i(Y_s) (X^i_t - X^i_s) + \text{ (higher order terms) },
\]
yet, in the multiplication of $f(Y)$ with $X$, integrability may be lost. A fine interplay with embeddings is required to ``close the loop''. This requires not only classical embeddings as in Proposition \ref{P:Besovembedding}, but also a generalization to two parameter maps belonging to $\mbb B^\alpha_{pq}$ (Proposition \ref{P:BBesovembedding}).

Besov sewing is likely to prove interesting in its own right. Applications of sewing range from RDEs, to rough partial differential equations \cite{GT10, andris2019hoermander}, mean-field RDEs  \cite{bailleul2020solving}, to the analysis of level sets in the Heisenberg group \cite{trevisan2018rough}, and even as a
most effective replacement of It\^o's lemma when put in a martingale context: L\^e's {\em stochastic sewing} \cite{le2020stochastic}. Besov sewing, and possible ramifications thereof along the said extension, is thus likely to be of interest beyond the precise application presented in this work. We finally note that the works \cite{hairer2017reconstruction, singh2018elementary} in the context of regularity structures suggest higher-dimensional generalization of sewing in the spirit of \cite{caravenna2020hairer}, but such investigations are not the purpose of this paper.

\subsection{Stochastic processes as Besov rough paths}

As always, (multdimensional) Brownian motion, with iterated integrals in It\^o or Stratonovich sense, is the prototypical example of a (level-$2$) rough path, a.k.a.\ Brownian rough path. Its precise ``$(1/2, \infty-, \infty)$'' regularity improves the standard (rough path) H\"older regularity $(1/2-,\infty,\infty)$. To the best of our knowledge, and despite several works on Besov rough paths, this result (Theorem \ref{thm:BMbesovrough}), despite its fundamental character,  appears to be new. (Our only explanation for this is that the original proof of the Brownian motion case \cite{MR1277166} is based on Haar wavelets and it does not easily extend to rough paths.) The argument also extends to fractional Brownian motion, as we briefly point out. (Although this is not pursued here, the works \cite{friz2016jain, kerkyacharian2018regularity} make us confident that a Gaussian rough path theory in the Besov scale is possible, which somewhat interpolates between the H\"older and variational theory introduced in \cite{friz2010differential}.)

\medskip

Another important class of Besov rough paths arises from semimartingales 
\cite{coutin2005semi,friz2008burkholder,liu2018examples,MR3909973}. The key issue here is to handle the local martingale part, which can be done in quantitative way employing ideas from harmonic analysis, developed in a series of papers \cite{friz2008burkholder,MR3909973,MR4003122,arxiv:2008.08897} with focus on $p$-variation (rough path) metrics. We formulate as Theorem \ref{thm:Martbesovrough} the corresponding statements in the Besov scale, relying crucially on the material of Appendix \ref{appendix}, kindly supplied to us by Pavel Zorin--Kranich.

Semimartingales (and more generally Dirichlet processes) are obtained from a local martingale $M$ by adding some path $V \in B^{\alpha'}_{p',q'} \subset V^{1/\alpha'}$, provided $\alpha' - 1/p' > 0$.
This requires understanding the integrals $\int \delta M \otimes d V, \int \delta V \otimes d M, \int \delta V \otimes d V$ in the $2$-parameter Besov scale, which is possible using the analytic estimate given in Section \ref{sec:YI}.
 
 \medskip
 
 For the sake of completeness we note that Besov regularity of (Feller) of Markov processes has been studied by several authors, e.g. \cite{herren1997levy, schilling1997feller}, also \cite{fageot2017besov} for a recent contribution. Studying the Besov rough path regularity of
 such processes is not the purpose of this paper. Note however that many Feller diffusions can be constructed as It\^o SDE solutions, which coincide with the solution provided by rough path theory, with It\^o Brownian rough driver, whenever both theories apply (see e.g. \cite[Sec. 9.1]{friz2020course}). It is then clear that Besov regularity results for RDE solutions, as supplied in this paper, have immediate implications for the Besov regularity of so obtained Markov processes.

\subsection{Notation}

Throughout, fix a finite time horizon $T > 0$ and set $\Delta_d(s,t) = \left\{ (r_1,r_2,\ldots, r_d) \in [s,t]^d : r_1 \le r_2 \le \cdots \le r_d \right\}$, 
given $s < t$ in $[0,T]$ and some integer $d$.  
 Given $f: [0,T] \to \RR^m$, we define $\delta f: \Delta_2(0,T) \to \RR^m$ by $\delta f_{st} := f_t - f_s$, $(s,t) \in \Delta_2(0,T)$, and, if $A: \Delta_2(0,T) \to \RR^m$, we define $\delta A: \Delta_3(0,T) \to \RR^m$ by $\delta A_{sut} := A_{st} - A_{su} - A_{ut}$ for $(s,u,t) \in \Delta_3(0,T)$. Note that, for $f: [0,T] \to \RR^m$, $\delta(\delta f) \equiv 0)$, and, for $A: \Delta_2(0,T) \to \RR^m$, if $\delta A \equiv 0$, then $A = \delta f$ for some path $f$ (see Lemma \ref{L:additive} below for an ``almost-everywhere'' version of this statement).

With slight abuse of notation, a map $f: [0,T] \to \RR^m$ is also understood as a function on $\Delta_2(0,T)$ by writing $f_{st} = f_s$ for $(s,t) \in \Delta_2(0,T)$. Then, for instance, if $f,g: [0,T] \to \RR^m$, the notation $f \delta g$ means $(f\delta g)_{st} = f_s \delta g_{st}$.

\subsection*{Acknowledgments}
PKF has received funding from the European Research Council (ERC) under
the European Union's Horizon 2020 research and innovation program (grant agreement No. 683164)
and the DFG Research Unit FOR 2402. BS acknowledges funding from the National Science Foundation Mathematical Sciences Postdoctoral Research Fellowship under Grant Number DMS-1902658 as well as a Dirichlet Visiting Fellowship from the Berlin Mathematical School.
PZK was supported in part by DFG (German Research Foundation) -- EXC-2047/1 -- 390685813 (Hausdorff Center for Mathematics).

%

%

\section{Spaces of Besov type}

\subsection{Metric space valued-functions}

Given a metric space $(E,d)$, measurable $f: [0,T] \to E$, $p \in (0,\oo]$, and $\tau \in [0,T]$, we set, with the usual modification when $p= \infty$, 
\begin{equation}\label{Lpmodulus}
	\omega_p(f,\tau) := \sup_{0 \le h \le \tau} \left[ \int_0^{T-h}  d(f_s,f_{s+h})^p ds \right]^{1/p}.
\end{equation}

\begin{definition}\label{D:Besov}
	Let $\alpha \in (0,1)$ and $p,q \in (0,\oo]$. We say $f \in B^\alpha_{pq}([0,T];E)$ if, for some (and therefore every) $x_0 \in E$, $ d(f,x_0) \in L^p([0,T])$, and 
	\[
		[f]_{B^\alpha_{pq}([0,T])} := \left[ \int_0^T \left( \frac{\omega_p(f,\tau)}{\tau^\alpha}\right)^q \frac{d\tau}{\tau} \right]^{1/q} < \oo.
	\]
	Given a non-decreasing function $\omega: [0,\oo) \to [0,\oo)$ that satisfies $\lim_{r\to0^+} \omega(r) =0$, we say $f \in B^\omega_{pq}([0,T];E)$ if $d(f,x_0) \in L^p([0,T])$ and
	\[
		[f]_{B^\omega_{pq}([0,T])} := \left[ \int_0^T \left( \frac{\omega_p(f,\tau)}{\omega(\tau)}\right)^q \frac{d\tau}{\tau} \right]^{1/q} < \oo;
	\]
	in both cases, the usual modification is made for $q=\infty$.
\end{definition}

\begin{lemma}\label{L:gammalarger1}
	Assume $f: [0,T] \to E$ is measurable and, for some $0 < p \le \oo$,
	\[
		\lim_{\tau \to 0}  \frac{\omega_p(f,\tau)}{\tau^{1\vee 1/p}} =0.
	\]
	Then $f$ has a version which is constant on $[0,T]$.
\end{lemma}

\begin{remark}\label{R:gammalarger1}
	The function $f$ satisfies the assumptions of Lemma \ref{L:gammalarger1} if, for some $0 < q \le \oo$ and $\gamma > 1 \vee \frac{1}{p}$,
	\[
		\int_0^T \left( \frac{\omega_p(f,\tau)}{\tau^{\gamma}}\right)^q \frac{d\tau}{\tau} < \oo.
	\]
	Thus, Definition \ref{D:Besov} is vacuous once $\alpha > 1 \vee \frac{1}{p}$.
\end{remark}

\begin{proof}[Proof of Lemma \ref{L:gammalarger1}]
	For $\tau \in [0,T]$, define $\rho(\tau) := \omega_p(f,\tau) \tau^{- \left( 1 \vee \frac{1}{p} \right) }$. Then, for all $h \in [0,T]$, $\int_0^{T-h} d(f_t,f_{t+h})^p dt \lesssim_p \rho(h)^p h^{1 \vee p}$. Given $t \in [0,T]$ and $h \in [0,T-t]$, we may write
	\[
		d(f_t,f_{t+h}) \le \sum_{k=1}^n d\left( f_{t + (k-1)h/n}, f_{t+kh/n} \right).
	\]
	If $p \ge 1$, we then have
	\[
		\left[\int_0^{T-h} d(f_t,f_{t+h})^p dt \right]^{1/p} \le \sum_{k=1}^n \left[\int_0^{T-h} d\left( f_{t + (k-1)h/n}, f_{t+kh/n} \right)^p dt \right]^{1/p} \lesssim_p \rho\left( \frac{h}{n} \right) h \xrightarrow{n \to \oo} 0,
	\]
	while if $0 < p < 1$,
	\[
		\int_0^{T-h} d(f_t,f_{t+h})^p dt \le \sum_{k=1}^n\int_0^{T-h} d\left( f_{t + (k-1)h/n}, f_{t+kh/n} \right)^p dt \lesssim n \rho\left( \frac{h}{n} \right)^p h \xrightarrow{n \to \oo} 0.
	\]
	In either case, we conclude that, for all $h \in [0,T]$, $d(f_t,f_{t+h}) = 0$ for Lebesgue almost-every $t \in [0,T-h]$. By Fubini's theorem, this implies that there exists $s \in [0,T]$ such that $d(f_s,f_t) = 0$ for Lebesgue almost every $t \in [0,T]$, and we conclude.
\end{proof}


\begin{remark}\label{R:Bequivalance}
	The characterization of Besov spaces given in Definition \ref{D:Besov} is useful for our purposes, since it reflects that Besov regularity is a generalization of H\"older regularity, wherein continuity is measured with the $L^p$-modulus $\omega_p$. Strictly speaking, and similar to $L^p = \mathcal{L}^p/\sim$ in $L^p$-theory with $f \sim g$ iff $f=g$ a.e. on $[0,T]$, the above defines a space of measurable functions $\mathcal{B}^\alpha_{pq} \subset \mathcal{L}^p$, the same quotienting procedure then yields $B^\alpha_{pq} \subset L^p$. We will not make this distinction explicit in what follows and in any case frequently can work with continuous representatives (though there is no a priori assumption in this regard!). In particular, for $0 < \alpha < 1$ we have
$B^\alpha_{\oo,\oo} = C^{0,\alpha}$, the classical H\"older space of exponent $\alpha$.
	
	We often write $B^\alpha_{pq}([0,T]) = B^\alpha_{pq}([0,T];E)$ when it does not confusion. A common example for $E$ is $\RR^n$ with the usual Euclidean metric; other choices, such as certain Lie groups, also arise naturally in the context of geometric and higher order rough paths (see Section \ref{sec:roughpaths}). In the Euclidean case, Definition \ref{D:Besov} is equivalent to the standard one, in terms of Fourier analysis and Littlewood-Paley blocks, exactly when $0 < \alpha < 1$, $\frac{1}{1+\alpha} < p \le \oo$, and $0 < q \le \oo$ (see Triebels \cite{T}). However, in what follows, nothing is lost by considering the space from Definition \ref{D:Besov} even when $p \le \frac{1}{1+\alpha}$.
\end{remark}

Suppose $E$ is a Banach space. If $1 \le p,q \le \oo$, then $B^\alpha_{pq}([0,T];E)$ is itself a Banach space. Otherwise, $\nor{\cdot}{B^\alpha_{pq}}$ is only a quasi-norm, but $B^\alpha_{pq}([0,T])$ is still a complete metric space, with metrics defined as follows:

\begin{definition}\label{D:Besovspacemetric}
Assume $E$ is a Banach space, $\omega: [0,\oo) \to [0,\oo)$ is non-decreasing, $\lim_{r \to 0^+} \omega(r) = 0$, and $0 < p,q \le \oo$. Then $B^\omega_{pq}([0,T];E)$ is made into a complete metric space with the metric given by
\[
	\mathrm d_{B^\omega_{pq}([0,T];E)}(f,g) :=
	\begin{dcases}
		\nor{f-g}{L^p([0,T];E)}^p + \int_0^T \left( \frac{\omega_p(f-g,\tau)}{\omega(\tau)} \right)^q \frac{d\tau}{\tau} & \text{if } 0 < q \le p < 1, \\
		\nor{f-g}{L^p([0,T];E)} + \int_0^T \left( \frac{\omega_p(f-g,\tau)}{\omega(\tau)} \right)^q \frac{d\tau}{\tau} & \text{if } 0 < q < 1 \le p, \\
		\nor{f-g}{L^p([0,T];E)}^p + \left[ \int_0^T \left( \frac{\omega_p(f-g,\tau)}{\omega(\tau)} \right)^q \frac{d\tau}{\tau} \right]^{p/q} & \text{if } 0 < p < 1 \text{ and } q > p.
	\end{dcases}
\]
\end{definition}

We refer also to the work \cite{chong2020characterization} on nonlinear Besov spaces, in which a further, discrete characterization is given when $\alpha > 1/p$. This is precisely when every $f \in B^\alpha_{pq}([0,T];E)$ has a H\"older-continuous version (see Proposition \ref{P:Besovembedding} below). The definition given in \cite{chong2020characterization} differs from Definition \ref{D:Besov} in that the $L^p$-modulus of continuity $\omega_p(f,\tau)$ in \eqref{Lpmodulus} is replaced with $\int_0^{T-\tau} d(f_t,f_{t+\tau})^pdt$. As the next result shows, these two definitions are equivalent.

\begin{lemma}\label{L:discreteBesovnorm}
	Assume $(E,d)$ is a complete metric space, and let $0 < \alpha < 1$ and $0 < p,q \le \oo$. Then, for all $f \in B^\alpha_{pq}([0,T],E)$,
	\[
		[f]_{B^\alpha_{pq}([0,T],E)} \lesssim_{\alpha,p,q}
		\left[ \sum_{n=1}^\oo \left( (2^n/T)^\alpha \nor{ d(f_\cdot, f_{\cdot + 2^{-n}T}) }{L^p([0,T(1-2^{-n})])} \right)^q \right]^{1/q}
		\lesssim_{\alpha,p,q} C_2 [f]_{B^\alpha_{pq}([0,T],E)}.
	\]
\end{lemma}

To prove Lemma \ref{L:discreteBesovnorm}, we require the following standard lemma. It follows immediately from subadditivity if $0 < q < 1$, and the case $q \ge 1$ is \cite[Lemma 1.d.20]{Konig}.

\begin{lemma}\label{L:konigseq}
	Let $\lambda > 1$ and $0 < q \le \oo$. Then, for all nonnegative $a_0, a_1, a_2, \ldots$,
	\[
		\sum_{n=1}^\oo \left(\lambda^n \sum_{j=n-1}^\oo a_j\right)^q \lesssim_{\lambda,q} \sum_{n=0}^\oo (\lambda^n a_n)^q.
	\]
\end{lemma}

\begin{proof}[Proof of Lemma \ref{L:discreteBesovnorm}]
	The inequality is invariant under scaling in $T$, so it suffices to consider $T = 1$. Given $f \in B^\alpha_{pq}([0,1],E)$, the right-hand inequality is established with the chain of inequalities
	\begin{align*}
		\left[ \sum_{n=1}^\oo \left(2^{n\alpha} \nor{ d(f_\cdot, f_{\cdot + 2^{-n}}) }{L^p([0,1-2^{-n}])} \right)^q \right]^{1/q}
		&\le \left[ \sum_{n=1}^\oo \left(2^{n\alpha} \omega_p(f,2^{-n}) \right)^q \right]^{1/q}\\
		&\le \frac{1}{(\log 2)^{1/q}} \left[ \sum_{n=1}^\oo  \int_{2^{-n-1}}^{2^{-n}} \left(\frac {\omega_p(f,2\tau) }{\tau^\alpha} \right)^q \frac{d\tau}{\tau}\right]^{1/q}\\
		&= \frac{2^{\alpha}}{(\log 2)^{1/q}} \left[ \sum_{n=1}^\oo  \int_{2^{-n}}^{2^{-n+1}} \left(\frac {\omega_p(f,\tau) }{\tau^\alpha} \right)^q \frac{d\tau}{\tau}\right]^{1/q}\\
		&= \frac{2^{\alpha}}{(\log 2)^{1/q}} [f]_{B^\alpha_{pq}([0,1],E)}.
	\end{align*}
	
	To obtain the left-hand inequality, we first write
	\[
		 [f]_{B^\alpha_{pq}([0,T],E)}
		 \le 2^{\alpha}(\log 2)^{1/q} \left[ \sum_{n=1}^\oo \left( 2^{n\alpha} \omega_p(f,2^{-n}) \right)^q \right]^{1/q}.
	\]
	Set $m_0 := 0$, choose $h_1 \in [0,1/2]$ such that $\omega_p(f,1/2) = \nor{ d(f_\cdot, f_\cdot + h_1)}{L^p([0,1-h_1])}$, and let $m_1 = 2,3,\ldots$ be such that $2^{-m_1} < s \le 2^{-(m_1-1)}$. Then, for all $n = 1,2,\ldots, m_1 - 1$,
	\[
		\nor{d(f_\cdot,f_{\cdot + h_1})}{L^p([0,1-h_1])} = \omega_p(f,2^{-n}).
	\]
	Continuing inductively, we define sequences $1/2 \ge h_1 > h_2 > h_3 > \cdots \to 0$ and $2 \le m_1 < m_2 < \cdots \to \oo$ such that, for $k = 1,2,\ldots$, we have $2^{-m_k} < s_k \le 2^{-(m_k - 1)}$ and, for all $n = m_{k-1}, m_{k-1} + 2, \ldots, m_k - 1$,
	\[
		\nor{ d(f_\cdot,f_{\cdot + h_k})}{L^p([0,1-h_k])} = \omega_p(f,2^{-n}).
	\]
	We thus write
	\begin{align*}
		\sum_{n=1}^\oo \left( 2^{n\alpha} \omega_p(f,2^{-n})\right)^q 
		&=  \sum_{k=1}^\oo \sum_{n = m_{k-1}}^{m_k - 1} \left( 2^{n\alpha} \omega_p(f,2^{-n}) \right)^q\\
		&\le \frac{1}{2^{\alpha q} - 1} \sum_{k=1}^\oo \left( 2^{m_k \alpha} \nor{ d(f_\cdot, f_{\cdot + h_k} ) }{L^p([0,1-h_k])} \right)^q.
	\end{align*}
	We next take the dyadic expansion of $h_k$, expressed as
	\[
		h_k = \sum_{j = m_k - 1}^\oo \eps_j^k 2^j \quad \text{for } \eps_j^k \in \{0,1\},
	\]
	so that the triangle inequality gives
	\[
		 \nor{ d(f_\cdot, f_{\cdot + h_k} ) }{L^p([0,1-h_k])}
		 \le 
		 \begin{dcases}
		 \sum_{j= m_k - 1}^\oo \nor{ d(f_\cdot, f_{\cdot + 2^{-j}} ) }{L^p([0,1-2^{-j}])} &\text{for } 1 \le p \le \oo \text{ and}\\
		 \left(\sum_{j= m_k - 1}^\oo \nor{ d(f_\cdot, f_{\cdot + 2^{-j}} ) }{L^p([0,1-2^{-j}])}^p\right)^{1/p} & \text{for } 0 < p < 1.
		 \end{dcases}
	\]
	The proof is then finished with the use of Lemma \ref{L:konigseq}: if $p \ge 1$, then
	\begin{align*}
		 \sum_{k=1}^\oo \left( 2^{m_k \alpha} \nor{ d(f_\cdot, f_{\cdot + h_k} ) }{L^p([0,1-h_k])} \right)^q
		 &\le \sum_{k=1}^\oo \left( 2^{m_k \alpha}  \sum_{j= m_k - 1}^\oo \nor{ d(f_\cdot, f_{\cdot + 2^{-j}} ) }{L^p([0,1-2^{-j}])} \right)^q\\
		 &\le \sum_{n=1}^\oo \left( 2^{n\alpha} \sum_{j = n-1}^\oo \nor{ d(f_\cdot, f_{\cdot + 2^{-j}} ) }{L^p([0,1-2^{-j}])} \right)^q\\
		 &\lesssim_{\alpha,q } \sum_{n=1}^\oo \left( 2^{n\alpha} \nor{ d(f_\cdot, f_{\cdot + 2^{-n}} ) }{L^p([0,1-2^{-n}])} \right)^q,
	\end{align*}
	and, if $0 < p < 1$, 
	\begin{align*}
		 \sum_{k=1}^\oo \left( 2^{m_k \alpha} \nor{  d(f_\cdot, f_{\cdot + h_k} ) }{L^p([0,1-h_k])} \right)^q
		 &\le \sum_{k=1}^\oo \left( 2^{m_k \alpha p}  \sum_{j= m_k - 1}^\oo \nor{ d(f_\cdot, f_{\cdot + 2^{-j}} ) }{L^p([0,1-2^{-j}])}^p \right)^{q/p}\\
		 &\le \sum_{n=1}^\oo \left( 2^{n\alpha p} \sum_{j = n-1}^\oo \nor{ d(f_\cdot, f_{\cdot + 2^{-j}} ) }{L^p([0,1-2^{-j}])}^p \right)^{q/p}\\
		 &\lesssim_{\alpha,p,q}  \sum_{n=1}^\oo \left( 2^{n\alpha} \nor{ d(f_\cdot, f_{\cdot + 2^{-n}} ) }{L^p([0,1-2^{-n}])} \right)^q.
	\end{align*}
\end{proof}

We often work in the regime where $\alpha > 1/p$ (the extreme case being $p = \oo$), in which case every $f \in B^\alpha_{pq}([0,T];E)$ admits a H\"older-continuous version. 

\begin{proposition}\label{P:Besovembedding}
	Assume $(E,d)$ is a complete metric space, $0 < \alpha < 1$, $1/\alpha < p \le \oo$, and $0 < q \le \oo$. Then, for all $Y \in B^\alpha_{pq}([0,T],E)$,
	\[
		[Y]_{C^{\alpha-1/p}([0,T])} \lesssim_{\alpha,p,q} [Y]_{B^\alpha_{pq}([0,T])}.
	\]
\end{proposition}
This result is classical when $E = \RR^m$; see for instance \cite{Simon,T}. The proof we give, which is based on the Campanato characterization of H\"older continuity, is delegated to subsection \ref{subsec:Besov2parameter}, where it is seen to follow from an analogous result for two-parameter spaces (see Proposition \ref{P:BBesovembedding}).

\begin{remark}\label{R:Besovsup}
	If $E$ is a Banach space, then, as a consequence of Proposition \ref{P:Besovembedding}, whenever $\alpha p > 1$, the (quasi-)norm $\nor{\cdot}{B^\alpha_{pq}}$ is equivalent to both
	\[
		Y \mapsto \nor{Y}{L^\oo([0,T])} + [Y]_{B^\alpha_{pq}([0,T])} \quad \text{and} \quad Y \mapsto |Y(0)| + [Y]_{B^\alpha_{pq}([0,T])},
	\]
	with proportionally constants depending on $T$ in addition to $\alpha$, $p$, and $q$. In particular, for fixed $y \in E$, the affine subspace
	\[
		B^\alpha_{pq}([0,T];E, y) := \left\{ Y \in B^\alpha_{pq}([0,T];E) : Y(0) = y \right\}
	\]
	is a complete metric space with the metric 
	\[
		(X,Y) \mapsto 
		\begin{dcases}
			[X - Y]_{B^\alpha_{pq}} & \text{if } q \ge 1 \text{ and}\\
			[X - Y]_{B^\alpha_{pq}}^q & \text{if } 0 < q < 1.
		\end{dcases}
	\]
	If $y = 0$ and $q \ge 1$, then $B^\alpha_{pq}([0,T];E,0)$ is a Banach space with the norm $[\cdot]_{B^\alpha_{pq}}$.
\end{remark}

When solving fixed-point problems for differential equations driven by Besov signals, we need the following result on composing regular functions with Besov paths.

\begin{lemma}\label{L:Besovcomposition}
	Assume $\delta \in (0,1]$, $\alpha \in (0,1)$, $p,q \in (0,\oo]$, $f \in C^\delta(\RR^m)$, and $Y \in B^\alpha_{pq}([0,T],\RR^m)$. Then
	\begin{equation}\label{fofBesov}
		[f(Y)]_{B^{\delta \alpha}_{p,q/\delta}} \le [f]_{C^\delta}T^{\frac{1-\delta}{p}} [Y]_{B^\alpha_{pq}}^\delta.
	\end{equation}
	If $\alpha > 1/p$, $f \in C^{1,\delta}$, and $\tilde Y \in B^\alpha_{pq}([0,T])$, then
	\begin{equation}\label{fofBesovdifference}
		\begin{split}
		[f(Y) - f(\tilde Y)]_{B^{\delta \alpha}_{p,q/\delta}} &\lesssim_{\alpha,p,q,T}  [Df]_{C^\delta} \left( [Y]_{B^\alpha_{pq}}^\delta + [\tilde Y]_{B^\alpha_{pq}}^\delta \right) |Y_0 - \tilde Y_0| \\
		&+ \nor{Df}{C^{1,\delta}}\left( 1 + [Y]_{B^\alpha_{pq}}^\delta + [\tilde Y]_{B^\alpha_{pq}}^\delta \right) [Y - \tilde Y]_{B^\alpha_{pq}}.
		\end{split}
	\end{equation}
\end{lemma}

\begin{proof}
	The bound \eqref{fofBesov} is a consequence of H\"older's inequality. To prove \eqref{fofBesovdifference}, we first write, for $t \in [0,T]$,
	\[
		f(Y_t) - f(\tilde Y_t) = \int_0^1 Df\left( \tau Y_t + (1-\tau) \tilde Y_t\right) d\tau \left( Y_t - \tilde Y_t\right).
	\]
	Then, for $t \in [0,T-h]$,
	\begin{align*}
		\abs{ \left( f(Y) - f(\tilde Y)\right)_{t+h} - \left( f(Y) - f(\tilde Y)\right)_t }
		&\le \nor{Df}{\oo} \abs{ (Y-\tilde Y)_{t+h} - (Y - \tilde Y)_t} \\
		&+ \nor{Y - \tilde Y}{\oo} [Df]_{C^\delta} \left( |Y_{t+h} - Y_t|^\delta + |\tilde Y_{t+h} - \tilde Y_t|^\delta\right),
	\end{align*}
	and so
	\begin{align*}
		&\omega_p(f(Y) - f(\tilde Y),\tau)
		\le \nor{Df}{\oo} \omega_p(Y - \tilde Y,\tau) + \nor{Y - \tilde Y}{\oo} [Df]_{C^\delta} \left( \omega_{\delta p}(Y,\tau) + \omega_{\delta p}(\tilde Y,\tau)\right)\\
		&\lesssim_{\alpha,p,q} \nor{Df}{\oo} \omega_p(Y - \tilde Y,\tau) + \left( |Y_0 - \tilde Y_0| + [Y - \tilde Y]_{B^\alpha_{pq}}T^{\alpha - 1/p} \right) [Df]_{C^\delta} \left( \omega_{\delta p}(Y,\tau) + \omega_{\delta p}(\tilde Y,\tau)\right)\\
		&\le \nor{Df}{\oo} \omega_p(Y - \tilde Y,\tau) + \left( |Y_0 - \tilde Y_0| + [Y - \tilde Y]_{B^\alpha_{pq}}T^{\alpha - 1/p} \right) [Df]_{C^\delta} \left( \omega_{p}(Y,\tau)^\delta + \omega_{p}(\tilde Y,\tau)^\delta\right)T^{\frac{1-\delta}{p}}.
	\end{align*}
	We conclude that
	\begin{align*}
		&[f(Y) - f(\tilde Y)]_{B^{\delta\alpha}_{p,q/\delta}}
		\lesssim_{\alpha,p,q} \nor{Df}{\oo} [Y - \tilde Y]_{B^\alpha_{pq}}T^{(1-\delta)\alpha} \\
		&+ \left( |Y_0 - \tilde Y_0| + [Y - \tilde Y]_{B^\alpha_{pq}} T^{\alpha - 1/p}\right) [Df]_{C^\delta} \left( [Y]_{B^\alpha_{pq}}^\delta + [\tilde Y]_{B^\alpha_{pq}}^\delta \right) T^{\frac{1-\delta}{p}}\\
		&\lesssim_T [Df]_{C^\delta} \left( [Y]_{B^\alpha_{pq}}^\delta + [\tilde Y]_{B^\alpha_{pq}}^\delta \right)|Y_0 - \tilde Y_0| + \nor{Df}{C^{1,\delta}}\left( 1 + [Y]_{B^\alpha_{pq}}^\delta + [\tilde Y]_{B^\alpha_{pq}}^\delta \right)[Y - \tilde Y]_{B^\alpha_{pq}}.
	\end{align*}
\end{proof}

The regime $\alpha > 1/p$ also allows for embeddings into variation spaces.  For $1 \le p < \oo$, define
\[
	\mathcal{V}^p([0,T],E) := \left\{ f: [0,T] \to E : [f]_{\mathcal{V}^p} := \sup_{P} \left( \sum_{i=1}^N d(f_{t_{i-1}}, f_{t_i})^p\right)^{1/p} < \oo \right\},
\]
where the supremum is taken over partitions $P := \{ 0 = t_0 < t_1 < \cdots < t_N = T\}$ of $[0,T]$. 
For any $f \in \mathcal{V}^p([0,T])$, write $\mbf f \in V^p = \mathcal{V}^p / \sim$ for the equivalence class of functions equal to $f$ up to a set of Lebesgue measure zero. (This is similar to the cassical construction of the Lebesgue spaces $L^p = \mathcal{L}^p / \sim$.) Note that ${V}^p \subset L^\infty$. Similar to e.g. \cite{bourdaud2006superposition} we set 
\begin{equation} \label{equ:Vpseminorm}
     [\mbf f]_{{V}^p([0,T];E)} :=  \inf_{f \in \mbf f} [f]_{\mathcal{V}^p([0,T];E)} .
\end{equation}
Note that ${V}^p([0,T];E)$ is Banach if $E$ is, seminormed with \eqref{equ:Vpseminorm}, or normed under $ [\mbf f]_{{L}^\infty} + [\mbf f]_{{V}^p}$.

In \cite{friz2006embedding, chong2020characterization} it is shown that $B^\alpha_{pq} \subset \mcl V^r$, where $r = 1/\alpha$ if $q \le p$ and $r = 1/\alpha + \eps$ for some $\eps > 0$ if $q > p$. For the purposes of later discussion, we verify the sharpness of the embedding when $q \le p$.

\begin{proposition}\label{P:strictvariation}
	Assume $1 \le q \le p \le \oo$ and $\alpha > 1/p$. Then there exists $f \in B^\alpha_{pq}([0,1])$ such that $f \notin \mcl V^r([0,1])$ for any $1 \le r < 1/\alpha$.
\end{proposition}

\begin{proof}
	Define $\chi: \RR \to \RR$ to be $1$-periodic such that
	\[
		\chi_t := 
		\begin{dcases}
			t & \text{if } t \in [0,1/2] \text{ and}\\
			1-t & \text{if } t \in [1/2,1],
		\end{dcases}
	\]
	and, for $n \in \NN$ and $t \in [0,1]$, define $f^n_t = 2^{-\alpha n} \chi_{2^n t}$. Then, for $m \in \NN$, $\omega_p(f^n,2^{-m}) \le 2^{-\alpha n} \left( 1 \wedge 2^{-m} \right)$, so that
	\[
		\sum_{m=1}^\oo \left( 2^{-m\alpha} \omega_p(f^n,2^{-m}) \right)^q
		\le \sum_{m=1}^{n-1} 2^{(m-n)\alpha q} + \sum_{m=n}^\oo 2^{-(1-\alpha)(m-n)q}
		\lesssim_{\alpha,q} 1,
	\]
	which yields
	\[
		\sup_{n \in \NN} [f^n]_{B^\alpha_{pq}[0,1]} \lesssim_{\alpha,p,q} 1.
	\]
	Moreover, extending $f^n$ to the rest of $\RR$ to be $0$ gives also
	\[
		\sup_{n \in \NN} [f^n]_{B^\alpha_{pq}([a,b])} \lesssim_{\alpha,p,q} 1
	\]
	for any interval $[0,1] \subset [a,b] \subset \RR$. Meanwhile,
	\[
		[f^n]_{\mcl V^r([0,1])} \ge 2^{n\left( \frac{1}{r} - \alpha\right)}.
	\]
	
	For some sequences $(a_k)_{k \in \NN} \subset [0,\oo)$ and $(n_k)_{k \in \NN} \subset \NN$ and $t \in [0,1]$, we set
	\[
		f_t = \sum_{k=1}^\oo a_k f^{n_k}_{2^k(t - 2^{-k})}.
	\]
	Then
	\[
		[f]_{B^\alpha_{pq}([0,1])} \le \sum_{k=1}^\oo a_k \left[f^{n_k}_{2^k(\cdot - 2^{-k})} \right]_{B^\alpha_{pq}([0,1])} = \sum_{k=1}^\oo a_k 2^{k(\alpha - 1/p)} \left[ f^{n_k} \right]_{B^\alpha_{pq}([-1, 2^k - 1])} \lesssim_{\alpha,p,q} \sum_{k=1}^\oo a_k 2^{k(\alpha - 1/p)},
	\]
	while the superadditivity of $[\cdot]_{\mcl V^r}^r$ yields
	\[
		[f]_{\mcl V^r([0,1])}^r \ge \sum_{k=1}^\oo a_k^r \left[f^{n_k}_{2^k(\cdot - 2^{-k})} \right]^r_{\mcl V^r([2^{-k}, 2^{1-k} ])} \ge \sum_{k=1}^\oo a_k^r 2^{n_k(1-\alpha r)}.
	\]
	We conclude upon choosing $a_k$ and $n_k$ so that the first series converges and the second diverges for any $r < 1/\alpha$ (for instance, we may take $a_k = 4^{-k(\alpha - 1/p)}$ and $n_k = k^2$ for $k \in \NN$).
\end{proof}

\subsection{On scale invariant Besov spaces} \label{sec:criticalBesov} 

We make several remarks in the case that $\alpha = 1/p$. For any $q \in (0,\oo]$, $B^{1/p}_{p,q}$ does not embed continuously into a space of H\"older continuous functions, which is related to the fact that the $B^{1/p}_{p,q}$ norm is invariant under time reparametrization. For $1 \le p \le \oo$ and $q = 1$, and for $E = \RR^m$, we have the standard embeddings (see \cite{T}) $B^{1/p}_{p,1} \subset B^0_{\oo,1} \subset C$. In fact, more is true, and we can relate the reparametrization-invariant Besov spaces $B^{1/p}_{p,q}$ to the spaces $V^p$ introduced in the previous subsection as follows:

\begin{proposition}\label{P:variationembedding}
	Assume $(E,d)$ is a complete metric space and $1 < p < \oo$. Then we have the continuous embeddings
	\[
		B^{1/p}_{p,1}([0,T],E) \subset c{V}^p([0,T],E) \subset {V}^p([0,T],E) \subset B^{1/p}_{p,\oo}([0,T],E ),
	\]
	where we write $c{V}^p$ for elements in ${V}^p$ with continuous representative, in $C \cap \mathcal{V}^p$.
\end{proposition}

The proof of the first embedding in \eqref{P:variationembedding} uses real interpolation methods, much in the same way as in \cite{BerghPeetreVp}, where analogous inclusions are established 
for homogenous Besov spaces on the whole real line. We note that the right-most inclusion is a slight strengthening of \cite[Proposition 4.3]{chong2020characterization}, which states in our notation that $c{V}^p([0,T]) \subset B^{1/p}_{p,\oo}([0,T])$ (the proof follows essentially the same argument). 
Also note that this result completes the picture of embedding Besov spaces into variation ones, studied, for instance, in \cite{friz2006embedding,chong2020characterization}, where embeddings of the form $B^{1/p}_{r,q} \subset {V}^p$ are proved for $r > p$.

We first note the following useful equivalent characterization of the $\mcl V^p$ seminorm:

\begin{lemma}\label{L:variationnorm} Let $ 1 \le p < \infty$ and assume $E$ is a Banach space.
 Then $f \in \mathcal{V}^p ([0, T],  E)$ if and only if
  \[ [f]_{\bar{\mathcal{V}}^p ([0, T])} := \sup_P \left( \sum_{[s,t] \in P} \inf_c \| f -
     c \|^p_{\infty ; [s, t]} \right)^{1 / p} < \infty \]
  with $\| f \|_{\infty ; I} = \sup_{t \in I} | f_t |$. Moreover,
  \[ \frac{1}{2} [f]_{\mathcal{V}^p ([0, T])} \le
     [f]_{\bar{\mathcal{V}}^p ([0, T])} \le [f]_{\mathcal{V}^p ([0,
     T])}^{} . \]
\end{lemma}

\begin{proof}
  Since $| f_t - f_s | \le \inf_c (| f_t - c | + | f_s - c |) \le
  2 \inf_c \| f - c \|_{\infty ; [s, t]}$ it is clear that $\frac{1}{2} [f]_{\mathcal{V}^p ([0, T])} \le
     [f]_{\bar{\mathcal{V}}^p ([0, T])}$.
  Conversely,
  \[ \inf_c \| f - c \|^p_{\infty ; [s, t]} \le \| f - f_s \|^p_{\infty
     ; [s, t]} \le [f]_{\mathcal{V}^p ([s, t])}^p \]
  and by super-additivity of the right-hand side we see that $\frac{1}{2} [f]_{\mathcal{V}^p ([0, T])} \le
     [f]_{\bar{\mathcal{V}}^p ([0, T])} \le [f]_{\mathcal{V}^p ([0,
     T])}^{} .$
\end{proof}

Lemma \ref{L:variationnorm} suggests a way to define $\mcl V^p$ for $p = \oo$: we say $f \in \mcl V^\oo$ if
$$
 [f]_{{\mathcal{V}}^\infty ([0, T])} := \sup_P \left( \max_{[s,t] \in P}  
  \inf_c \| f -    c \|_{\infty ; [s, t]} \right) < \oo.
$$
Of course, the supremum is attained for the trivial partition $P = \{0,T\}$, and so
\[
	[f]_{\mcl V^\oo([0,T])} = \frac{1}{2} \sup_{(s,t) \in \Delta_2(0,T)} |f_t - f_s|,
\]
that is, $\mathcal{V}^\oo([0,T],E) = \mcl L^\oo([0,T],E)$, modulo constants.

We next relate the $p$-variation spaces with different powers using real interpolation. We recall (see for instance \cite{berghlofstrom}) that, for two compatible normed spaces $X_0$ and $X_1$ (that is, both $X_0$ and $X_1$ belong to a common Hausdorff topological space), we define the $K$-functional, for $t > 0$ and $f \in X_0 + X_1$, by
\begin{equation}\label{Kfunctional}
	K(t,f,X_0,X_1) = \inf \left\{ \nor{f_0}{X_0} + t \nor{f_1}{X_1} : f = f_0 + f_1, \; f_i \in X_i \right\},
\end{equation}
and, for $\theta \in (0,1)$ and $p \in [1,\oo]$, we define the real interpolation space
\begin{equation}\label{interpolationspace}
	(X_0,X_1)_{\theta,p} := \left\{ f \in X_0 + X_1 : \nor{f}{(X_0,X_1)_{\theta,p}} := \left[ \int_0^\oo \left( \frac{K(t,f,X_0,X_1)}{t^\theta} \right)^p \frac{dt}{t} \right]^{1/p} < \oo \right\}.
\end{equation}

The following result is a variant of one appearing in \cite{BerghPeetreVp}, as Lemma 2.1. 

\begin{lemma}\label{L:variationinterpolation}
	Let $1 \le p_0 < p_1 \le \oo$, fix $\theta \in (0,1)$, and assume $E$ is a Banach space. Then we have the continuous embedding
	\[
		(\mcl V^{p_0}([0,T],E), \mcl V^{p_1}([0,T],E))_{\theta,p} \subset \mcl V^p([0,T],E), \qquad \text{where }
		\frac{1}{p} = \frac{1-\theta}{p_0} + \frac{\theta}{p_1}.
	\]
\end{lemma}

\begin{proof}
	Fix $f \in (\mcl V^{p_0}([0,T],E), \mcl V^{p_1}([0,T],E))_{\theta,p}$, a partition $P = \{0 = t_0 < t_1 < \cdots < t_N = T\}$ of $[0,T]$, and $f_0 \in \mcl V^{p_0}([0,T],E)$ and $f_1 \in \mcl V^{p_1}([0,T],E)$ such that $f = f_0 + f_1$. Let $\phi,\phi_0,\phi_1: [0,\oo) \to [0,\oo)$ be the non-increasing, right-continuous rearrangements of respectively the sequences
	\begin{align*}
		&\left( \inf_{c \in E} \sup_{t \in [t_{i-1}, t_i]} \nor{f(t) - c}{E} \right)_{i=1}^N, \quad
		\left( \inf_{c \in E} \sup_{t \in [t_{i-1}, t_i]} \nor{f_0(t) - c}{E} \right)_{i=1}^N, \\
		&\text{and} \quad
		\left( \inf_{c \in E} \sup_{t \in [t_{i-1}, t_i]} \nor{f_1(t) - c}{E} \right)_{i=1}^N;
	\end{align*}
	that is, for $i = 0,1,\emptyset$ and $\lambda \ge 0$, $\phi_i$ satisfies
	\[
		| \left\{ t: \phi_i(t) > \lambda \right\} | = \# \left\{ n \in \{1,2,\ldots, N\} : \inf_{c \in E} \sup_{t \in [t_{i-1}, t_i]} \nor{f_i(t) - c}{E}  > \lambda \right\}.
	\]
	We have 
	\[
		\inf_{c \in E}  \sup_{t \in [t_{i-1}, t_i]} \nor{f(t) - c}{E} \le \inf_{c \in E}  \sup_{t \in [t_{i-1}, t_i]} \nor{f_0(t) - c}{E} + \inf_{c \in E}  \sup_{t \in [t_{i-1}, t_i]} \nor{f_1(t) - c}{E},
	\]
	and elementary computations then give, for all $s > 0$, $\phi(s) \le \phi_0(s/2) + \phi_1(s/2)$. Then, by H\"older's inequality, for $\tau > 0$,
	\begin{align*}
		\tau^{1/p_0} \phi(\tau) &\le \left( \int_0^T \phi(s)^{p_0} ds \right)^{1/p_0}\\
		&\le \left( \int_0^\tau \phi_0(s/2)^{p_0} ds \right)^{1/p_0} + \left( \int_0^\tau \phi_1(s/2)^{p_0}ds \right)^{1/p_0}
		\lesssim \nor{\phi_0}{L^{p_0}([0,\oo))} + \tau^{\frac{1}{p_0} - \frac{1}{p_1}} \nor{\phi_1}{L^{p_1}([0,\oo))} \\
		&= \left( \sum_{i=1}^N \inf_{c \in E}  \sup_{t \in [t_{i-1}, t_i]} \nor{f_0(t) - c}{E}^{p_0} \right)^{1/p_0} +  \tau^{\frac{1}{p_0} - \frac{1}{p_1}}\left( \sum_{i=1}^N \inf_{c \in E}  \sup_{t \in [t_{i-1}, t_i]} \nor{f_1(t) - c}{E}^{p_1}\right)^{1/p_1}\\
		&\le [f_0]_{{\mcl V}^{p_0}} + \tau^{\frac{1}{p_0} - \frac{1}{p_1}} [f_1]_{{\mcl V}^{p_1}}.
	\end{align*}
	With the relation $t = \tau(t)^{\frac{1}{p_0} - \frac{1}{p_1}}$, upon taking the infimum over all such $f_0$ and $f_1$, we find from Lemma \ref{L:variationnorm} that
	\[
		\tau(t)^{1/p_0} \phi(\tau(t)) \lesssim K(t,f, \mcl V^{p_0}, \mcl V^{p_1}).
	\]
	Invoking the definition \eqref{interpolationspace}, as well as the relationship between $p$, $\theta$, $p_0$, and $p_1$, we conclude that
	\begin{align*}
		\nor{f}{(V^{p_0}, V^{p_1})} &\gtrsim \left[ \int_0^\oo \tau(t)^{p/p_0} \phi(\tau(t))^p t^{-\theta p} \frac{dt}{t} \right]^{1/p} \asymp \left[ \int_0^\oo \phi(\tau)^p d\tau \right]^{1/p}\\
		& = \left( \sum_{i=1}^N \inf_{c \in E}  \sup_{t \in [t_{i-1}, t_i]} \nor{f(t) - c}{E}^p \right)^{1/p}.
	\end{align*}
	We conclude upon taking the supremum over all partitions $P$.
\end{proof}

\begin{proof}[Proof of Proposition \ref{P:variationembedding}]
	Only the first embedding is proved here. As we have noted, the second inclusion follows exactly as in \cite[Proposition 4.3]{chong2020characterization}.

	Recall the {\em Kuratowski embedding}: every metric space $E$ can be isometrically embedded in a Banach space. Thus, we assume without loss of generality that $E$ is a (possibly non-separable) Banach space.

	We next note that we have
	\begin{equation}\label{Besovinterp}
		B^{1/p}_{p,1}([0,T],E) = \left( B^1_{1,1}([0,T],E), B^0_{\oo,1}([0,T],E) \right)_{1 - \frac{1}{p},p}.
	\end{equation}
	To see this, note first that
	\[
		B^{1/p}_{p,1}(\RR,E) = \left( B^1_{1,1}(\RR,E), B^0_{\oo,1}(\RR,E) \right)_{1 - \frac{1}{p},p},
	\]
	which is a generalization of \cite[Theorem 6.4.5 (3)]{berghlofstrom} to $E$-valued function spaces, and follows from the method of retracts \cite[Theorem 6.4.2]{berghlofstrom} and general results on interpolation of weighted vector-valued Lebesgue spaces \cite[Theorem 5.6.2]{berghlofstrom}.
	
	We then have common continuous extension and restriction maps (see \cite{T})
	\[
		\left\{ B^{1/p}_{p,1}([0,T]), B^1_{1,1}([0,T]), B^0_{\oo,1}([0,T]) \right\} \leftrightharpoons \left\{ B^{1/p}_{p,1}(\RR), B^1_{1,1}(\RR), B^0_{\oo,1}(\RR) \right\},
	\]
	and so \eqref{Besovinterp} is again a consequence of the method of retracts.
	
	From the standard embeddings $B^1_{1,1}(\RR) \subset W^{1,1}(\RR)$ and $B^0_{\oo,1}(\RR) \subset C(\RR)$, upon restricting to the interval $[0,T]$, we have the embeddings
	\[
		B^1_{1,1}([0,T],E) \subset W^{1,1}([0,T],E) \subset  c V^1([0,T],E)\quad \text{and} \quad B^0_{\oo,1}([0,T],E) \subset C([0,T],E) = c V^\oo([0,T],E).
	\]
	We conclude from Lemma \ref{L:variationinterpolation} that
	\begin{align*}
		B^{1/p}_{p,1}([0,T],E) &= \left( B^1_{1,1}([0,T],E), B^0_{0,\oo}([0,T],E) \right)_{1 - \frac{1}{p},p}\\
		&\subset C([0,T],E) \cap \left( \mcl V^1([0,T],E), \mcl V^\oo([0,T],E) \right)_{1-\frac{1}{p}, p} \subset c \mcl V^p([0,T],E).
	\end{align*}
	The fact that the proportionality constant is independent of $T$ follows from the scale-invariance of the inequality.
\end{proof}


In what follows $1 \le p < \infty$ and $(E,d)$ a general metric space, unless further specified. Recall $V^p ([0, T], E) = \mathcal{V}^p ([0, T],  E) / \sim$ defined in terms of $[\cdot ]_{{V}^p([0,T])}$ given in \eqref{equ:Vpseminorm}. We are grateful to Pavel Zorin-Kranich for removing an unnecessary use of Helly's selection principle in an earlier version of this paper.
\begin{proposition}
(i) The inclusion 
\[ V^p ([0, T], E) \subset  \hat{V}^p ([0, T],  E) := 
  \left\{ \mathbf{f} \in L^{\infty}
     ([0, T],  E) : [\mathbf{f}]_{\hat{V}^p ([0, T])} \equiv \sup \left(
     \sum \inf_c \| d(\mathbf{f}_\cdot, c) \|^p_{L^{\infty} ; [s, t]} \right)^{1 /
     p} < \infty \right\} 
     \]
holds and is continuous in the sense that $[\mathbf{f}]_{\hat{V}^p ([0, T])} \le [\mathbf{f}]_{{V}^p ([0,T])}$ for all $\mathbf{f} \in V^p$.

(ii)
Suppose that $E$ is a separable complete metric space.
Then, every equivalence class $\mathbf{f} \in \hat{V}^p$ has a representative $\tilde{\mathbf{f}}$ of finite $p$-variation with $\norm{\tilde{\mathbf{f}}}_{\mathcal{V}^p} \lesssim \norm{\mathbf{f}}_{\hat{V}^{p}}$.
\end{proposition}

\begin{proof} (i) Consider $\mathbf{f} \in V^p$ with representative $f \in \mathcal{V}^p$.
Then, using Lemma \ref{L:variationnorm},
\[ [\mathbf{f}]_{\hat{V}^p ([0, T])} =
\sup \left( \sum \inf_c  \underset{u \in [s, t]}{\mathrm{ess\, sup}}_{} | f (u) - c | \right)^{1 / p} \le
[f]_{\bar{\mathcal{V}}^p ([0, T])} \le [f]_{\mathcal{V}^p ([0,T])}^{}  \]
(ii) Let $f \in \hat{V}^{p}([0,T],E)$.
We show that $f$ has an essential Cauchy property at every point, that is, for every $t \in (0,T]$, we have
\begin{equation}
\label{eq:essential-left-Cauchy}
(\forall \epsilon>0) (\exists \delta>0) \esssup_{s,s' \in (t-\delta,t)} d(f_{s},f_{s'}) < \epsilon.
\end{equation}
Indeed, suppose that \eqref{eq:essential-left-Cauchy} fails for some $t \in (0,T]$ and $\epsilon > 0$. Then, we can construct a sequence of intervals $(a_{j},b_{j})$ and points $e_{j}\in E$ with the following properties:
\[
b_{j} < a_{j+1} < t,
\quad
d(e_{j},e_{j+1}) > 3\epsilon/4,
\quad
| \{ s\in (a_{j},b_{j}) | d(s,e_{j}) < \epsilon/4 \} | > 0.
\]
It is then routine to derive a contradiction to the finiteness of the $\hat{V}^{p}$ norm.
In order to construct such a sequence, let $\tilde{E} \subseteq E$ be a countable dense subset.
We start with $a_{0}=0$, $b_{0}=t/2$, say, and any $e_{0} \in \tilde{E}$ such that the required positive measure property holds.
Given $b_{j},e_{j}$, from the failure of \eqref{eq:essential-left-Cauchy}, we conclude that
\[
\esssup_{s \in (b_{j},t)} d(f_{s},e_{j}) \geq \epsilon.
\]
Hence, we can choose $e_{j+1} \in \tilde{E}$ such that
\begin{equation}
\label{eq:1}
| \{ s \in (b_{j},t) \;  | \; d(f_{s},e_{j}) \geq \epsilon, \; d(f_{s},e_{j+1}) < \epsilon/4 \} | > 0.
\end{equation}
The non-emptyness of the set in \eqref{eq:1} guarantees that $d(e_{j},e_{j+1}) > 3\epsilon/4$.
It remains to choose an interval $[a_{j+1},b_{j+1}] \subset (b_{j},t)$ whose intersection with the set in \eqref{eq:1} has positive measure.

The essential Cauchy condition \eqref{eq:essential-left-Cauchy} implies the existence of an essential left limit $f_{t-} \in E$, which is the unique point such that
\begin{equation}
\label{eq:essential-left-limit}
(\forall \epsilon>0) (\exists \delta>0) \esssup_{s \in (t-\delta,t)} d(f_{s},f_{t-}) < \epsilon.
\end{equation}
Analogously, one can construct the essential right limits $f_{t+}$.
Applying the Lebesgue differentiation theorem to the functions $t \mapsto d(e,f_{t})$ for each $e\in \tilde{E}$ (local integrability of these functions follows e.g.\ from the existence of left and right essential limits), we obtain a full measure subset $X \subseteq [0,T]$ such that, for every $t \in X$ and $e \in \tilde{E}$, we have
\[
\lim_{\delta\to 0} \delta^{-1} \int_{t-\delta}^{t} |d(e,f_{t}) - d(e,f_{s})| ds = 0.
\]
For any $e\in\tilde{E}$ and $t\in X$, by the definition of $f_{t-}$ and the above identity, it follows that
\begin{align*}
d(f_{t},f_{t-})
&\leq
d(f_{t},e) + \limsup_{\delta\to 0} \delta^{-1} \int_{s=t-\delta}^{t} d(e,f_{s}) + d(f_{s},f_{t-}) ds
\\ &=
d(f_{t},e) + \limsup_{\delta\to 0} \delta^{-1} \int_{s=t-\delta}^{t} d(e,f_{s}) ds
\\ &=
d(f_{t},e) + \limsup_{\delta\to 0} \delta^{-1} \int_{s=t-\delta}^{t} d(e,f_{t}) ds
\\ &=
2 d(f_{t},e).
\end{align*}
Since $e \in \tilde{E}$ was arbitrary, it follows that $f_{t}=f_{t-}$.
Let
\[
\tilde{f}_{t} :=
\begin{cases}
f_{0+}, & t=0,\\
f_{t-}, & t>0.
\end{cases}
\]
This function coincides with $f$ almost everywhere, and we have $\tilde{f}_{0+}=\tilde{f}_{0}$ and $\tilde{f}_{t-}=\tilde{f}_{t}$ for all $t>0$.
It is easy to see that $\norm{\tilde{f}}_{\mathcal{V}^p} \lesssim \norm{f}_{\hat{V}^{p}}$.
\end{proof}

We finish this sub-section with a discussion of the finer regularity properties of paths in scale-invariant Besov spaces. A representative example is the Heaviside function
\[
	H_t =
	\begin{cases}
		0 & \text{if } 0 \le t < 1/2,\\
		1 & \text{if } 1/2 \le t \le 1,
	\end{cases}
\]
which satisfies $[H]_{B^{1/p}_{p,\oo}([0,1])} = 1$ for all $0 < p < \oo$, and we see that $B^{1/p}_{p,\oo} \not\subset C$. As it turns out, such jump discontinuities are not permissible for functions in $B^{1/p}_{p,q}$ as soon as $q$ is finite.

\begin{proposition}\label{P:nojumps}
	Let $1 < p < \oo$ and $1\le q < \oo$, $f: [0,T] \to \RR^m$, and $t_0 \in (0,T)$. Assume that $f$ has a left and right limit at $t_0$ that differ. Then $[f]_{B^{1/p}_{p,q}([0,T])} = \oo$.
\end{proposition}

\begin{proof}
	There exists $\mu > 0$ such that, for all sufficiently small $\tau > 0$, if $s \in (t_0 - \tau,t_0)$ and $t \in (t_0,t_0 + \tau)$, then $|f_s - f_t| > \mu$. Therefore, for some sufficiently small $\tau_0 \in (0,T)$ and all $0 < \tau < \tau_0$,
	\[
		\omega_p(f,\tau)^p \ge \int_{t_0 - \tau}^{t_0} |f_{t+\tau} - f_t|^p dt \ge \tau \mu^p,
	\]
	and so
	\[
		\int_0^T \left(\frac{\omega_p(f,\tau)}{\tau^{1/p}} \right)^q \frac{d\tau}{\tau} \ge \mu^q \int_0^{\tau_0} \frac{d\tau}{\tau} = \oo.
	\]
\end{proof}

Proposition \ref{P:nojumps} rules out jump discontinuities for $B^{1/p}_{p,q}$-functions with $1 < p < \oo$ and $1 \le q < \oo$. However, if $q > 1$, some discontinuities are still allowed, and functions may even be nowhere locally bounded.

\begin{proposition}\label{P:loglog}
	Let $\chi: \RR \to \RR$ be smooth, even, supported in $[-1/2,1/2]$, and equal to $1$ in $[-1/4,1/4]$, and define
	\[
		f_t = \chi_t \log |\log t| \quad \text{for } t \in [-1,1].
	\]
	Then $f \in B^{1/p}_{p,q}([-1,1])$ for any $1 < p < \oo$ and $1 < q \le \oo$.
\end{proposition}

\begin{proof}
	If $n = 3,4,\ldots$, then, because $f$ is even,
	\[
		\int_{-1}^{1-2^{-n}} |f_{t + 2^{-n}} - f_t|^p dt = 2\int_0^{1-2^{-n}} |f_{t+2^{-n}} - f_t|^p dt + 2\int_0^{2^{-n-1}} |f_{2^{-n} - t} - f_t|^pdt.
	\]
	We compute
	\[
		f'_t = \chi'_t \log |\log t| - \chi_t \frac{1}{t|\log t|} \quad \text{for } t \ne 0.
	\]
	Fix $m = 0,1,2,\ldots$, and $t \in [2^{-m-1}, 2^{-m}]$. If $m = 0,1,2,\ldots, n$, the mean value theorem gives
	\[
		|f_{t+2^{-n}} - f_t|^p \asymp_p \left(\log m + \frac{2^m}{m}\right)^p 2^{-np} \lesssim_p \frac{2^{(m-n)p}}{m^p},
	\]
	while for $k \in \NN$ and $m = n+k,n+k+1,\ldots$,
	\[
		|f_{t+2^{-n}} - f_t|^p + |f_{2^{-n} - t} - f_t|^p \lesssim_p |f_t|^p \lesssim_p (\log m)^p.
	\]
	Finally, if $t \in [2^{-n-k}, 2^{-n-1}]$, then the mean value theorem again gives
	\[
		|f_{t+2^{-n}} - f_t| \le \log |\log 2^{-n-k})| - \log |\log (2^{-n-1} + 2^{-n}) | \le \log(n+k) - \log (n-1) \lesssim \frac{k}{n},
	\]
	and similarly $|f_{2^{-n} - t} - f_t|^p \lesssim \frac{k}{n}$. Combining all three estimates gives
	\begin{align*}
		\int_0^{1 -2^{-n}}& |f_{t+2^{-n}} - f_t|^pdt + \int_0^{2^{-n-1}} |f_{2^{-n} - t} - f_t|^p dt\\
		&= \sum_{m =0}^\oo \int_{2^{-m-1}}^{2^{-m}} |f_{t+2^{-n}} - f_t|^pdt 
		+ \sum_{m=n=1}^\oo  \int_{2^{-m-1}}^{2^{-m}} |f_{2^{-n} - t} - f_t|^p dt
		\\
		&\lesssim_p 2^{-np} \sum_{m=1}^n 2^{(p-1)m} m^{-p} + \sum_{m=n+k}^\oo 2^{-m} (\log m)^p + 2^{-n}\frac{k^p}{n^p}\\
		&\lesssim_p 2^{-n} n^{-p} + 2^{-n-k} \log (n+k)^p + 2^{-n} \frac{k^p}{n^p}.
	\end{align*}
	Choosing $k \asymp \log n$ yields
	\[
		\int_0^{1 -2^{-n}} |f_{t+2^{-n}} - f_t|^pdt + \int_0^{2^{-n-1}} |f_{2^{-n} - t} - f_t|^p dt \lesssim_p 2^{-n} \left( \frac{\log n}{n} \right)^p,
	\]
	and we conclude that
	\[
		2^{n/p} \left( \int_{-1}^{1-2^{-n}} |f_{t+2^{-n}} - f_t|^p dt \right)^{1/p} \lesssim_p \frac{\log n}{n}
	\]
	is $\ell^q$-summable in $n$ if $1 < q \le \oo$. The proof is finished in view of Lemma \ref{L:discreteBesovnorm}.
\end{proof}

By constructing a convergent series in the space $B^{1/p}_{p,q}$ of appropriate translations of the function $f$ from Proposition \ref{P:loglog}, we immediately have the following.

\begin{corollary}\label{C:loglog}
	If $1 < p < \oo$ and $1 < q \le \oo$, then $B^{1/p}_{p,q}([0,T])$ contains functions that are unbounded on every sub-interval of $[0,T]$.
\end{corollary}

\subsection{Two-parameter Besov spaces}\label{subsec:Besov2parameter}

Let $A: \Delta_2(0,T) \to \RR^m$ be measurable and, for $p \in (0,\oo]$ and $0 < \tau \le T$, we define
\[
	\Omega_p(A,\tau) := \sup_{0 \le h \le \tau} \left( \int_0^{T-h} |A_{r,r +h}|^p \; dr\right)^{1/p}.
\]
For a three-parameter map $A: \Delta_3(0,T) \to \RR^m$, we define
\[
	\oline{\Omega}_p(A,\tau) := \sup_{0 \le \theta \le 1} \sup_{0 \le h \le \tau} \left( \int_0^{T-h} |A_{r,r + \theta h,r + h}|^p\;dr \right)^{1/p}.
\]
The dependence of $\Omega$ and $\oline{\Omega}$ on the interval $[0,T]$ is suppressed for notational convenience. Note that, for $f: [0,T] \to \RR^m$, we have $\Omega_p(\delta f,t) = \omega_p(f,t)$.

\begin{definition}\label{D:Besovtwoparameter}
Fix $\alpha > 0$ and $p,q \in (0,\oo]$. We define, for $A: \Delta_2(0,T) \to \RR^m$,
\[
	\mbb B^\alpha_{pq}([0,T]) := \left\{ A: \Delta_2(s,t) \to \RR^m : \nor{A}{\mbb B^\alpha_{pq}([0,T])} := \left[ \int_0^{T} \left( \frac{\Omega_p(A,\tau)}{\tau^\alpha}\right)^q \frac{d\tau}{\tau} \right]^{1/q} < \oo \right\}.
\]
We also set, for $A: \Delta_3(0,T) \to \RR^m$,
\[
	\nor{A}{\oline{\mbb B}^\alpha_{pq}([0,T])} := \left[ \int_0^{T} \left( \frac{\oline{\Omega}_p(A,\tau)}{\tau^\gamma}\right)^q \frac{d\tau}{\tau} \right]^{1/q}.
\]
Finally, for a non-decreasing function $\omega: [0,\oo) \to [0,\oo)$ satisfying $\lim_{r \to 0^+} \omega(r) = 0$, we set
\[
	\mbb B^\omega_{pq}([0,T]) := \left\{ A: \Delta_2(s,t) \to \RR^m : \nor{A}{\mbb B^\alpha_{pq}([0,T])} := \left[ \int_0^{T} \left( \frac{\Omega_p(A,\tau)}{\omega(\tau)}\right)^q \frac{d\tau}{\tau} \right]^{1/q} < \oo \right\}
\]
and, for $A: \Delta_3([0,T],\RR^m)$,
\[
	\nor{A}{\oline{\mbb B}^\omega_{pq}([0,T])} := \left[ \int_0^{T} \left( \frac{\oline{\Omega}_p(A,\tau)}{\omega(\tau)}\right)^q \frac{d\tau}{\tau} \right]^{1/q}.
\]
\end{definition}

We note that, when $p = q = \oo$, $\nor{\cdot}{\mbb B^\alpha_{pq}}$ measures H\"older-type regularity, and we use the notation 
\[
	\mbb C^\alpha([0,T]) := \mbb B^\alpha_{\oo,\oo}([0,T]).
\]

As before, when $1 \le p,q \le \oo$, $\nor{\cdot}{\mbb B^\alpha_{pq}}$ is a norm and $\mbb B^\alpha_{pq}$ is a Banach space. In all cases, $\mbb B^\alpha_{pq}$ is a complete metric space:

\begin{definition}\label{D:BBesovspacemetric}
	Assume $\alpha > 0$ and $0 < p,q \le \oo$. Then $\mbb B^\alpha_{pq}([0,T],\RR^m)$ is a complete metric space with the metric
	\begin{equation}\label{BBmetric}
	\mathrm d_{\mbb B^\alpha_{pq}([0,T])}(A, \tilde A) :=
	\begin{dcases}
		\nor{A - \tilde A}{\mbb B^\alpha_{pq}([0,T])} & \text{if } 1 \le p,q \le \oo,\\
		\int_0^T \left( \frac{\Omega_p(A - \tilde A,\tau)}{\tau^\alpha} \right)^q \frac{d\tau}{\tau} & \text{if } 0 < q < 1 \text{ and } q \le p, \text{ and} \\
		\left[ \int_0^T \left( \frac{\Omega_p(A - \tilde A,\tau)}{\tau^\alpha} \right)^q \frac{d\tau}{\tau} \right]^{p/q} & \text{if } 0 < p < 1 \text{ and } q > p.
	\end{dcases}
	\end{equation}
	The same holds if $\tau \mapsto \tau^\alpha$ is replaced with an arbitrary modulus $\omega$.
\end{definition}

%

We now present a useful Besov-H\"older embeddings in the two-parameter setting. The difference between $\mbb B^\omega_{pq}([0,T],\RR^m)$ and $B^\omega_{pq}([0,T],\RR^m)$ is that elements $A$ of $\mbb B^\omega_{pq}$ do not in general satisfy $\delta A = 0$, and therefore, in order to generalize Proposition \ref{P:Besovembedding}, some condition is needed on $\delta A$. The one we present here is a kind of mixed continuity condition. A more general condition can perhaps be devised; however, the continuity condition for $\delta A$ is in practice easy to check in our applications, and it usually reduces to the case of Proposition \ref{P:Besovembedding}. 

\begin{proposition}\label{P:BBesovembedding}
	Assume that $A: \Delta_2(0,T) \to \RR^m$ is measurable, $0 < p,q \le \oo$, $\omega: [0,\oo) \to [0,\oo)$ is non-decreasing, $\lim_{r \to 0^+} \omega(r) = 0$,
	\begin{equation}\label{omegacontrol}
		\omega(2\tau) \lesssim \omega(\tau) \quad \text{for all }\tau > 0,
	\end{equation}
	\begin{equation}\label{zetacompatible}
		\left\{
		\begin{split}
		&[0,\oo) \ni \tau \mapsto \zeta(\tau) := \left( \frac{\omega(\tau)}{\tau^{1/p}} \right)^{1 \wedge p} \text{ satisfies}\\
		&\int_0^T \zeta(h) \frac{dh}{h} \le K \zeta(T) \quad \text{for some } K > 0 \text{ and all } T > 0,
		\end{split}
		\right.
	\end{equation}
	$\nor{A}{\mbb B^\omega_{pq}([0,T])} < \oo$, and, for some $\theta \in (0,1/2]$ and $M > 0$,
	\begin{equation}\label{deltaAHolder}
		|\delta A_{sut}|  \le M \zeta\left( \left( (u-s) \wedge (t-s)\right)^{\theta} \left( (u-s) \vee (t-s) \right)^{1-\theta} \right)^{1 \vee \frac{1}{p}}.
	\end{equation}
	Then there exists a continuous version of $A$, denoted also by $A$, such that
	\begin{equation}\label{eq:2}
		\sup_{0 \le s < t \le T} \frac{|A_{st}|}{\omega(t-s) (t-s)^{-1/p} } \lesssim_{K,p,q,\theta} \nor{A}{\mbb B^\omega_{pq}([0,T])} + M.
	\end{equation}
\end{proposition}

\begin{remark}\label{R:BBesovembeddingmodulus}
	The modulus $\omega$ satisfies the assumptions of Proposition \ref{P:BBesovembedding} if, for instance,
	\[
		\omega(r) = r^\gamma |\log (r \wedge 1/2)|^\beta \quad \text{for some } \gamma > \frac{1}{p} \text{ and } \beta \ge 0.
	\]
\end{remark}

Proposition \ref{P:BBesovembedding} is proved by taking advantage of some Campanato-type characterizations of H\"older continuity, which, for the Besov-type spaces in question, are routine to verify. More precisely, if $E$ is a Banach space, $f: [0,T] \to E$, and $\beta > 0$, then the $\beta$-H\"older semi-norm of $f$ is equivalent to
\[
\sup_{t_0 \in [0,T]} \sup_{0 < r < t_0 \wedge T - t_0} \frac{1}{r^\beta} \frac{1}{2r} \int_{r_0 - r}^{t_0 + r} \nor{f_t - \frac{1}{2r} \int_{t_0 - r}^{t_0 + r} f_sds}{E}dt,
\]
a result which goes back at least to related work of Campanato \cite{Campanato}.
The analytic essence of this result is contained in the following estimate, which we already formulate in a way applicable to two-parameter functions (in the one-parameter case, the approximate triangle inequality \eqref{deltaAHolderCampanato} can be replaced by a genuine triangle inequality).

\begin{proposition}\label{P:twoparameterCampanato}
Let $\zeta: [0,\oo) \to [0,\oo)$ be a non-decreasing function such that
\begin{equation}\label{zetacontrol}
\zeta(2r) \lesssim \zeta(r) \quad \text{for all } r > 0,
\end{equation}
and
\begin{equation}\label{zetaDini}
\int_0^T \zeta(h) \frac{dh}{h} \le K \zeta(T) \quad \text{for some } K > 0 \text{ and all } T > 0.
\end{equation}
Let $\rho: [0,T]^2 \to [0,\oo)$ be a locally integrable function such that, for all $t_0 \in (0,T)$ and $r < t_0 \wedge T - t_0$,
\begin{equation}\label{ACampanato}
\frac{1}{r^2} \int_{t_0-r}^{t_0 + r} \int_{t_0 - r}^{t_0+r} \rho_{st} dtds \le \zeta(r)
\end{equation}
and, for some $\theta \in (0,1/2]$ and all $r,s,t \in [0,T]$,
\begin{equation}\label{deltaAHolderCampanato}
\rho_{rt} \le \rho_{rs} + \rho_{st} +
\zeta\left( (|s-r| \wedge |t-s|)^{\theta}(|s-r| \vee |t-s|)^{(1-\theta)} \right).
\end{equation}
Then, there exists a full measure subset $X\subseteq [0,T]$ such that
\[
\sup_{s,t \in X, s \ne t} \frac{\rho_{st}}{\zeta(|t-s|)} \lesssim_{\theta,K} 1.
\]
\end{proposition}

\begin{proof}	
For brevity, we denote averages by $\fint_{a}^{b} := (b-a)^{-1} \int_{a}^{b}$.

{\it Step 1.} Fix $t_0 \in (0,T)$, $r < t_0 \wedge (T - t_0)$, and $n \ge 1$, and set
\begin{equation}\label{alphan}
\alpha_n(t_{0}) := \fint_{t_0 - 2^{-n}r}^{t_0 + 2^{-n}r} \fint_{t_0 - r}^{t_0 + r} \rho_{st} dtds.
\end{equation}
Using \eqref{deltaAHolderCampanato}, we write, for all $s,u,t$ with $|s-t_0| < 2^{-n}r$, $|t-t_0| < r$, $|u-t_{0}|<2^{-n+1}r$,
\[
\rho_{st} \le \rho_{su} + \rho_{ut} + \zeta(4 \cdot 2^{-n\theta} r).
\]
Averaging in $s,t,u$, we obtain
\begin{align*}
\alpha_{n}(t_{0})
&\le
\fint_{t_0 - 2^{-n+1}r}^{t_0 + 2^{-n+1}r} \fint_{t_0 - 2^{-n}r}^{t_0 + 2^{-n}r}\rho_{su} ds du
+ \fint_{t_0 - 2^{-n+1}r}^{t_0 + 2^{-n+1}r} \fint_{t_0 - r}^{t_0 + r} \rho_{ut} dt du
+ \zeta(4 \cdot 2^{-n\theta} r)
\\ &=
\I + \II + \III.
\end{align*}
From \eqref{ACampanato}, we obtain
\[
\I \le
2 \fint_{t_0 - 2^{-n+1}r}^{t_0 + 2^{-n+1}r} \fint_{t_0 - 2^{-n+1}r}^{t_0 + 2^{-n+1}r} \rho_{su} duds
\le 2\zeta(2^{-n+1} r)
\lesssim \zeta(2^{-n} r).
\]
We also have $\II = \alpha_{n-1}(t_{0})$, and, by \eqref{zetacontrol}, $\III \lesssim_{\theta} \zeta(2^{-n\theta} r)$.
We conclude that
\[
\alpha_n(t_{0}) - \alpha_{n-1}(t_{0})
\lesssim
\zeta(2^{-n}r) + \zeta(2^{-n\theta} r)
\lesssim
\zeta(2^{-n\theta} r).
\]
By hypothesis, $\alpha_0 \le \zeta(r)$, and so, taking $n \to \oo$ in \eqref{alphan} and using the Lebesgue differentiation theorem gives, for every $r>0$,
\begin{equation*}
\esssup_{t_{0} \in (r,1-r)}
\frac{1}{2r}  \int_{t_0 - r}^{t_0 + r} \rho_{t_0,t} dt
\lesssim
\zeta(r) + \sum_{n=1}^\oo \zeta(2^{-n\theta}r)
\lesssim_{\theta,K} \zeta(r),
\end{equation*}
where we used \eqref{zetaDini} in the last step.
Using the above inequality for a countable dense set of $r$'s and \eqref{zetacontrol}, we see that there exists a full measure subset $X \subseteq (0,T)$ such that, for every $t_{0} \in X$, and every $r < \min(t_{0},T-t_{0})$, we have
\begin{equation}\label{comparetoaverage}
\fint_{t_0 - r}^{t_0 + r} \rho_{t_0,t} dt
+
\fint_{t_0 - r}^{t_0 + r} \rho_{t,t_0} dt
\lesssim_{\theta,K} \zeta(r),
\end{equation}
the second term being bounded by employing a symmetric argument.

{\it Step 2.} Fix $s,t \in X$ such that
\[
r := |t-s| < \dist(\{s,t\},\{0,T\}).
\]
By \eqref{deltaAHolderCampanato}, for $u \in (s\wedge t,s\vee t)$, we have
\[
\rho_{st} \leq \rho_{su} + \rho_{ut} + \zeta(r).
\]
Averaging this inequality in $u$ and using \eqref{comparetoaverage}, we obtain $\rho_{st} \lesssim_{K,\theta} \zeta(|t-s|)$.

{\it Step 3.} Fix $0 < s < t \le T/2$ with $s,t \in X$ and $s \leq t/2$.
Choose $t=t_{0}>t_{1}>\dotsc>t_{k}=s$ such that $t_{j}/2 < t_{j+1} \leq t_{j}/\sqrt{2}$ and $t_{j} \in X$ for all $j = 0,\dotsc,k-1$.
By induction on $l$, using \eqref{deltaAHolderCampanato}, we obtain
\[
\rho_{st} \leq \rho_{s t_{l}} + \sum_{j=0}^{l-1} \rho_{t_{j+1} t_{j}} + \zeta(t_{j}-s).
\]
Using this inequality with $l=k-1$ and the result from Step 2, we obtain
\begin{align*}
\rho_{st}
\leq \sum_{j=1}^k \rho_{t_{j}, t_{j-1}}
+ \sum_{j=0}^{k-2} \zeta( t_{j} - s ) 
\lesssim_{\theta,K} \sum_{j=0}^{k-1} \zeta( t_{j} - s ) 
\lesssim_{K,\theta} \zeta(t-s).
\end{align*}
A symmetric argument gives similar estimates for $T/2 \le s < t < T$, as well as for $\rho_{ts}$.
It is then easy to conclude for all $s,t \in X$.
\end{proof}

Proposition~\ref{P:twoparameterCampanato} can be specialized to a Campanato-type characterization of H\"older continuity for functions with values in a metric space. We thank Pavel Zorin-Kranich for stream-lining an earlier proof of ours.

\begin{proposition}\label{P:generalCampanato} 
If $f: [0,T] \to E$ is measurable and $\beta \in (0,1)$, then
\begin{equation}\label{HandC}
\sup_{(s,t) \in \Delta_2(0,T)} \frac{d(f_s,f_t)}{(t-s)^\beta}
\asymp_\beta
\sup_{t_0 \in (0,T)} \sup_{0 < r < t_0 \wedge T - t_0} \frac{1}{r^\beta} \frac{1}{(2r)^{2}} \int_{t_0-r}^{t_0+r} \int_{t_0-r}^{t_0+r} d(f_{s},f_{t}) dt ds.
\end{equation}
More precisely, when the right-hand side is finite, there exists a version of $f$ which is $\beta$-H\"older continuous, and the equivalence of semi-norms holds.
\end{proposition}


\begin{proof}
The $\gtrsim_\beta$ estimate is obvious. To prove the $\lesssim_\beta$ direction, we invoke Proposition~\ref{P:twoparameterCampanato} with $\rho_{st}=d(f_{s},f_{t})$ and $\zeta(r)=Cr^{\beta}$ to conclude that there exists a full measure subset $X \subset [0,T]$ such that, for every $s,t \in X$,
\[
d(f_{s},f_{t}) \lesssim \abs{s-t}^{\beta}.
\]
It follows that $f|_{X}$ can be extended to a $\beta$-H\"older continuous function on $[0,T]$ that coincides with $f$ on $X$.
\end{proof}

As an immediate consequence, we present the proof of Proposition \ref{P:Besovembedding}.

\begin{proof}[Proof of Proposition \ref{P:Besovembedding}]
	For $f \in B^\alpha_{pq}([0,T];E)$, we have, for all $h \in (0,T)$,
	\[
		\left[ \int_0^{T-h} d(f_{t+h}, f_t)dt \right]^{1/p} \lesssim [f]_{B^\alpha_{pq}} h^\alpha.
	\]
For any $t_0 \in (0,T)$ and $0 < r < t_0 \wedge (T - t_0)$, we have
	\begin{align*}
        \MoveEqLeft
 \frac{1}{(2r)^2} \int_{t_0 -r}^{t_0 + r} \int_{t_0 - r}^{t_0 + r} d(f_t, f_s)dtds \\
		&= \frac{1}{(2r)^2} \int_{t_0 - r}^{t_0 +r} \int_{t_0 - r - s}^{t_0 + r - s}  d( f_{s+h}, f_s)dh ds\\
		&= \frac{1}{(2r)^2} \int_{-2r}^{2r} \int_{t_0 - r + h_-}^{t_0 + r - h_+} d( f_{s+h},  f_s)ds dh\\
		&\le \frac{1}{(2r)^2} \int_{-2r}^{2r} (2r - |h|)^{1 - 1/p} \left[ \int_{h_-}^{T-h_+}  d( f_{s+h}, f_s)^pds \right]^{1/p}dh \\
		&\lesssim \frac{[f]_{B^\alpha_{pq}}}{r^{1+1/p}} \int_0^{2r} h^\alpha dh \\
		&\lesssim [f]_{B^\alpha_{pq}} r^{\alpha - 1/p}.
	\end{align*}
	The result now follows from Proposition \ref{P:generalCampanato}.
        \end{proof}
        
\begin{remark}\label{R:2to1param}
	Proposition \ref{P:Besovembedding} can also be proved by immediately appealing to Proposition \ref{P:BBesovembedding}.
\end{remark}
        
We finally present the proof of the two-parameter Besov-H\"older embedding.
        
\begin{proof}[Proof of Proposition \ref{P:BBesovembedding}]
If $p = \oo$, then the result is trivial, so assume $p < \oo$.

The map $A$ is defined only for $(s,t)$ with $s \le t$.
To match the setting of Proposition \ref{P:twoparameterCampanato}, for $s > t$, we set $A_{st} = -A_{ts}$.
Then $|\delta A_{rst}|$ remains constant over all permutations of $r,s,t \in [0,T]$. Therefore, \eqref{deltaAHolder} continues to hold regardless of the order of $s$, $u$, and $t$, and, for $\tau \in [0,T]$,
\[
	\sup_{0 < h < \tau} \left( \int_h^T \abs{ A_{t,t-h}}^p dt\right)^{1/p} = \Omega_p(A,\tau).
\]
For all $0 < h < T$, we have
\[
	\left( \int_h^T \abs{ A_{t,t-h}}^p dt\right)^{1/p} + \left( \int_0^{T-h} \abs{A_{t,t+h}}^{p} dt\right)^{1/p} \lesssim_{p,q} \nor{A}{\mbb B^\omega_{pq}} \omega(h).
\]

Now fix $t_0 \in (0,T)$ and $r < t_0 \wedge T - t_0$. If  $1 \le p < \oo$, then
\begin{align*}
	\frac{1}{r^2} \int_{t_0 - r}^{t_0 + r} \int_{t_0 - r}^{t_0 + r} \abs{ A_{st}}dsdt
	&= \frac{1}{r^2} \int_{t_0 - r}^{t_0 + r} \int_{t_0 -r -s}^{t_0 + r-s} \abs{ A_{s,s+h}}dhds \\
	&= \frac{1}{r^2} \int_{-2r}^{2r} \int_{t_0 - r + h_-}^{t_0 + r - h_+} \abs{ A_{s,s+h}}ds dh\\
	&\le \frac{2^{1-1/p}}{r^{1 + 1/p}} \int_{-2r}^{2r} \left( \int_{h_-}^{T-h_+} \abs{ A_{s,s+h}}^{p}ds\right)^{1/p} dh \\
	&\lesssim_{p,q} \frac{1}{r^{1+1/p}}\nor{A}{\mbb B^\omega_{pq}} \int_{-2r}^{2r} \omega(|h|) dh \\
	&\lesssim \nor{A}{\mbb B^\omega_{pq}} \zeta(r).
\end{align*}
Thus, \eqref{ACampanato} and \eqref{deltaAHolderCampanato} are satisfied with $\rho_{st} = |A_{st}|$. It is easy to see that $\zeta$ satisfies the assumptions of Proposition \ref{P:twoparameterCampanato} (the implicit constants depend additionally on $p$).
Proposition~\ref{P:twoparameterCampanato} implies that \eqref{eq:2} holds for $s,t$ in a full measure subset $X \subset [0,T]$.
Using~\eqref{deltaAHolder}, we can extend $A|_{X\times X}$ to a continuous function on $[0,T]^{2}$ that still satisfies \eqref{eq:2}.

If $0 < p < 1$, then a similar computation yields
\[
	\frac{1}{r^2} \int_{t_0 - r}^{t_0 + r} \int_{t_0 - r}^{t_0 + r} \abs{ A_{st}}^p dsdt
	\lesssim_{p,q} \nor{A}{\mbb B^\omega_{pq}}^p \zeta(r),
\]
and so the result follows upon appealing to Proposition \ref{P:twoparameterCampanato} with $\rho_{st} = |A_{st}|^p$.
\end{proof}

Proposition \ref{P:BBesovembedding} will often be paired with the following interpolation estimate, which allows for a simultaneous loss of regularity and gain of integability.

\begin{lemma}\label{L:interpolate}
	Assume that $0 < \alpha < \gamma$, $0 < p < r \le \oo$, $0 < q \le \oo$, and, for some $\delta > 0$,
	\begin{equation}\label{A:interpolate}
		\alpha < \frac{\gamma p}{r} + \delta \left( 1 - \frac{p}{r} \right).
	\end{equation}
	Then, for all $A: \Delta_2(0,T) \to \RR^m$,
	\[
		\nor{A}{\mbb B^{\alpha}_{r,q}([0,T])} \lesssim_{\alpha, \gamma,p, r,q,\delta} T^{\delta\left(1 - \frac{p}{r}\right) + \frac{\gamma p}{r} - \alpha} \nor{A}{\mbb C^\delta([0,T])}^{1 - \frac{p}{r}}\nor{A}{\mbb B^{\gamma}_{p,q}([0,T])}^{\frac{p}{r}}.
	\]
\end{lemma}


\begin{proof}
	We compute
	\[
		\Omega_{r}(A,\tau) \le \Omega_\oo(A,\tau)^{1 - \frac{p}{r}} \Omega_{p}(A,\tau)^{\frac{p}{r}} \le \nor{A}{\mbb C^\delta}^{1 - \frac{p}{r}} \tau^{\delta\left(1 - \frac{p}{r}\right)} \Omega_{p}(A,\tau)^{\frac{p}{r}}
	\]
	and so
	\[
		\frac{\Omega_{r}(A,\tau)}{\tau^{\alpha}}
		\le \nor{A}{\mbb C^\delta}^{1 - \frac{p}{r}} \tau^{\delta\left(1 - \frac{p}{r}\right) + \frac{\gamma p}{r} - \alpha} \left( \frac{\Omega_{p}(A,\tau)}{\tau^{\gamma}} \right)^{\frac{p}{r}}.
	\]
	We take the $L^q(d\tau/\tau)$ (quasi)-norm on both sides and conclude by H\"older's inequality that
	\begin{align*}
		\nor{A}{\mbb B^{\alpha}_{r,q}} 
		&\le \nor{A}{\mbb C^\delta}^{1 - \frac{p}{r}} \left[ \int_0^T \tau^{q\left[\delta\left(1 - \frac{p}{r}\right) + \frac{\gamma p}{r} - \alpha\right]}  \left( \frac{\Omega_{p}(A,\tau)}{\tau^{\gamma}} \right)^{\frac{pq}{r}} \frac{d\tau}{\tau} \right]^{1/q}\\
		&\lesssim T^{\delta\left(1 - \frac{p}{r}\right) + \frac{\gamma p}{r} - \alpha} \nor{A}{\mbb C^\delta}^{1 - \frac{p}{r}} \nor{A}{\mbb B^{\gamma}_{p,q}}^{\frac{p}{r}}.
	\end{align*}
\end{proof}

Lemma \ref{L:interpolate} can be extended to allow for more general moduli; the proof is almost identical, and thus we omit it.

\begin{lemma}\label{L:interpolatemodulus}
	Assume that $\alpha > 0$, $0 < p < r \le \oo$, $0 < q \le \oo$, $\omega,\rho:[0,\oo) \to [0,\oo)$ are non-decreasing and satisfy $\lim_{\tau\to0^+} \omega(\tau) = \lim_{\tau\to0^+} \rho(\tau) = 0$, and
	\[
		\sigma(T) := \left[\int_0^T \left(\frac{\rho(\tau) \omega(\tau)^{\frac{p}{r-p} } }{\tau^{\frac{\alpha r}{r-p} } } \right)^q \frac{d\tau}{\tau}\right]^{\frac{r-p}{rq}} < \oo.
	\]
	Then, for all $A: \Delta_2(0,T) \to \RR^m$,
	\[
		\nor{A}{\mbb B^{\alpha}_{r,q}([0,T])} \lesssim_{\alpha,p, r,q} \sigma(T) \nor{A}{\mbb C^\rho([0,T])}^{1 - \frac{p}{r}}\nor{A}{\mbb B^{\omega}_{p,q}([0,T])}^{\frac{p}{r}}.
	\]
\end{lemma}

We conclude with a result which is instrumental in overcoming the difficulty that functions in $\mbb B^\alpha_{pq}$ are a priori only defined up to sets of Lebesgue measure zero. 

\begin{lemma}\label{L:additive}
	Let $A: \Delta_2(0,T) \to \RR^m$ be measurable, and assume that, for some $0 < p \le \oo$, $\oline{\Omega}_p(\delta A,\tau) = 0$ for all $\tau \in (0,T]$. Then there exists a measurable function $F:[0,T] \to \RR^m$ and a version of $A$, denoted also by $A$, such that $A = \delta F$.
\end{lemma}

\begin{proof}	
	For $0 \le t < s \le T$, define $A_{s,t} := -A_{s,t}$, and then extend the definition of $\delta A_{r,s,t}$ to all $(r,s,t) \in [0,T]^3$. We have, for all $h \in (0,T]$ and $0 < \theta < 1$,
	\[
		\left[ \int_0^{T-h} |\delta A_{t,t+\theta h, t +h}|^p dt \right]^{1/p} = 0,
	\]
	with the obvious modfication if $p = \oo$, and, therefore, there exists $N_{\theta,h} \subset [0,T-h]$ such that $|N_{\theta,h}| = 0$ and
	\[
		\delta A_{t,t+\theta h, t+h} = 0 \quad \text{for all } t \in [0,T-h] \backslash N_{\theta,h}.
	\]
	By Fubini's theorem, there exists $N \subset [0,T]^3$ with $|N| = 0$ such that $\delta A = 0$ in $[0,T]^3 \backslash N$. Another application of Fubini's theorem implies that there exists $r \in (0,T)$ such that
	\[
		\delta A_{r,s,t} =0 \quad \text{for almost every } (s,t) \in [0,T]^2.
	\]
	We now define $F_t := A_{r,t}$ for $t \in [0,T]$. Then, for almost every $(s,t) \in [0,T]^2$,
	\[
		F_t - F_s = A_{st} + \delta A_{r,s,t} = A_{st}.
	\]
	\end{proof}

\section{The Besov ``sewing lemma''} \label{sec:BesovSew}

The goal of this section is to generalize the ``sewing lemma,'' which has been proved in H\"older and variation contexts \cite{feyel2006curvilinear, gubinelli2004controlling,feyel2008noncommutative}, to the Besov scale. Given $A: \Delta_2(0,T) \to \RR^m$ for which 
\[
	\nor{A}{\mbb B^\alpha_{p_1,q_1}([0,T])} + \nor{\delta A}{\oline{\mbb B}^\gamma_{p_2,q_2}([0,T])} < \oo
\]
for appropriate parameters $0 < \alpha < \gamma$ and $p_1,p_2,q_1,q_2 \in (0,\oo]$, we construct a generalized integral as a limit of Riemann sums. Due to the various cases involving different regimes of parameters, the result is split up into three theorems.  In Theorem \ref{T:Besovsewing}, the regularity satisfies a strict inequality $\gamma > 1 \vee \frac{1}{p_2}$, while in Theorem \ref{T:Besovsewingendpoint}, the case $\gamma = 1 \vee \frac{1}{p_2}$ can be treated as long as $0 < q_2 \le 1 \wedge p_2$. Finally, Theorem \ref{T:Besovsewingcontinuous} combines the previous theorems with the generalized H\"older embedding Proposition \ref{P:BBesovembedding} in order to regain integrability, especially in the case where $p_1 > p_2$.

\subsection{The statements}

Given $A: \Delta_2(0,T) \to \RR^m$ and a partition $P := \left\{ 0 = \tau_0 < \tau_1 < \cdots < \tau_N = 1 \right\}$ of the unit interval $[0,1]$, we define, for $(s,t) \in \Delta_2(0,T)$,
\[
	\mathscr I_P A_{st} := \sum_{i=1}^N A_{s + \tau_{i-1}(t-s), s + \tau_i (t-s)} \quad \text{and} \quad \mathscr R_P A := \mathscr I_PA - A.
\]

\begin{theorem}\label{T:Besovsewing}
	Assume that $0 < p_1,p_2,q_1,q_2 \le \oo$, $0 < \alpha < 1$, $\gamma > 1 \vee \frac{1}{p_2}$, and
	\[
		\nor{A}{\mbb B^\alpha_{p_1,q_1}([0,T])} + \nor{\delta A}{\oline{\mbb B}^\gamma_{p_2,q_2}([0,T])} < \oo.
	\]
	Then there exist $\mathscr IA \in B^\alpha_{p_1 \wedge p_2, q_1 \vee q_2}([0,T])$ and $\mathscr RA \in \mbb B^\gamma_{p_2,q_2}([0,T])$ such that
	\begin{equation}\label{generalconvergence}
		\lim_{\norm{P} \to 0} \nor{\mathscr I_{P} A  - \delta \mathscr IA}{\mbb B^\gamma_{p_2,r}} =
		\lim_{\norm{P} \to 0} \nor{\mathscr R_{P} A  - \mathscr RA}{\mbb B^\gamma_{p_2,r}} =
		 0 \quad \text{for all } r \ge q_2 \text{ and } r > \frac{1}{\gamma}.
	\end{equation}
	Moreover, we have the bounds
	\begin{equation}\label{Rbound}
			\nor{\mathscr R A}{\mbb B^\gamma_{p_2,q_2}([0,T])} \lesssim_{p_2,q_2,\gamma} \nor{\delta A}{\oline{\mbb B}^\gamma_{p_2,q_2}([0,T])}
	\end{equation}
	and
	\begin{equation}\label{Ibound}
			[\mathscr IA]_{B^\alpha_{p_1 \wedge p_2,q_1 \vee q_2}([0,T])} \lesssim_{T,p_1,p_2,q_1,q_2,\gamma} \nor{A}{\mbb B^\alpha_{p_1,q_1}} + T^{\gamma - \alpha} \nor{\delta A}{\oline{\mbb B}^{\gamma}_{p_2,q_2}([0,T])}.
	\end{equation}
\end{theorem}

\begin{remark}\label{R:Besovsewingpqless1}
	In Theorem \ref{T:Besovsewing}, if $p_2 \ge 1$, then the condition on $\gamma$ reads simply $\gamma > 1$, which is familiar from other versions of the sewing lemma. However, once $0 < p_2 < 1$, it is necessary to assume that $\gamma > 1/p_2$. This is more than just a technical difficulty. For instance, in the Young integration regime, this condition puts restrictions on the types of discontinuities allowed by the paths (see Remark \ref{R:Youngintegrals}).
	
	If $q_2 > 1/\gamma$, then the limit of $\mathscr I_P A$ along arbitrary partitions with mesh-size tending to $0$ may be taken in $\mbb B^\gamma_{p_2,q_2}$. Otherwise, if $q_2 \le 1/\gamma$, the convergence is weaker. We note that, in the course of the proof, it will be shown that the convergence along dyadic partitions always holds in $\mbb B^\gamma_{p_2,q_2}$.
\end{remark}

We next address an interesting endpoint case that does not arise in the H\"older or variation sewing frameworks. By tuning the secondary integration parameter $q_2$ lower, the regularity parameter $\gamma$ for $\delta A$ may be taken down to $\gamma = 1 \vee (1/p_2)$. The price paid is that the ``remainder map'' $\mathscr RA$ no longer has the same regularity modulus $\tau \mapsto \tau^{1 \vee \frac{1}{p_2}}$ as $\delta A$, and there is some at-most-logarithmic loss (compare \eqref{Rbound} to \eqref{Romegabound} below).

Throughout the rest of the paper, we fix $(\ell_r)_{0 < r \le \oo} : [0,T] \to [0,\oo)$ satisfying
\begin{equation}\label{loglike}
	\left\{
	\begin{split}
		&\text{for each $r \in (0,\oo]$, $\ell_r: [0,T] \to [0,\oo)$ is non-increasing,}\\
		&\int_0^T \left( \frac{1}{\ell_r(h)} \right)^r \frac{dh}{h} < \oo, \text{ and}\\
		&\int_0^\delta \ell_r(h)^\eta h^\zeta \frac{dh}{h} \lesssim_{r,\eta,\zeta} \ell_r(\delta)^\eta \delta^\zeta \quad \text{for all } \eta,\zeta,\delta > 0.
	\end{split}
	\right.
\end{equation}
For instance, we may take $\ell_r(h) = |\log (h \wedge 1/2)|^{1/r + \eps}$ for $\eps > 0$ if $r < \oo$, and $\ell_\oo \equiv 1$. We then set
\begin{equation}\label{omegar}
	\omega_r(h) := h^{1 \vee \frac{1}{p_2}} \ell_r(h) \quad \text{for } h \in [0,T] \text{ and } 0 < r \le \oo.
\end{equation}

\begin{theorem}\label{T:Besovsewingendpoint}
	Assume that $0 < p_1,q_1 \le \oo$, $0 < q_2 \le 1 \wedge p_2$, $0 < \alpha < 1$, and
	\[
		\nor{A}{\mbb B^\alpha_{p_1,q_1}([0,T])} + \nor{\delta A}{\oline{\mbb B}^{1 \vee \frac{1}{p_2}}_{p_2,q_2}([0,T])} < \oo.
	\]
	Then there exist $\mathscr IA \in B^\alpha_{p_1 \wedge p_2, q_1 \vee q_2}([0,T])$ and $\mathscr RA \in \bigcap_{r \in [q_2,\oo]} \mbb B^{\omega_r}_{p_2,r}([0,T])$  such that
	\begin{equation}\label{generalconvergence}
		\lim_{\norm{P} \to 0} \nor{\mathscr I_{P} A  - \delta \mathscr IA}{\mbb B^{1 \vee \frac{1}{p_2}}_{p_2,\oo}} = \lim_{\norm{P} \to 0} \nor{\mathscr R_{P} A  - \mathscr RA}{\mbb B^{1 \vee \frac{1}{p_2}}_{p_2,\oo}} = 0.
	\end{equation}
	Moreover, we have the bounds
	\begin{equation}\label{Romegabound}
			\nor{\mathscr R A}{\mbb B^{\omega_r}_{p_2,r}([0,T])} \lesssim_{p_2,r} \nor{\delta A}{\oline{\mbb B}^{1 \vee\frac{1}{p_2}}_{p_2,q_2}([0,T])} \quad \text{for all } r \in [q_2,\oo],
	\end{equation}
	and
	\begin{equation}\label{Iomegabound}
			[\mathscr IA]_{B^\alpha_{p_1 \wedge p_2,q_1 \vee q_2}([0,T])} \lesssim_{T,p_1,p_2,q_1,q_2,\kappa} \nor{A}{\mbb B^\alpha_{p_1,q_1}} + T^{\left(1 \vee \frac{1}{p_2} \right) - \alpha} \ell_{q_2}(T) \nor{\delta A}{\oline{\mbb B}^{1 \vee \frac{1}{p_2}}_{p_2,q_2}([0,T])}.
	\end{equation}
\end{theorem}

In order to use the general Besov sewing machinery to solve fixed point problems arising in the study of rough differential equations, it is important that the ``integral path'' increments $\delta\mathscr IA$ belong to the same space as $A$, that is, for $A \in \mbb B^\alpha_{p_1,q_1}$, we should have $\mathscr IA \in B^\alpha_{p_1,q_1}$. As indicated by \eqref{Ibound} and \eqref{Iomegabound}, this is the case when $p_1 \le p_2$ and $q_1 \ge q_2$. However, we will need to consider the case $p_1 > p_2$, because $A$ usually arises as a (tensor) product of increments of paths, which leads to a loss in integrability (for the same reason, we usually will have $q_1 > q_2$, so this parameter does not present a problem). In order to deal with this issue, we take advantage of the Besov-H\"older embeddings for the space $\mbb B^\alpha_{pq}$ given by Proposition \ref{P:BBesovembedding}.

\begin{theorem}\label{T:Besovsewingcontinuous}
	Let $0 < p_1,p_2,q_1,q_2 \le \oo$ and $0 < \alpha < \gamma$ and $A: \Delta_2(0,T) \to \RR^m$.
	\begin{enumerate}[(a)]
	\item Under the same conditions as Theorem \ref{T:Besovsewing}, assume in addition that $0 < p_2 < p_1 \le \oo$, $\gamma - \frac{1}{p_2} > \alpha - \frac{1}{p_1}$, and, for some $\theta \in (0,1/2)$ and $M > 0$ and for all $(s,u,t) \in \Delta_3(0,T)$,
	\begin{equation}\label{sewingdeltaAHolder}
		|\delta A_{sut}| \le M \left( (u-s) \wedge (t-s)\right)^{\theta(\gamma -1/p_2)} \left( (u-s) \vee (t-s) \right)^{(1-\theta)(\gamma - 1/p_2)}. 
	\end{equation}
	Then 
	\begin{equation}\label{RHolder}
		\nor{\mathscr RA}{\mbb C^{\gamma - 1/p_2}([0,T])} \lesssim_{p_2,q_2,\gamma,\theta} \nor{\delta A}{\oline{\mbb B}^\gamma_{p_2,q_2}} + M,
	\end{equation}
	and
	\begin{equation}\label{Iboundalt}
			[\mathscr IA]_{B^\alpha_{p_1,q_1 \vee q_2}([0,T])} \lesssim_{p_1,p_2,q_1,q_2,\alpha,\gamma} \nor{A}{\mbb B^\alpha_{p_1,q_1}} + T^{\gamma - \frac{1}{p_2} - \alpha + \frac{1}{p_1}} \left( \nor{\delta A}{\oline{\mbb B}^\gamma_{p_2,q_2}} + M\right).
	\end{equation}
	\item Under the same conditions as Theorem \ref{T:Besovsewingendpoint}, assume in addition that $0 < q_2 \le 1 < p_2 < p_1 \le \oo$, $0 < q_1 \le \oo$, $0 < \alpha < 1$, $1 - \frac{1}{p_2} > \alpha - \frac{1}{p_1}$, and, for some $M > 0$ and $\theta \in (0,1/2)$ and all $(s,u,t) \in \Delta_3(0,T)$,
	\begin{equation}\label{sewingdeltaAHolderendpoint}
		|\delta A_{sut}| \le M \left( (u-s) \wedge (t-s)\right)^{\theta(1 -1/p_2)} \left( (u-s) \vee (t-s) \right)^{(1-\theta)(1 - 1/p_2)}. 
	\end{equation}
	Then
	\begin{equation}\label{Rctsendpoint}
		\sup_{0 \le s \le t \le T} \frac{ |\mathscr R A_{st} |}{\omega_{q_2}(t-s)(t-s)^{-1/p_2}} \lesssim_{K,p_2,q_2,\theta} \nor{\delta A}{\oline{\mbb B}^1_{p_2,q_2}} + M,
	\end{equation}
	and
	\begin{equation}\label{Iboundaltendpoint}
		[\mathscr IA]_{B^\alpha_{p_1,q_2 \vee q_2}} \lesssim_{p_1,p_2,q_1,q_2,\alpha,\kappa,\theta,K} \nor{A}{\mbb B^\alpha_{p_1,q_1}} + T^{1 - \frac{1}{p_2} - \alpha + \frac{1}{p_1} } \ell_{q_2}(T) \left( \nor{\delta A}{\oline{\mbb B}^1_{p_2,q_2} } + M \right).
	\end{equation}
	\end{enumerate}
\end{theorem}

\begin{remark}\label{R:powerofT}
	The functions of $T$ in the various bounds \eqref{Ibound}, \eqref{Iboundalt}, and \eqref{Iboundaltendpoint} are used to establish the local Lipschitz continuity of the It\^o-Lyons map for the differential equations considered later. We note that the proof of the contractive property alone in the Picard iteration to construct solutions does not require this, as long as $q < \infty$, which can be seen as another (mild) advantage of the Besov setting.
\end{remark}

\subsection{The proofs of Theorems \ref{T:Besovsewing}, \ref{T:Besovsewingendpoint}, and \ref{T:Besovsewingcontinuous}}

An important first step is to establish convergence along the dyadic partitions 
\begin{equation}\label{dyadicpartitions}
	P_n := \left( \frac{k}{2^n}\right)_{k = 1,2,\ldots,2^n}, \quad n = 0,1,2,\ldots.
\end{equation}

\begin{lemma}\label{L:RnCauchy}
	Assume $0 < p_2,q_2 < \oo$ and $A:\Delta_2(0,T) \to \RR^m$. Then, for all $1 \le m < n \le \oo$ and $0 < h < T$,
	\[
		\Omega_{p_2}(\mathscr R_{P_n} A - \mathscr R_{P_m} A,h)
		\lesssim_{p_2,q_2}
		\begin{dcases}
			h \int_0^{h 2^{-m+1}} \frac{\oline{\Omega}_{p_2}(\delta A,\tau)}{\tau} \frac{d\tau}{\tau} & \text{if } 1 \le p_2,q_2 \le \oo,\\
			h^{1/p_2} \left[ \int_0^{h2^{-m+1}} \frac{\Omega_{p_2}(\delta A,\tau)^{p_2}}{\tau} \frac{d\tau}{\tau} \right]^{1/p_2} & \text{if } 0 < p_2 < 1 \text{ and } q_2 > p_2,\\
			h^{1 \vee \frac{1}{p_2}} \left[ \int_0^{h 2^{-m+1}} \left( \frac{\oline{\Omega}_{p_2}(\delta A,\tau)}{\tau^{1 \vee \frac{1}{p_2} }} \right)^{q_2} \frac{d\tau}{\tau} \right]^{1/q_2} & \text{if } 0 < q_2 < 1 \text{ and } q_2 \le p_2.
		\end{dcases}
	\]
\end{lemma}

\begin{proof}
	We first compute, for $n = 0,1,2,\ldots$, $h \in [0,T]$, and $r \in [0, T- h]$,
	\begin{equation}\label{Rn+1A-RnA}
		\mathscr R_{P_{n+1}}A_{r,r+h} - \mathscr R_{P_n}A_{r,r+h} = \sum_{k=1}^{2^{n+1}} 
		\delta A_{r+ \frac{2k-2}{2^{n+1}}h, r + \frac{2k-1}{2^{n+1}} h, r + \frac{2k}{2^{n+1}} h}.
	\end{equation}
	
	{\it Case 1:} $1 \le p_2,q_2 \le \oo$. Taking the $L^{p_2}$ norm over $r \in [0,T-h]$ on both sides of \eqref{Rn+1A-RnA} yields
	\[
		\Omega_{p_2}\left( \mathscr R_{P_{n+1}}A - \mathscr R_{P_n}A, h\right) \le 2^{n+1} \oline{\Omega}_{p_2}(\delta A, h2^{-n}),
	\]
	and so
	\begin{align*}
		\Omega_{p_2}\left(\mathscr R_{P_n}A - \mathscr R_{P_m} A,h\right)
		&\le \sum_{\ell = m}^{n-1} \Omega_{p_2}\left(\mathscr R_{P_{\ell+1}} A - \mathscr R_{P_\ell} A,h\right)
		\le 2\sum_{\ell = m}^{n-1} 2^\ell \oline{\Omega}_{p_2}\left( \delta A, h2^{-\ell}\right) \\
		&= 2h \sum_{\ell=m}^{n-1}  \frac{ \oline{\Omega}_{p_2}\left( \delta A, h2^{-\ell} \right) }{h 2^{-\ell}}
	\lesssim h\int_{0}^{h2^{-m+1}} \frac{\oline{\Omega}_{p_2}(\delta A,\tau)}{\tau} \frac{d\tau}{\tau}.
	\end{align*}
	
	{\it Case 2:} $0 < p_2 < 1$ and $q_2 > p_2$. We have $\Omega_{p_2}^{p_2} \left( \mathscr R_{P_{n+1}}A - \mathscr R_{P_n}A, h\right) \le 2^{n+1} \oline{\Omega}_{p_2}^{p_2}(\delta A, h2^{-n})$, and so
\begin{equation}\label{p2lessonetriangle}
	\Omega_{p_2}^{p_2}(\mathscr R_{P_n} A - \mathscr R_{P_m} A, h) 
	\le 2\sum_{\ell=m}^{n-1} 2^\ell \oline{\Omega}_{p_2}^{p_2}(\delta A, h 2^{-\ell}).
\end{equation}
We then argue similarly as in Case 1.

{\it Case 3:} $0 < q_2 < 1 \le p_2$. Taking the $q_2$ power on both sides of the inequality $\Omega_{p_2}(\mathscr R_{P_n} A - \mathscr R_{P_m} A,h) \le 2 \sum_{\ell = m}^{n-1} 2^{\ell} \oline{\Omega}_{p_2}(\delta A, h2^{-\ell})$, using that $0 < q_2 < 1$, and arguing similarly as in Case 1, we obtain
\begin{align*}
	\Omega_{p_2}(\mathscr R_{P_n} A - \mathscr R_{P_m}A, h)^{q_2} 
	&\le 2^{q_2} \sum_{\ell=m}^{n-1} 2^{\ell q_2} \oline{\Omega}_{p_2}^{q_2} (\delta A, h2^{-\ell})
	\lesssim_{q_2} h^{q_2} \int_0^{h 2^{-m+1}} \left( \frac{\oline{\Omega}_{p_2}(\delta A,\tau)}{\tau} \right)^{q_2}\frac{d\tau}{\tau}.
\end{align*}

{\it Case 4:} $0 < q_2 \le p_2 < 1$. We raise both sides of \eqref{p2lessonetriangle} to the power $q_2/p_2 \in (0,1]$, use sub-additivity, and proceed as in Case 3.
\end{proof}

\begin{proof}[Proof of Theorem \ref{T:Besovsewing}]
	{\it Step 1: convergence along dyadic partitions.} We begin by showing that the sequence $(\mathscr R_{P_n} A)_{n =0}^\oo$ is Cauchy in $\left( \mbb B^\gamma_{p_2,q_2}([0,T]), \mathrm d_{\mbb B^\gamma_{p_2,q_2}} \right)$, as defined in Definition \ref{D:BBesovspacemetric} (and therefore so is $(\mathscr I_{P_n} A)_{n = 0}^\oo$ in view of the identity $\mathscr I_{P_n} A - \mathscr I_{P_m} A = \mathscr R_{P_n} A - \mathscr R_{P_m} A$ for all $m,n \ge 0$). We do this by establishing the identity, for $0 \le m < n$,
	\begin{equation}\label{ineqinfourcases}
	\begin{split}
	&\left[ \int_0^{T} \left( \frac{ \Omega_{p_2}(\mathscr R_{P_n} A - \mathscr R_{P_m} A, h) }{h^\gamma} \right)^{q_2} \frac{d\tau}{\tau}\right]^{1/q_2}\\
	&\lesssim_{p_2,q_2,\gamma} 2^{-m(\gamma - (1 \vee 1/p_2))} \left[ \int_0^{2^{-m+1}T}  \left( \frac{\oline{\Omega}_{p_2}(\delta A,\tau) }{\tau^\gamma} \right)^{q_2} \frac{d\tau}{\tau}\right]^{1/q_2}.
	\end{split}
	\end{equation}
	
	{\it Case 1:} $1 \le p_2,q_2 \le \oo$. We divide both sides of the identity in Lemma \ref{L:RnCauchy} by $h^\gamma$ and take the $q_2$ power, whence, because $q_2 \ge 1$ and $\gamma > 1$, Jensen's inequality gives
\begin{align*}
	\int_0^{T} &\left( \frac{ \Omega_{p_2}\left( \mathscr R_{P_n}A - \mathscr R_{P_m} A,h\right)}{h^\gamma}\right)^{q_2} \frac{dh}{h}\\
	&\lesssim_{q_2}  2^{-mq_2(\gamma-1)} \int_0^{T} \left( \frac{2^{m(\gamma-1)}}{ h^{\gamma-1}} \int_0^{h2^{-m+1}} \tau^{\gamma-1} \frac{\oline{\Omega}_{p_2}(\delta A,\tau)}{\tau^{\gamma}} \frac{d\tau}{\tau} \right)^{q_2} \frac{dh}{h} \\
	&\lesssim_{q_2,\gamma}  2^{-m(q_2-1)(\gamma-1)} \int_0^{T} \frac{1}{h^{\gamma-1}} \int_0^{h2^{-m+1}} \tau^{\gamma-1}  \left( \frac{\oline{\Omega}_{p_2}(\delta A,\tau)}{\tau^{\gamma}}\right)^{q_2} \frac{d\tau}{\tau} \frac{dh}{h} \\
	&=  2^{-m(q_2-1)(\gamma-1)} \int_0^{2^{-m+1}T} \left[\tau^{\gamma-1} \int_{2^{m-1} \tau}^{T} \frac{1}{h^{\gamma-1}} \frac{dh}{h} \right] \left( \frac{\oline{\Omega}_{p_2}(\delta A,\tau)}{\tau^{\gamma}}\right)^{q_2} \frac{d\tau}{\tau}.
\end{align*}
The term in brackets satisfies
\[
	\tau^{\gamma-1} \int_{2^{m-1} \tau}^{T} \frac{1}{h^{\gamma-1}} \frac{dh}{h}
	= \frac{1}{\gamma - 1} \left( 2^{-(m-1)(\gamma-1)} - \tau^{\gamma-1} T^{\gamma - 1}\right)
	\le \frac{2^{\gamma-1}}{\gamma-1} 2^{-m(\gamma-1)},
\]
and so \eqref{ineqinfourcases} holds.

{\it Case 2:} $0 < p_2 < 1$, $q_2 \ge p_2$. We have $q_2/p_2 \ge 1$ and $\gamma > \frac{1}{p_2}$, and so, arguing as in Case 1 with Jensen's inequality and Fubini's theorem yields
\begin{align*}
	\int_0^{T} &\left(\frac{\Omega_{p_2}(\mathscr R_{P_n}A - \mathscr R_{P_m}A,h) }{h^\gamma}\right)^{q_2} \frac{dh}{h}\\
	&\lesssim_{q_2} 2^{-m(\gamma p_2 - 1)q_2/p_2} \int_0^{T} \left( \frac{1}{h^{\gamma p_2 - 1} 2^{-m(\gamma p_2 - 1)} } \int_0^{h 2^{-m+1}}\frac{ \oline{\Omega}_{p_2}^{p_2} (\delta A, \tau)}{\tau^{\gamma p_2} } \tau^{\gamma p_2 - 1} \frac{d\tau}{\tau }\right)^{q_2/p_2} \frac{dh}{h} \\
	&\lesssim_{p_2,q_2,\gamma} 2^{-m(\gamma p_2 - 1)(q_2-p_2)/p_2} \int_0^{T} \frac{1}{h^{\gamma p_2 -1}} \int_0^{h 2^{-m+1}} \left( \frac{ \oline{\Omega}_{p_2}(\delta A, \tau)}{\tau^{\gamma} }\right)^{q_2} \tau^{\gamma p_2 - 1} \frac{d\tau}{\tau} \frac{dh}{h} \\
	&\lesssim_{p_2,\gamma} 2^{-m(\gamma p_2 - 1) q_2/p_2} \int_0^{2^{-m+1}T} \left( \frac{ \oline{\Omega}_{p_2}(\delta A, \tau)}{\tau^{\gamma} }\right)^{q_2} \frac{d\tau}{\tau}.
\end{align*}

{\it Case 3:} $0 < q_2 < 1 \le p_2$. Using Lemma \ref{L:RnCauchy} and Fubini's theorem as in the previous cases,
\begin{align*}
	\int_0^{T}  \left( \frac{\Omega_{p_2}(\mathscr R_{P_n} A - \mathscr R_{P_m}A, h)}{h^\gamma} \right)^{q_2} \frac{dh}{h}
	&\lesssim_{q_2} \int_0^{T} \frac{1}{h^{q_2(\gamma - 1)}} \int_0^{h2^{-m+1}} \left( \frac{\oline{\Omega}_{p_2}(\delta A,\tau)}{\tau^\gamma} \right)^{q_2} \tau^{q_2(\gamma - 1)}\frac{d\tau}{\tau} \frac{dh}{h}\\
	&\lesssim_{q_2,\gamma} 2^{-m q_2(\gamma - 1)} \int_0^{2^{-m+1}T} \left( \frac{\oline{\Omega}_{p_2}(\delta A,\tau)}{\tau^\gamma} \right)^{q_2} \frac{d\tau}{\tau}.
\end{align*}

{\it Case 4:} $0 < q_2 \le p_2 < 1$. Arguing as in Case 3,
\begin{align*}
	\int_0^{T}& \left( \frac{ \Omega_{p_2}(\mathscr R_{P_n} A- \mathscr R_{P_m} A, h) }{h^\gamma} \right)^{q_2} \frac{d\tau}{\tau}\\
	&\lesssim_{p_2,q_2} \int_0^{T} \frac{1}{h^{q_2(\gamma - 1/p_2)} } \int_0^{h2^{-m+1}} \left( \frac{\oline{\Omega}_{p_2}(\delta A,\tau) }{\tau^\gamma} \right)^{q_2} \tau^{q_2(\gamma - 1/p_2)} \frac{d\tau}{\tau}\\
	&= \int_0^{2^{-m+1}T} \left[ \tau^{q_2(\gamma - 1/p_2)} \int_{2^{m-1} \tau}^{T} \frac{1}{h^{q_2(\gamma - 1/p_2)} } \frac{dh}{h} \right] \left( \frac{\oline{\Omega}_{p_2}(\delta A,\tau) }{\tau^\gamma} \right)^{q_2} \frac{d\tau}{\tau} \\
	&\lesssim_{p_2,q_2,\gamma} 2^{-mq_2(\gamma - 1/p_2)} \int_0^{2^{-m+1}T}  \left( \frac{\oline{\Omega}_{p_2}(\delta A,\tau) }{\tau^\gamma} \right)^{q_2} \frac{d\tau}{\tau}.
\end{align*}

We conclude in all cases that, as $n \to \oo$, $\mathscr R_{P_n} A$ has a limit $\mathscr RA \in \mbb B^\gamma_{p_2,q_2}([0,T])$. We deduce \eqref{Rbound} upon sending $n \to \oo$, letting $m = 1$, and noting the bound $\nor{\mathscr R^1 A}{\mbb B^\gamma_{p_2,q_2}} \le \nor{\delta A}{\oline{\mbb B}^\gamma_{p_2,q_2}}$.

{\it Step 2: additivity.} We next show that the two-parameter map $\tilde{\mathscr I} A := A + \mathscr R A = \lim_{n \to \oo} \mathscr I_{P_n} A$ is the increment of a path in $B^\alpha_{p_1 \wedge p_2, q_1 \vee q_2}$, that is, for some $\mathscr IA \in B^\alpha_{p_1 \wedge p_2, q_1 \vee q_2}$, $\tilde{\mathscr I}A = \delta \mathscr IA$.

For $\theta \in (0,1)$ and $A: \Delta_2(0,T) \to \RR^m$, define
\[
	\delta_\theta A_{st} := \delta A_{s, \theta t + (1-\theta) s, t} = A_{st} - A_{s,\theta t + (1-\theta) s} - A_{\theta t + (1-\theta ) s,t}.
\]
Fix $N = 1,2,\ldots$ and $K = 1,2,\ldots, 2^N - 1$, and set $\theta = \frac{K}{2^N}$ and $L := 2^N - K$. Then, for $h \in [0,T]$,
\begin{align*}
	\mathscr I_{P_n} A_{s,s+\theta h} &= \sum_{i=1}^{2^n} A_{s + \frac{(i-1)K}{2^{n+N}}h, s + \frac{iK}{2^{n+N}}h }, \quad  \mathscr I_{P_n} A_{s+\theta h,s+h} = \sum_{j=1}^{2^n} A_{s + \theta h + \frac{(j-1)L}{2^{n+N}}h, s + \theta h + \frac{jL}{2^{n+N}}h },
\quad \text{and}\\
	\mathscr I_{P_{n+N}} A_{s,s+h} &= \sum_{k=1}^{2^{n+N}} A_{s + \frac{k-1}{2^{n+N}}h, s + \frac{k}{2^{n+N}}h}.
\end{align*}
Therefore, $\mathscr I_{P_{n+N}} A_{s,s+h} - \mathscr I_{P_n} A_{s,s+\theta h} - \mathscr I_{P_n} A_{s+\theta h,s+h} = \I + \II$, where
\begin{align*}
	\I &:= -\sum_{i=1}^{2^n} \left( A_{s + \frac{(i-1)K}{2^{n+N}}h, s + \frac{iK}{2^{n+N}}h } - \sum_{\ell=1}^K A_{s+ \frac{iK - K + \ell-1}{2^{n+N}}h, s +\frac{iK - K + \ell}{2^{n+N}}h } \right)\\
	&= - \sum_{i=1}^{2^n} \sum_{\ell=1}^{K-1} \delta A_{s+ \frac{iK - K + \ell-1}{2^{n+N}}h, s +\frac{iK - K + \ell}{2^{n+N}}h , s + \frac{iK}{2^{n+N}}h }
\end{align*}
and
\begin{align*}
	\II &:= - \sum_{j=1}^{2^n} \left( A_{s + \theta h + \frac{(j-1)L}{2^{n+N}}h, s + \theta h + \frac{jL}{2^{n+N}}h } - \sum_{m=1}^{L} A_{s + \theta h + \frac{jL - L + m - 1}{2^{n+N}}h, s + \theta h +  \frac{jL -  + m}{2^{n+N}}h} \right)\\
	&= - \sum_{j=1}^{2^n} \sum_{m=1}^{L-1} \delta A_{s + \theta h + \frac{jL - L + m - 1}{2^{n+N}}h, s + \theta h + \frac{jL -  + m}{2^{n+N}}h, s + \theta h + \frac{jL}{2^{n+N}}h}.
\end{align*}
If $1 \le p_2 \le \oo$, then the triangle inequality gives
\begin{align*}
	&\nor{ \mathscr I_{P_{n+N}} A_{\cdot,\cdot+h} - \mathscr I_{P_n} A_{\cdot,\cdot+\theta h} - \mathscr I_{P_n} A_{\cdot+\theta h,\cdot+h}}{L^{p_2}([0,T-h])}\\
	&\le \sum_{i=1}^{2^n} \sum_{\ell=1}^{K-1} \nor{ \delta A_{\cdot+ \frac{iK - K + \ell-1}{2^{n+N}}h, \cdot +\frac{iK - K + \ell}{2^{n+N}}h , \cdot + \frac{iK}{2^{n+N}}h }}{L^{p_2}([0,T-h])} \\
	&+ \sum_{j=1}^{2^n} \sum_{m=1}^{L-1} \nor{\delta A_{\cdot + \theta h + \frac{jL - L + m - 1}{2^{n+N}}h, \cdot + \theta h + \frac{jL -  + m}{2^{n+N}}h, \cdot + \theta h + \frac{jL}{2^{n+N}}h} }{L^{p_2}([0,T-h])} \\
	&\le 2^{n+N} \oline{\Omega}_{p_2}(\delta A, 2^{-n} h),
\end{align*}
and we conclude that
\begin{equation}\label{additivebeforelimitpge1}
\begin{split}
	&\left[ \int_0^T \left( \frac{ \nor{ \mathscr I_{P_{n+N}} A_{\cdot,\cdot+h} - \mathscr I_{P_n} A_{\cdot,\cdot+\theta h} - \mathscr I_{P_n} A_{\cdot+\theta h,\cdot+h}}{L^{p_2}([0,T-h])}}{h^\gamma} \right)^{q_2} \frac{dh}{h} \right]^{1/q_2}\\
	&\le 2^{n+N} \left[ \int_0^T \left( \frac{ \oline{\Omega}_{p_2}(\delta A, 2^{-n} h)}{h^\gamma} \right)^{q_2} \frac{dh}{h} \right]^{1/q_2} = 2^N 2^{-n(\gamma-1)} \left[ \int_0^{2^{-n}T} \left( \frac{ \oline{\Omega}_{p_2}(\delta A, h)}{h^\gamma} \right)^{q_2}\frac{dh}{h} \right]^{1/q_2}.
\end{split}
\end{equation}
Otherwise, if $1/\gamma < p_2 < 1$, then the sub-additivity of $\nor{\cdot}{L^{p_2}([0,T-h])}^{p_2}$ gives
\[
	\nor{ \mathscr I_{P_{n+N}} A_{\cdot,\cdot+h} - \mathscr I_{P_n} A_{\cdot,\cdot+\theta h} - \mathscr I_{P_n} A_{\cdot+\theta h,\cdot+h}}{L^{p_2}([0,T-h])}^{p_2} 
	\le 2^{n+N} \oline{\Omega}_{p_2}^{p_2}(\delta A, 2^{-n} h),
\]
and thus
\begin{equation}\label{additivebeforelimitpless1}
\begin{split}
	&\left[ \int_0^T \left( \frac{ \nor{ \mathscr I_{P_{n+N}} A_{\cdot,\cdot+h} - \mathscr I_{P_n} A_{\cdot,\cdot+\theta h} - \mathscr I_{P_n} A_{\cdot+\theta h,\cdot+h}}{L^{p_2}([0,T-h])}}{h^\gamma} \right)^{q_2} \frac{dh}{h} \right]^{1/q_2}\\
	&\le 2^N 2^{-n(\gamma-1/p_2)} \left[ \int_0^{2^{-n}T} \left( \frac{ \oline{\Omega}_{p_2}(\delta A, h)}{h^\gamma} \right)^{q_2}\frac{dh}{h} \right]^{1/q_2}.
\end{split}
\end{equation}
In either case, taking $n \to \oo$ and invoking Step 1 gives $\nor{\delta_\theta (\tilde{\mathscr I}A)}{\mbb B^\gamma_{p_2,q_2}([0,T])} = 0$ for all dyadic $\theta$. 

If $p_2 < \oo$, then, for any $h$, $\theta \mapsto \nor{\delta_\theta (\tilde{\mathscr I}A)_{\cdot,\cdot +h}}{L^{p_2}([0,T-h]}$ is continuous, and therefore equal to $0$ for all $\theta \in (0,1)$. In particular, $[\tilde{\mathscr I}A]_{\oline{\mbb B}^\gamma_{p_2,q_2}} = 0$, and we conclude from Lemma \ref{L:additive} that there exists $\mathscr IA$ such that, for almost every $(s,t) \in \Delta_2(0,T)$, $\mathscr IA_t - \mathscr IA_s = \tilde{\mathscr I}A_{st}$. If $p _2 = \oo$, then we have $\nor{\delta_\theta (\tilde{\mathscr I}A)}{\mbb B^\gamma_{r,q_2}([0,T])} = 0$ for all dyadic $\theta$ and $r < p_2$, and we argue as above.

The bound \eqref{Ibound} now follows upon invoking \eqref{Rbound} and taking the $L^{q_1 \vee q_2}(d\tau/\tau)$-(quasi)-norm on both sides of
\[
	\frac{\omega_{p_1 \wedge p_2}(\mathscr IA,\tau)}{\tau^\alpha} \lesssim_{T,p_1,p_2} \frac{\Omega_{p_1}(A,\tau)}{\tau^\alpha} + \tau^{\gamma - \alpha} \frac{\Omega_{p_2}(\mathscr RA,\tau)}{\tau^\gamma}.
\]

{\it Step 3: convergence along arbitrary partitions.} We finish by proving \eqref{generalconvergence}. Below, we always take $r > 1/\gamma$ and $r \ge q_2$.

Let $P := \left\{ 0 = \tau_0 < \tau_1 < \cdots < \tau_N = 1 \right\}$ be an arbitrary partition of $[0,1]$, and, for $i = 1,2,\ldots, N$, define $\delta_i := \tau_i - \tau_{i-1}$. Fix $h \in [0,T]$ and $s \in [0,T-h]$. Then
\[
	\delta{\mathscr I}A_{s,s+h} - \mathscr I_P A_{s,s+h} = \sum_{i=1}^N \left[ \delta{\mathscr I}A_{s + \tau_{i-1}h,s + \tau_i h} - A_{s + \tau_{i-1}h, s + \tau_i h} \right] = \sum_{i=1}^N \mathscr R A_{s + \tau_{i-1}h, s + \tau_i h},
\]
and so
\begin{equation}\label{additivityuse}
	\Omega_{p_2}(\delta \mathscr IA - \mathscr I_P A, h)^{1 \wedge p_2} \le \sum_{i=1}^N \Omega_{p_2}(\mathscr RA, \delta_i h)^{1 \wedge p_2}.
\end{equation}

Assume first that $1 \le p_2 \le \oo$. If $r \ge 1$, then Minkowski's inequality yields
\begin{align*}
	\left[ \int_0^T \left( \frac{\Omega_{p_2}(\delta\mathscr I A - \mathscr I_P A, h)}{h^\gamma}\right)^{r} \frac{dh}{h} \right]^{1/r}
	&\le \sum_{i=1}^N \left[ \int_0^T \left( \frac{ \Omega_{p_2}(\mathscr RA, \delta_i h)}{h^\gamma}\right)^{r} \frac{dh}{h} \right]^{1/r}\\
	&\le \sum_{i=1}^N \delta_i^\gamma\left[ \int_0^{\delta_iT} \left( \frac{ \Omega_{p_2}(\mathscr RA, h)}{h^\gamma}\right)^{r} \frac{dh}{h} \right]^{1/r}\\
	&\lesssim_{p_1,q_1,p_2,q_2,r,\gamma} T\norm{P}^{\gamma-1} \nor{\delta A}{\oline{\mbb B}^\gamma_{p_2,q_2}}.
\end{align*}
and otherwise, if $1/\gamma < r < 1$, then
\begin{align*}
	\int_0^T \left( \frac{\Omega_{p_2}(\delta\mathscr I A - \mathscr I_P A, h)}{h^\gamma}\right)^{r} \frac{dh}{h}
	&\le \sum_{i=1}^N \int_0^T \left( \frac{ \Omega_{p_2}(\mathscr RA, \delta_i h)}{h^\gamma}\right)^{r} \frac{dh}{h}\\
	&\lesssim_{p_1,q_1,p_2,q_2,r,\gamma} T\norm{P}^{\gamma r -1} \nor{\delta A}{\oline{\mbb B}^\gamma_{p_2,q_2}}^{q_2}
	\end{align*}

Assume next that $1/\gamma < p_2 < 1$. If $1/\gamma < r \le p_2 < 1$, then
\begin{align*}
	\int_0^T \left( \frac{\Omega_{p_2}(\delta \mathscr I A - \mathscr I_P A, h)}{h^\gamma}\right)^{r} \frac{dh}{h}
	&\le \int_0^T \left( \frac{ \sum_{i=1}^N\Omega_{p_2}(\mathscr RA, \delta_i h)^{p_2}}{h^{\gamma p_2}}\right)^{r/p_2} \frac{dh}{h}\\
	&\le \sum_{i=1}^N \int_0^T \left( \frac{ \Omega_{p_2}(\mathscr RA, \delta_i h)^{p_2}}{h^{\gamma}}\right)^{r} \frac{dh}{h} \\
	&\lesssim_{p_1,q_1,p_2,q_2,r,\gamma} T\norm{P}^{\gamma r -1}\nor{\delta A}{\oline{\mbb B}^\gamma_{p_2,q_2}}^{q_2}.
\end{align*}
Otherwise, if $1/\gamma < p_2 < 1$ and $r \ge p_2$, then, by Minkowski's inequality with the exponent $r/p_2$,
\begin{align*}
	\left[\int_0^T \left( \frac{\Omega_{p_2}(\delta\mathscr I A - \mathscr I_P A, h)}{h^\gamma}\right)^{r} \frac{dh}{h} \right]^{p_2/r} 
	&\le \left[\int_0^T \left( \frac{\sum_{i=1}^N \Omega_{p_2}(\mathscr R A ,\delta_i h)^{p_2} }{h^{\gamma p_2}}\right)^{r/p_2} \frac{dh}{h} \right]^{p_2/r} \\
	&\le \sum_{i=1}^N \left[\int_0^T \left( \frac{ \Omega_{p_2}(\mathscr R A ,\delta_i h) }{h^\gamma}\right)^{r} \frac{dh}{h} \right]^{p_2/r} \\
	&\lesssim_{p_1,q_1,p_2,q_2,r,\gamma} T \norm{P}^{\gamma p_2 - 1}\nor{\delta A}{\oline{\mbb B}^\gamma_{p_2,q_2}}^{p_2}.
\end{align*}
In all cases, we conclude upon sending $\norm{P} \to 0$.
\end{proof}

\begin{proof}[Proof of Theorem \ref{T:Besovsewingendpoint}]
{\it Step 1: convergence along dyadic partitions.} Fix $r \in [q_2,\oo]$ and $1 \le m < n$. The third case of Lemma \ref{L:RnCauchy} and Minkowski's inequality yield, for all $T' \in [0,T]$,
\begin{align*}
	\left[\int_0^{T'} \left( \frac{\Omega_{p_2}(\mathscr R_{P_n}A - \mathscr R_{P_m }A ,h) }{\omega_r(h)} \right)^{r} \frac{dh}{h}\right]^{1/r}
	&\lesssim_{p_2,q_2} \left( \int_0^{T'} \left( \frac{1}{\ell_r(h)} \right)^r \left[ \int_0^{h2^{-m+1}} \left( \frac{\oline{\Omega}_{p_2}(\delta A,\tau)}{\tau^{1 \vee \frac{1}{p_2}} } \right)^{q_2} \frac{d\tau}{\tau}\right]^{r/q_2} \frac{dh}{h} \right)^{1/r}\\
	&\le \left( \int_0^{2^{-m+1} T'} \left[ \int_{2^{m-1} \tau}^{T'}  \left( \frac{1}{\ell_r(h)} \right)^r \frac{dh}{h} \right]^{q_2/r} \left( \frac{\oline{\Omega}_{p_2}(\delta A,\tau)}{\tau^{1 \vee \frac{1}{p_2}} } \right)^{q_2} \frac{d\tau}{\tau} \right)^{1/q_2}\\
	&\lesssim_{\ell_r} \left[\int_0^{2^{-m+1} T'} \left( \frac{\oline{\Omega}_{p_2}(\delta A,\tau)}{\tau^{1 \vee \frac{1}{p_2}} } \right)^{q_2} \frac{d\tau}{\tau}\right]^{1/q_2}.
\end{align*}
We conclude as in the proof of Theorem \ref{T:Besovsewing} that $(\mathscr R_{P_n} A)_{n =0}^\oo \subset \mbb B^{\omega_r}_{p_2,r}([0,T])$ is Cauchy, and thus has a limit $\mathscr R A \in \mbb B^{\omega_{q_2}}_{p_2,q_2}([0,T])$ that satisfies, for all $T' \in [0,T]$,
\begin{equation}\label{RomegaboundarbitraryT}
	\left[\int_0^{T'} \left( \frac{\Omega_{p_2}(\mathscr R A ,h) }{\omega_r(h)} \right)^{r} \frac{dh}{h}\right]^{1/r}
	\lesssim_{p_2,q_2,\ell_r}
	 \left[\int_0^{T'} \left( \frac{\oline{\Omega}_{p_2}(\delta A,\tau)}{\tau^{1 \vee \frac{1}{p_2}} } \right)^{q_2} \frac{d\tau}{\tau}\right]^{1/q_2},
\end{equation}
and thus in particular \eqref{Romegabound}. 

{\it Step 2: additivity.} We argue as in Step 2 of Theorem \ref{T:Besovsewing}: given $\theta = \frac{K}{2^N}$ for some $N = 1,2,\ldots$ and $K = 1,2,\ldots, 2^N-1$, we deduce as in \eqref{additivebeforelimitpge1} and \eqref{additivebeforelimitpless1} that
\begin{align*}
	&\left[ \int_0^T \left( \frac{ \nor{ \mathscr I_{P_{n+N}} A_{\cdot,\cdot+h} - \mathscr I_{P_n} A_{\cdot,\cdot+\theta h} - \mathscr I_{P_n} A_{\cdot+\theta h,\cdot+h}}{L^{p_2}([0,T-h])}}{h^{1 \vee \frac{1}{p_2}}} \right)^{q_2} \frac{dh}{h} \right]^{1/q_2}\\
	&\le 2^N  \left[ \int_0^{2^{-n}T} \left( \frac{ \oline{\Omega}_{p_2}(\delta A, h)}{h^{1 \vee \frac{1}{p_2}} } \right)^{q_2}\frac{dh}{h} \right]^{1/q_2}.
\end{align*}
Taking $n \to \oo$ gives $\Omega_{p_2}(\delta_\theta A,h) = 0$ for all $h \in [0,T]$ and dyadic $\theta \in (0,1)$. The conclusion that $\tilde{ \mathscr I}A = \delta \mathscr IA$ and the bound \eqref{Iomegabound} then follow exactly as in Theorem \ref{T:Besovsewing}.

{\it Step 3: convergence along arbitrary partitions.} As in the proof of Theorem \ref{T:Besovsewing}, let $P = \{0 = \tau_0 < \tau_1 < \cdots < \tau_N = 1\}$ be a partition of $[0,1]$ and write $\delta_i := \tau_i - \tau_{i-1}$ for $i = 1,2,\ldots, N$. Then \eqref{additivityuse} and \eqref{RomegaboundarbitraryT} with $r = \oo$ give
\begin{align*}
	\nor{\delta \mathscr I A - \mathscr I_P A }{\mbb B^{1 \vee \frac{1}{p_2}}_{p_2,\oo} }^{1 \wedge p_2}
	&\le \sum_{i=1}^N \sup_{0 \le h \le T} \frac{\Omega_{p_2}(\mathscr RA, \delta_i h)^{1 \wedge p_2}}{h} \\
	&\le T \sup_{0 \le h \le \norm{P} T} \frac{\Omega_{p_2}(\mathscr RA, h)^{1 \wedge p_2}}{h} 			\lesssim_{p_2,q_2} T \left[ \int_0^{ \norm{P} T} \left( \frac{ \oline{\Omega}_{p_2}(\delta A, h)}{ h^{1 \vee \frac{1}{p_2}}} \right)^{q_2} \frac{dh}{h} \right]^{\frac{1 \wedge p_2}{q_2}},
\end{align*}
and we conclude upon sending $\norm{P} \to 0$.
\end{proof}

\begin{remark}\label{R:RAsmallT}
	For $\gamma > 0$ and $0 < p \le \oo$, we introduce the following closed subspace of $\mbb B^\gamma_{p,\oo}$:
	\[
		\mbb B^\gamma_{p,\oo; \circ} ([0,T]) := \left\{ A \in \mbb B^\gamma_{p,\oo}([0,T]) : \lim_{\tau \to 0^+} \sup_{0 < h \le \tau} \frac{ \Omega_p(A,\tau)}{\tau^\gamma} = 0 \right\}.
	\]
	Then the bound \eqref{RomegaboundarbitraryT} implies that, with $r = \oo$ and $\ell_\oo \equiv 1$, in the setting of Theorem \ref{T:Besovsewingendpoint}, we in fact have $\mathscr RA \in \mbb B^{1 \vee \frac{1}{p_2}}_{p_2,\oo; \circ}([0,T])$. We will use this observation later to make sense of some rough differential equations in the Davie sense (see Proposition \ref{P:Davies}).
\end{remark}

\begin{proof}[Proof of Theorem \ref{T:Besovsewingcontinuous}]
In view of the relation $\delta \mathscr RA = \delta \left(\delta \mathscr I A - A\right) = - \delta A$, we see that \eqref{sewingdeltaAHolder} and \eqref{sewingdeltaAHolderendpoint} in respectively parts (a) and (b) are satisfied with $\mathscr RA$ in place of $A$.

In part (a), we have $\gamma > 1/p_2$. We may then apply Proposition \ref{P:BBesovembedding} with $\omega(r) = r^\gamma$, which, combined with \eqref{Rbound}, immediately yields \eqref{RHolder}. We now apply the interpolation estimate Lemma \ref{L:interpolate} with $p = p_2$, $r = p_1$, and $\delta = \gamma - 1/p_2$, for which the hypotheses are satisfied because of the condition on $\gamma$ and $\alpha$, to obtain
\begin{align*}
	\nor{\mathscr RA}{\mbb B^\alpha_{p_1,q_2}([0,T])} 
	&\lesssim_{\alpha,\gamma,p_1,p_2,q_2} T^{\gamma - \frac{1}{p_2} - \alpha + \frac{1}{p_1} } \nor{\mathscr RA}{\mbb C^{\gamma - \frac{1}{p_2}}([0,T])}^{1 - \frac{p_2}{p_1}}\nor{\mathscr RA}{\mbb B^\gamma_{p_2,q_2}([0,T])}^{\frac{p_2}{p_1}}\\
	&\lesssim_{p_2,q_2,\gamma,\theta} T^{\gamma - \frac{1}{p_2} - \alpha + \frac{1}{p_1}} \left( \nor{\delta A}{\oline{\mbb B}^\gamma_{p_2,q_2}([0,T])} + M \right)^{1 - \frac{p_2}{p_1}}\nor{\delta A}{\oline{\mbb B}^\gamma_{p_2,q_2}([0,T])}^{\frac{p_2}{p_1}}\\
	&\le T^{\gamma - \frac{1}{p_2} - \alpha + \frac{1}{p_1} } \left( \nor{\delta A}{\oline{\mbb B}^\gamma_{p_2,q_2}([0,T])} + M \right).
\end{align*}
This concludes part (a) in view of the estimate
\begin{equation}\label{IAbetterbound}
	[\mathscr IA]_{ B^\alpha_{p_1,q_1 \vee q_2}([0,T])} \lesssim_{q_1,q_2}
	\nor{A}{B^\alpha_{p_1,q_1}([0,T])} + \nor{\mathscr RA}{\mbb B^\alpha_{p_1,q_2}([0,T])}.
\end{equation}

In part (b), the estimate \eqref{Rctsendpoint} follows as in part (a) from Proposition \ref{P:BBesovembedding} and \eqref{Romegabound}. Setting
\[
	\omega = \omega_{q_2}, \quad \rho(\tau) = \frac{\omega_{q_2}(\tau)}{\tau^{1/p_2}}, \quad p = p_2, \quad r = p_1, \quad \text{and} \quad q = q_2
\]
gives, with the notation of Lemma \ref{L:interpolatemodulus},
\[
	\sigma(T) = \left[ \int_0^T \tau^{\frac{p_1q_2}{p_1 - p_2} \left( 1 - \frac{1}{p_2} - \alpha + \frac{1}{p_1} \right) } \ell_{q_2}(\tau)^{\frac{p_1q_2}{p_1 - p_2}} \frac{d\tau}{\tau} \right]^{\frac{p_1 - p_2}{p_1 q_2} }\lesssim_{p_1,p_2,q_1,q_2, \alpha} T^{1 - \frac{1}{p_2} - \alpha + \frac{1}{p_1} } \ell_{q_2}(T),
\]
where we have used the final condition in \eqref{loglike}. Applying Lemma \ref{L:interpolatemodulus} to $\mathscr RA$ yields
\begin{align*}
	\nor{\mathscr RA}{\mbb B^\alpha_{p_1,q_2}} 
	&\lesssim_{\alpha, p_1,p_2,q_2} T^{1 - \frac{1}{p_2} - \alpha + \frac{1}{p_1} } \ell_{q_2}(T)\left( \sup_{0 \le s < t \le T} \frac{|\mathscr RA_{st}|}{\omega_{q_2}(t-s)(t-s)^{-1/p_2} } \right)^{1 - \frac{p_2}{p_1}} \nor{\mathscr RA}{\mbb B^\omega_{p_2,q_2}}^{\frac{p_2}{p_1}}\\
	&\lesssim_{\alpha,p_2,q_2,K,\theta} T^{1 - \frac{1}{p_2} - \alpha + \frac{1}{p_1} } \ell_{q_2}(T)\left( \nor{\delta A}{\oline{B}^1_{p_2,q_2}} + M \right)^{1 - \frac{p_2}{p_1}}\nor{\delta A}{\oline{\mbb B}^1_{p_2,q_2}}^{\frac{p_2}{p_1}}\\
	&\le T^{1 - \frac{1}{p_2} - \alpha + \frac{1}{p_1} } \ell_{q_2}(T)\left( \nor{\delta A}{\oline{B}^1_{p_2,q_2}} + M \right).
\end{align*}
We conclude in conjunction with \eqref{IAbetterbound}.
\end{proof}

\section{The Young regime}

Theorems \ref{T:Besovsewing}, \ref{T:Besovsewingendpoint}, and \ref{T:Besovsewingcontinuous} are used to obtain results on integrating Besov-regular paths. This recovers some existing results and also provides some refinements. The results are related to the question of making sense of the product of two Besov distributions. These integration results and the precise estimates that accompany them are used to solve differential equations with Besov driving signals in the Young regime. 

\subsection{Young integration} \label{sec:YI}

Given two measurable real-valued (the finite-dimensional vector-valued setting can be recovered component by component) $f: [0,T] \to \RR$ and $g: [0,T] \to \RR$, we define, for $(s,t) \in \Delta_2(0,T)$ and a partition $P = \{0 = \tau_0 < \tau_1 < \cdots < \tau_N = 1\}$ of $[0,1]$,
\[
	\mathscr I^P_{st} = \sum_{k=1}^N f_{s + \tau_{k-1}(t-s)} \delta g_{s + \tau_{k-1} (t-s), s + \tau_k (t-s)}.
\]
We recall the definition of the families $(\ell_r)_{0 < r \le \oo}$ in \eqref{loglike} and $(\omega_r)_{0 < r \le \oo}$ in \eqref{omegar}.


\begin{theorem}\label{T:YoungBesov}
	Let $\alpha_0,\alpha_1 \in (0,1)$ and $0 < p_0,p_1,q_0,q_1 \le \oo$, define
	\[
		\gamma = \alpha_0 + \alpha_1, \quad \frac{1}{p_2} = \frac{1}{p_0} + \frac{1}{p_1}, \quad \text{and} \quad \frac{1}{q_2} = \frac{1}{q_0} + \frac{1}{q_1},
	\]
	and assume that $f \in B^{\alpha_0}_{p_0,q_0}([0,T])$ and $g \in B^{\alpha_1}_{p_1,q_1}([0,T])$.
	
	\begin{enumerate}[(a)]
	\item If $\gamma > 1 \vee \frac{1}{p_2}$, then there exists $\mathscr I \in B^{\alpha_1}_{p_2,q_1}([0,T])$, which we write as
	\[
		\mathscr I_t := \int_0^t f_s dg_s \quad \text{for } t \in [0,T],
	\]
	such that, for all $r \in (0,\oo)$ such that $r > 1/\gamma$ and $r \ge q_2$ (note that if $q_2 > 1/\gamma$, taking $r = q_2$ is allowed),
	\begin{equation}\label{Younggeneralconvergence}
		\lim_{\norm{P} \to 0} \nor{\mathscr I^{P} - \delta\mathscr I}{\mbb B^\gamma_{p_2,r}([0,T])} = 0.
	\end{equation}
	Moreover,
	\begin{equation}\label{Youngerrorbound}
		\nor{\delta \mathscr I - f \delta g }{\mbb B^\gamma_{p_2,q_2}([0,T])} \lesssim_{\gamma,p_2,q_2} [f]_{B^{\alpha_0}_{p_0,q_0}} [g]_{B^{\alpha_1}_{p_1,q_1}},
	\end{equation}
	and
	\begin{equation}\label{Youngintegralbound}
		\left[\mathscr I \right]_{B^{\alpha_1}_{p_2,q_1}} \lesssim_{\alpha_0,\alpha_1,p_0,p_1,q_0,q_1} \left( \nor{f}{L^{p_0}([0,T])} + T^{\alpha_0} [f]_{B^{\alpha_0}_{p_0,q_0}} \right) [g]_{B^{\alpha_1}_{p_1,q_1}}.
	\end{equation}
	If, in addition, $\alpha_0 > \frac{1}{p_0}$ and $\alpha_1 > \frac{1}{p_1}$, then
	\begin{equation}\label{betterYoungintegralbound}
		\left[\mathscr I\right]_{B^{\alpha_1}_{p_1,q_1}} \lesssim_{\alpha_0,\alpha_1,p_0,p_1,q_0,q_1} \left( |f(0)| + T^{\alpha_0 - \frac{1}{p_0}} [f]_{B^{\alpha_0}_{p_0,q_0}} \right) [g]_{B^{\alpha_1}_{p_1,q_1}}.
	\end{equation}
	
	\item If $\gamma = 1 \vee \frac{1}{p_2} \le \frac{1}{q_2}$, then
	\begin{equation}\label{Youngendpointconvergence}
		\lim_{\norm{P} \to 0} \nor{\mathscr I^{P} - \delta\mathscr I}{\mbb B^{1 \vee \frac{1}{p_2} }_{p_2,\oo}([0,T])} = 0,
	\end{equation}
	\begin{equation}\label{Youngerrorboundendpoint}
		\nor{\delta \mathscr I - f \delta g }{\mbb B^{\omega_r}_{p_2,r}([0,T])} \lesssim_{\kappa,p_2,r} [f]_{B^{\alpha_0}_{p_0,q_0}} [g]_{B^{\alpha_1}_{p_1,q_1}} \quad \text{for all } r \in [q_2,\oo],
	\end{equation}
	and
	\begin{equation}\label{Youngintegralboundendpoint}
		\left[\mathscr I \right]_{B^{\alpha_1}_{p_2,q_1}} \lesssim_{\alpha_0,\alpha_1,p_0,p_1,q_0,q_1,\kappa} \left( \nor{f}{L^{p_0}([0,T])} + T^{\alpha_0} \ell_{q_2}(T)  [f]_{B^{\alpha_0}_{p_0,q_0}} \right) [g]_{B^{\alpha_1}_{p_1,q_1}}.
	\end{equation}
	If, in addition, $\alpha_0 > \frac{1}{p_0}$ and $\alpha_1 > \frac{1}{p_1}$ (note that this implies $\gamma = 1$), then
	\begin{equation}\label{betterYoungintegralboundendpoint}
		\left[\mathscr I\right]_{B^{\alpha_1}_{p_1,q_1}} \lesssim_{\alpha_0,\alpha_1,p_0,p_1,q_0,q_1} \left( |f(0)| + T^{\alpha_0 - \frac{1}{p_0} }(1 + \ell_{q_2}(T)) [f]_{B^{\alpha_0}_{p_0,q_0}} \right) [g]_{B^{\alpha_1}_{p_1,q_1}}.
	\end{equation}
	\end{enumerate}
\end{theorem}

\begin{remark}\label{R:Youngintegrals}
	The requirement throughout Theorem \ref{T:YoungBesov} that $\gamma \ge 1/p_2$ puts restrictions on the types of discontinuities that $f$ or $g$ may have. For instance, if $\alpha_0 \le 1/p_0$, which means that $f$ can be discontinuous, then $\alpha_1 > 1/p_1$, so that Proposition \ref{P:Besovembedding} gives $g \in C^{\alpha_1 - 1/p_1}$ (and vice versa). Of course, the bounds \eqref{betterYoungintegralbound} and \eqref{betterYoungintegralboundendpoint} require that both $f$ and $g$ are H\"older continuous, albeit with \emph{less} H\"older regularity than is required in the pure-H\"older Young regime.
	
	According to part (b), in the critical case, we may take $f \in B^{1/p_0}_{p_0,q_0}$ and $g \in B^{1/p_1}_{p_1,q_1}$ with
	\[
		1 \le \frac{1}{p_0} + \frac{1}{p_1} \le \frac{1}{q_0} + \frac{1}{q_1}.
	\]
	This implies that one of $q_0$ or $q_1$ is finite, and, indeed, if \emph{both} are finite, then both $f$ and $g$ are allowed to have discontinuities, or even be nowhere locally bounded; see for instance Proposition \ref{P:loglog} or Corollary \ref{C:loglog}. However, in view of Proposition \ref{P:nojumps}, they are not allowed to both have \emph{jump} discontinuities. Moreover, as soon as one of $q_0$ or $q_1$ is $\oo$, the other must be no greater than $1$, which means that, if one of $f$ or $g$ has jump discontinuities, the other must be continuous.
	
	These observations are in line with the theory of rough paths with jumps put forward in \cite{friz2018differential}. Indeed, the Heaviside function
	\[
		H_t =
		\begin{cases}
			0 & \text{if } 0 \le t < 1/2,\\
			1 & \text{if } 1/2 \le t \le 1
		\end{cases}
	\]
	satisfies $\nor{H}{B^{1/p}_{p,\oo}} = 1$ for all $0 < p < \oo$, and the integral $\int_0^1 f_t dH_t$ cannot be defined as limits of arbitrary Riemann sums if $f$ itself has a jump discontinuity at $t = 1/2$ (formally, $dH_t = \delta_{1/2}(t)$).
\end{remark}

\begin{remark}\label{R:noBrownian}
	It has been known since \cite{lyonspathintegrals} that the continuous bilinear map defining $\mathscr I$ does not extend from $C^1 \times C^1$ to spaces on which the Wiener measure is supported. This is consistent with Theorem \ref{T:YoungBesov}; indeed, if $q < \oo$, then, with probability one, Brownian paths do not belong to $B^{1/2}_{p,q}$ for any $p$.
\end{remark}

\begin{proof}[Proof of Theorem \ref{T:YoungBesov}]
The map $A: \Delta_2(0,T) \to \RR$ defined by $A_{st} := f_s(g_t - g_s)$ satisfies $\delta A_{rst} := -\delta f_{rs} \delta g_{st}$. H\"older's inequality yields
\[
	\nor{\delta A}{\oline{\mbb B}^\gamma_{p_2, q_2}} \lesssim [f]_{B^{\alpha_0}_{p_0,q_0}} [g]_{B^{\alpha_1}_{p_1,q_1}} \quad \text{and} \quad \nor{A}{\mbb B^{\alpha_1}_{p_2,q_1}} \le \nor{f}{L^{p_0}([0,T])} [g]_{B^{\alpha_1}_{p_1,q_1}}.
\]
With the notation from Section \ref{sec:BesovSew}, we have $\mathscr I^P = \mathscr I_P A$ and $\mathscr I = \mathscr IA$. The convergence statements \eqref{Younggeneralconvergence} and \eqref{Youngendpointconvergence} and the bounds \eqref{Youngerrorbound}, \eqref{Youngintegralbound}, and \eqref{Youngerrorboundendpoint} follow immediately from Theorems \ref{T:Besovsewing} and \ref{T:Besovsewingendpoint}.

If $\alpha_i> 1/p_i$ for $i = 0,1$, then Proposition \ref{P:Besovembedding} gives, for all $0 \le s \le u \le t \le T$,
\begin{align*}
	\abs{ \delta A_{sut}} &\le \abs{ \delta f_{su} } \abs{ \delta g_{ut}} 
	\le [f]_{C^{\alpha_0 - 1/p_0}} [g]_{C^{\alpha_1 - 1/p_1}} (u-s)^{\alpha_0 -1/p_0} (t-u)^{\alpha_1 - 1/p_1}\\
	&\lesssim_{\alpha_0,\alpha_1,p_0,p_1,q_0,q_1}
	[f]_{B^{\alpha_0}_{p_0,q_0}} [g]_{B^{\alpha_1}_{p_1,q_1}} (u-s)^{\alpha_0 -1/p_0} (t-u)^{\alpha_1 - 1/p_1}
\end{align*}
and
\[
	\nor{A}{\mbb B^{\alpha_1}_{p_1,q_1}} \le \nor{f}{\oo} [g]_{B^{\alpha_1}_{p_1,q_1}} \lesssim_{\alpha_0,p_0,q_0} \left(|f(0)| + T^{\alpha_0 - 1/p_0} [f]_{B^{\alpha_0}_{p_0,q_0}}\right) [g]_{B^{\alpha_1}_{p_1,q_1}}.
\]
The bounds \eqref{betterYoungintegralbound} and \eqref{betterYoungintegralboundendpoint} are then consequences of Theorem \ref{T:Besovsewingcontinuous}, respectively parts (a) and (b).
\end{proof}

\subsection{Differential equations: Young regime} 
For $f: \RR^m \to \RR^m \otimes \RR^n$, $X:[0,T] \to \RR^m$, and $y \in \RR^m$, we consider the initial value problem\begin{equation}\label{E:Young}
	dY_t = f(Y_t)\cdot dX_t \quad \text{in } [0,T] \quad \text{and} \quad Y_0 = y
\end{equation}
in the Young regime, which, in the Besov context, means
\begin{equation}\label{Youngregime}
	\left\{
	\begin{split}
	&\frac{1}{2} \le \alpha < 1, \quad \frac{1}{\alpha} < p \le \oo, \\
	&0 < q \le \oo \quad \text{if } \alpha > \frac{1}{2}, \text{ and}\\
	&0 < q \le 2 \quad \text{if } \alpha = 1/2,
	\end{split}
	\right.
\end{equation}
with the nonlinearity satisfying
\begin{equation}\label{Youngf}
	\begin{dcases}
	f \in C^{1,\delta}(\RR^m;\RR^m \otimes \RR^n) & \text{if } \alpha > \frac{1}{2}, \text{ where } (1+\delta)\alpha > 1 \text{ and } \delta \alpha > 1/p,\\
	f \in C^2(\RR^n;\RR^m \otimes \RR^n) & \text{if } \alpha = \frac{1}{2}.
	\end{dcases}
\end{equation}

\begin{theorem}\label{T:Youngeq}
	Assume \eqref{Youngregime}, and fix $X \in B^\alpha_{pq}([0,T]; \RR^n)$ and $f$ satisfying \eqref{Youngf}. Then, for every fixed $y \in \RR^m$, there exists a unique solution $Y \in B^\alpha_{pq}([0,T];\RR^m)$ of \eqref{E:Young}. Moreover, there exists a constant $M$ depending only on $\alpha$, $p$, $q$, $T$, $\nor{f}{C^1}$, and $[X]_{B^\alpha_{pq}}$ such that $[Y]_{B^\alpha_{pq}} \le M$.
\end{theorem}

\begin{remark}\label{R:variationembedding?}
	As was mentioned in the introduction, a possible shortcut to solving \eqref{E:Young} is to use embeddings to reduce the problem to the variation setting. However, the power of the Besov sewing results of Section \ref{sec:BesovSew} is that one can read off all the estimates on the Besov scale, which not only immediately yields the contraction property for the fixed-point map in the proof of Theorem \ref{T:Youngeq}, but also leads to the local Lipschitz continuity of the solution map in the Besov metric (see Theorem \ref{T:YoungItoMap} below), something that could not be accomplished in, for example, \cite{liu2020sobolev}.
	
	Moreover, upon embedding Besov- into variation-spaces, some information may actually be lost. For example, in the borderline regime of \eqref{Youngregime}, that is, when $X \in B^{1/2}_{pq}$ with $q \le 2 < p$, the known variation embeddings give $X \in \mcl V^2$ \cite[Proposition 4.1(3)]{chong2020characterization}, which falls outside of the Young regime in the variation setting. Meanwhile, there exists $X \in B^{1/2}_{pq}$ with $q \le 2 < p$ such that $X$ does not belong to $\mcl V^r$ for any $r \in [1,2)$ (see Proposition \ref{P:strictvariation}) so Theorem \ref{T:Youngeq} does indeed strictly expand the Young regime of solvability for \eqref{E:Young}.
\end{remark}

\begin{proof}[Proof of Theorem \ref{T:Youngeq}]
	For $M > 0$ and $0 < T_0 \le 1 \wedge T$, define
	\[
		\mathscr X_M := \left\{ Y \in B^\alpha_{pq}([0,T_0]) : Y_0 =y, \; [Y]_{B^\alpha_{pq}} \le M \right\},
	\]
	which, equipped with the metric $(Y,\tilde Y) \mapsto [Y - \tilde Y]_{B^\alpha_{pq}([0,T_0])}^{1 \wedge q}$, is a complete metric space (see Remark \ref{R:Besovsup}). We will show that, for some sufficiently large $M$ and sufficiently small $T_0$, depending only on $f$ and $X$, the map
	\[
		(\mathscr M Y)_t = y + \int_0^t f(Y_s)\cdot dX_s \quad \text{for } t \in [0,T_0],
	\]
	defined in the sense of Theorem \ref{T:YoungBesov}, has a fixed point in $\mathscr X_M$.

	{\it Step 1: $\mathscr M(\mathscr X_M) \subset \mathscr X_M$.} Let $Y \in \mathscr X_M$. By Lemma \ref{L:Besovcomposition}, if $\alpha > 1/2$, then $f(Y) \in B^{\delta \alpha}_{p,q/\delta}$, and, if $\alpha = 1/2$, then $f(Y) \in B^\alpha_{p,q}$. In any case, Theorem \ref{T:YoungBesov} implies that $Z := \mathscr MY \in B^\alpha_{pq}([0,T_0])$ is well-defined. Moreover, if $\alpha > 1/2$, then \eqref{betterYoungintegralbound} gives
	\begin{align*}
		[Z]_{B^\alpha_{pq}} 
		&\lesssim_{\alpha,p,q} \left( |f(Y_0)| + T_0^{\alpha - 1/p} [f(Y)]_{B^{\delta \alpha}_{p,q/\delta}} \right)[X]_{B^\alpha_{pq}} \\
		&\lesssim_{\alpha,p,q} \left( \nor{f}{\oo} + T_0^{\alpha - 1/p} [f]_{C^\delta} [Y]^\delta_{B^\alpha_{pq}} \right) [X]_{B^\alpha_{pq}},
	\end{align*}
	and, if $\alpha = 1/2$ (and thus $q \le 2$), then \eqref{betterYoungintegralboundendpoint} yields
	\begin{align*}
		[Z]_{B^{1/2}_{pq}} 
		\lesssim_{p,q}
		\left( \nor{f}{\oo} + T_0^{1/2 - 1/p}(1+\ell_{q/2}(T_0)) [f]_{C^1} [Y]_{B^{1/2}_{pq}} \right) [X]_{B^{1/2}_{pq}}.
	\end{align*}
	In either case, choosing first $M$ sufficiently large and then $T_0$ sufficiently small makes the right-hand side no greater than $M$, as desired.
	
	{\it Step 2: $\mathscr M: \mathscr X_M \to \mathscr X_M$ is a contraction}. For $Y,\tilde Y \in \mathscr X_M$, define $Z = \mathscr M Y$ and $\tilde Z = \mathscr M \tilde Y$, so that
	\[
		Z_t - \tilde Z_t = \int_0^t \left( f(Y_s) - f(\tilde Y_s)\right)\cdot dX_s.
	\]
	Appealing once more to Theorem \ref{T:YoungBesov} and Lemma \ref{L:Besovcomposition}, we have, for $\alpha > 1/2$,
	\begin{align*}
		[Z - \tilde Z]_{B^\alpha_{pq}([0,T_0])}
		&\lesssim_{\alpha,p,q} T_0^{\delta \alpha - 1/p} [f(Y) - f(\tilde Y)]_{B^{\delta \alpha}_{p,q/\delta}} [X]_{B^\alpha_{pq}} \\
		&\lesssim_M T_0^{\delta \alpha - 1/p} \nor{f}{C^{1,\delta}} [Y - \tilde Y]_{B^{\alpha}_{p,q}} [X]_{B^\alpha_{pq}},
	\end{align*}
	and, if $\alpha = 1/2$ and $q \le 2$,
	\begin{align*}
		[Z - \tilde Z]_{B^{1/2}_{pq}([0,T_0])}
		&\lesssim_{\beta,p,q} T_0^{1/2 - 1/p}(1+\ell_{q/2}(T_0)) [f(Y) - f(\tilde Y)]_{B^{1/2}_{p,q}} [X]_{B^{1/2}_{pq}} \\
		&\lesssim_M T_0^{ 1/2 - 1/p} (1+\ell_{q/2}(T_0)) \nor{f}{C^{1,\delta}} [Y - \tilde Y]_{B^{1/2}_{p,q}} [X]_{B^{1/2}_{pq}}.
	\end{align*}
	The contraction property is seen upon taking $T_0$ sufficiently small. Because of the choice of $T_0$, this process may be iterated on $[T_0, 2T_0]$, $[2T_0, 3T_0]$, and so on, and we conclude.
	\end{proof}

We conclude this section with the local Lipschitz continuity of the It\^o-Lyons map.

\begin{theorem}\label{T:YoungItoMap}
	Assume \eqref{Youngregime}, fix $M$, and let $X^1,X^2 \in B^\alpha_{pq}([0,T])$, $f^1,f^2$ satisfying \eqref{Youngf}, and $y^1,y^2 \in \RR^m$ be such that $[X^1]_{B^\alpha_{pq}} \vee [f^1]_{C^{1,\delta}} \vee [X^2]_{B^\alpha_{pq}} \vee [f^2]_{C^{1,\delta} }\le M$ (with $C^{1,\delta}$ replaced by $C^2$ if $\alpha = 1/2$). Then the solutions $Y^1$ and $Y^2$ of \eqref{E:Young} corresponding to respectively $(X^1,f^1,y^1)$ and $(X^2,f^2,y^2)$ satisfy
	\[
		[Y^1 - Y^2]_{B^\alpha_{pq}} \lesssim_{\alpha,p,q,M,T} \left( |y^1 - y^2| + [X^1 - X^2]_{B^\alpha_{pq}} + \nor{f^1 - f^2}{C^{\delta}} \right).
	\]
\end{theorem}

\begin{proof}
	We write
	\[
		Y^1_t - Y^2_t = y^1 - y^2 + \underbrace{\int_0^t \left( (f^1 - f^2)(Y^1)\right) dX^1_s}_{\I} + \underbrace{\int_0^t \left[ f^2(Y^1_s) - f^2(Y^2_s) \right] dX^1_s}_{\II} + \underbrace{\int_0^t f^2(Y^2_s) d\left[ X^1_s - X^2_s \right]}_{\III}.
	\]
	We estimate $\I$, $\II$, and $\III$ with the use of Theorem \ref{T:YoungBesov} and Lemma \ref{L:Besovcomposition}, noting first that, by Theorem \ref{T:Youngeq}, $[Y^1]_{B^\alpha_{pq}} \vee [Y^2]_{B^\alpha_{pq}} \lesssim_{M,T} 1$. If $\alpha > 1/2$, then
	\[
		[\I]_{B^\alpha_{pq}}
		\lesssim_{\alpha,\delta,p,q,M,T} \nor{f^1 - f^2}{\oo} + [(f^1-f^2)(Y^1)]_{B^{\delta \alpha}_{p,q/\delta}}
		\lesssim_{M,\alpha,\delta,p,q} \nor{f^1 - f^2}{\oo} + [f^1 - f^2]_{C^\delta},
	\]
	\begin{align*}
		[\II]_{B^\alpha_{pq}} 
		&\lesssim_{\alpha,p,q,\delta,M} |f^2(y^1) - f^2(y^2)| + T^{\delta\alpha - 1/p}[f^2(Y^1) - f^2(Y^2)]_{B^{\delta\alpha}_{p,q/\delta}}\\
		&\lesssim_{M,\alpha,\delta,p,q} |y^1 - y^2| + T^{\delta\alpha - 1/p} [Y^1 - Y^2]_{B^\alpha_{p,q}}.
	\end{align*}
	and $[\III]_{B^\alpha_{pq}} \lesssim_{\alpha,p,q,M,\delta} [X^1 - \tilde X^2]_{B^\alpha_{pq}}$. Combining the estimates for $\I$, $\II$, and $\III$ and taking $T_0$ sufficiently small, depending only on $M$, we obtain the desired estimate on $[0,T_0]$. The estimate on the whole of $[0,T]$, for a constant $C$ depending additionally on $T$, can be obtained by iterating the estimate on the intervals $[T_0,2T_0]$, $[2T_0,3T_0]$, etc. The argument for $\alpha = 1/2$ and $q \le 2$ proceeds in exactly the same way, replacing $f \in C^{1,\delta}$ with $f \in C^2$, and the factors of $T_0^{\alpha -1/p}$ with $T_0^{1/2 - 1/p}(1 +\ell_{q/2}(T_0))$.
\end{proof}

\section{The rough path setting}\label{sec:roughpaths}

This section develops the theory of rough paths and rough differential equations with Besov regularity. We first introduce a general definition of Besov rough paths with arbitrarily low regularity. We give a full description of controlled paths in the level-$2$ case, and, with the use of the sewing techniques from Section \ref{sec:BesovSew}, prove well-posedness for controlled rough differential equations driven by Besov rough signals. We also outline the framework that leads to the same results for level-$N$ geometric rough paths for $N \ge 3$.

\subsection{Besov rough paths}

For $N \in \NN$, define the the truncated tensor algebra of level $N$ by
\[
	T^{(N)}(\RR^n) = \RR \oplus \bigoplus_{k=1}^N (\RR^n)^{\otimes k},
\]
with the noncommutative tensor product $\otimes$. For $a \in \RR$, define also
\[
	T^{(N)}_a(\RR^n) := \left\{ a + \sum_{k=1}^N v_j \in T^{(N)}(\RR^n), v_k \in (\RR^n)^{\otimes k} \right\}.
\]
For $\lambda > 0$ and $k = 0,1,2,\ldots,N$, we define $\delta_\lambda : T^{(N)}(\RR^n) \to T^{(N)}(\RR^n)$ and $\tau_k : T^{(N)}(\RR^n) \to T^{(k)}(\RR^n)$ as follows: for $v = v_0 + v_1 + \cdots + v_N \in T^{(N)}(\RR^n)$ with $v_k \in (\RR^n)^{\otimes k}$,
\[
	\delta_\lambda v= v_0 + \lambda v_1 + \lambda^2 v_2 + \cdots + \lambda^N v_N \quad \text{and} \quad \tau_k v = v_0 + v_1 + \cdots + v_k.
\]
The level-$N$ Besov rough path regime shall be described with the parameters
\begin{equation}\label{levelNparameters}
	\left\{
	\begin{split}
	&\frac{1}{N+1} \le \alpha < 1, \quad \frac{1}{\alpha} < p \le \oo, \\
	&0 < q \le \oo \quad \text{if } \alpha > \frac{1}{N+1}, \text{ and}\\
	&0 < q \le N+1 \quad \text{if } \alpha = \frac{1}{N+1}.
	\end{split}
	\right.
\end{equation}

\begin{definition}\label{D:BesovRP}
For $N \in \NN$ and $\alpha$, $p$, and $q$ as in \eqref{levelNparameters}, $\mbf X: \Delta_2(0,T) \to T^{(N)}_1(\RR^n)$ is a level-$N$ $\mbf B^\alpha_{pq}([0,T])$-rough path if 
\begin{equation}\label{Chen}
	\mbf X_{st} \otimes \mbf X_{tu} = \mbf X_{su} \quad \text{for all } (s,t,u) \in \Delta_3(0,T)
\end{equation}
and
\begin{equation}\label{roughquasinorm}
	\normm{\mbf X}_{\mbf B^\alpha_{pq}([0,T])} := \sum_{k=1}^N \nor{\mbf X^{(k)}}{\mbb B^{k\alpha}_{p/k,q/k}}^{1/k} < \oo.
\end{equation}
To emphasize the level-$N$ setting, we write also $\mbf B^\alpha_{pq} = \mbf B^\alpha_{pq; N}$ and omit the ``$N$'' when it does not create confusion. The space $\mbf B^\alpha_{pq}([0,T])$ is a complete metric space equipped with
\begin{equation}\label{roughmetric}
	\rho_{\mbf B^\alpha_{pq}([0,T])}(\mbf X,\mbf Y) := \sum_{k=1}^N \mathrm d_{\mbb B^{k\alpha}_{p/k,q/k}} \left( \mbf X^{(k)}, \mbf Y^{(k)} \right),
\end{equation}
where the metric $\mathrm d_{\mbb B^{k\alpha}_{p/k,q}}$ is defined as in \eqref{BBmetric}.
\end{definition}

\begin{remark}\label{R:beyond2}
	When $N = 2$, Definition \ref{D:BesovRP} allows for a complete theory of rough differential equations, as discussed in subsections \ref{SS:ctrld} and \ref{SS:BesovRDE}. For $N \ge 3$, and for non-linear equations, additional algebraic considerations are necessary, as discussed in subsection \ref{sec:beyond}.

%
%
\end{remark}  

\begin{remark}\label{R:dilation}
Even if $p,q\ge N$ (and thus $p/k,q/k \ge 1$ for all $k = 1,2,\ldots, N$), the quantity \eqref{roughquasinorm} defines only a quasi-norm. We have the following useful homogeneity property: for $\lambda \in \RR$ and $\mbf X \in \mbf B^\alpha_{pq}$, $\delta_\lambda \mbf X$ is a level-$N$ $\mbf B^\alpha_{pq}$ rough path, with
\[
	\normm{\delta_\lambda \mbf X}_{\mbf B^\alpha_{pq}} = \lambda \normm{\mbf X}_{\mbf B^\alpha_{pq}}.
\]
\end{remark}

\begin{remark}\label{R:tensormetric}
With slight abuse of notation, we set $\mbf X_t := \mbf X_{0,t}$. Then \eqref{Chen} implies that $\mbf X$ is a bona fide path taking values in $\mbf T^{(N)}_1(\RR^n)$, and in fact $\mbf X_{s,t} = \mbf X_s^{-1} \otimes \mbf X_t$.

For $\mbf x,\mbf y \in T^{(N)}_1(\RR^n)$, we define a metric $ d$ by $d(\mbf x, \mbf y) = \normm{\mbf x^{-1} \otimes \mbf y}_*$, where $\normm{\mbf x}_* := \frac{1}{2} \left( N(\mbf x) + N(\mbf x^{-1})\right)$ and $N(\mbf x) = \max_{k = 1,2,\ldots,N} \left( k! |\mbf x^{(k)}| \right)^{1/k}$, and the norm $|\cdot |$ on $(\RR^n)^{\otimes k}$ is taken to be permutation-invariant; see \cite{friz2020course} for more details. It is then easy to see that
\[
	\normm{\mbf X}_{\mbf B^\alpha_{pq}} \asymp_{\alpha,p,q,N} \left[ \int_0^T \left( \frac{ \sup_{0 < h < \tau}  \nor{ d(\mbf X_\cdot,\mbf X_{\cdot + h}) }{L^p([0,T-h])} }{\tau^\alpha} \right)^q \frac{d\tau}{\tau} \right]^{1/q}
	= [\mbf X]_{B^\alpha_{pq}([0,T], T^{(N)}_1(\RR^n))};
\]
that is, $\mbf B^\alpha_{pq}([0,T]) = B^\alpha_{pq}\left([0,T], T^{(N)}_1(\RR^n)\right)$.
\end{remark}

\subsubsection{H\"older embeddings} \label{sec:HoelderEmb}
For a level-$N$ rough path $\mbf X \in \mbf B^\alpha_{pq}$, because $\alpha > \frac{1}{p}$, $\mbf X$ is H\"older continuous, and, moreover, the distance between two such rough paths in the H\"older distance is controlled by the Besov one, as demonstrated by the next result. Throughout the rest of the section, we thus may work with the continuous version of $\mbf X$.

\begin{proposition}\label{P:roughBesovembedding}
	Assume \eqref{levelNparameters} and let $\mbf X, \mbf{\tilde X}$ be level-$N$ $\mbf B^\alpha_{pq}$ rough paths. Then, for all $k = 1,2,\ldots, N$,
	\begin{equation}\label{roughHolderBesovnorms}
		\nor{\mbf X^{(k)}}{\mbb C^{k(\alpha - 1/p)}([0,T]) }^{1/k}
		\lesssim_{\alpha,p,q}  \normm{\tau_k \mbf X}_{\mbf B^\alpha_{pq}([0,T])}.
	\end{equation}
	and
	\begin{equation}\label{roughHolderBesovdistance}
	\begin{split}
		&\nor{\mbf X^{(k)} - \tilde{\mbf X}^{(k)} }{\mbb C^{k(\alpha - 1/p)}([0,T])}\\
		&\lesssim_{\alpha,p,q} \sum_{j=1}^k \left( \normm{\tau_{k} \mbf X}_{\mbf B^\alpha_{pq}([0,T])} \vee  \normm{\tau_{k}\tilde{\mbf X}}_{\mbf B^\alpha_{pq}([0,T])} \right)^{k-j} \nor{\mbf X^{(j)} - \tilde{\mbf X}^{(j)} }{\mbb B^{j\alpha}_{p/j,q/j}([0,T])}.
	\end{split}
	\end{equation}
\end{proposition}

Proposition \ref{P:roughBesovembedding} is proved with an application of Proposition \ref{P:BBesovembedding}. Note also that \eqref{roughHolderBesovnorms} follows from Proposition \ref{P:Besovembedding}, in view of Remark \ref{R:tensormetric}.

\begin{proof}[Proof of Proposition \ref{P:roughBesovembedding}]
	The proof of \eqref{roughHolderBesovnorms} follows by induction. First, Chen's relations imply that $\delta \mbf X^{(1)} \equiv 0$, so Proposition \ref{P:BBesovembedding} gives the estimate for $\mbf X^{(1)}$ (alternatively, we have $\mbf X^{(1)}_{st} = X_t - X_s$ for some $X \in B^\alpha_{pq}$, and so the result is classical).
	
	Now let $n \ge 2$ be fixed and assume the estimate holds for each $\mbf X^{(m)}$ with $m < n$. Chen's relations then imply that, for $0 \le s \le u \le t \le T$,
	\[
		\delta \mbf X^{(n)}_{sut} = \sum_{m=1}^{n-1} \mbf X^{(m)}_{su} \otimes \mbf X^{(n-m)}_{ut}.
	\] 
	The inductive hypothesis then yields
	\begin{align*}
		\abs{ \delta \mbf X^{(n)}_{sut}} 
		&\lesssim \sum_{m=1}^{n-1} \normm{\tau_m \mbf X}_{\mbf B^\alpha_{pq}}^m \normm{\tau_{n-m} \mbf X}_{\mbf B^\alpha_{pq}}^{n-m} (u-s)^{m(\alpha - 1/p)} (t-u)^{(n-m)(\alpha - 1/p)}\\
		&\lesssim \normm{\tau_n \mbf X}_{\mbf B^\alpha_{pq}}^n \left[ (u-s) \vee (t-u) \right]^{\alpha - 1/p} \left[ (u-s) \wedge (t-u) \right]^{(n-1)(\alpha - 1/p)},
	\end{align*}
	and we conclude by appealing to Proposition \ref{P:BBesovembedding}.
	
	To prove \eqref{roughHolderBesovdistance}, we note first that $\delta(\mbf X^{(1)} - \tilde{\mbf X}^{(1)}) = 0$, so that the estimate for $\mbf X^{(1)} - \tilde{\mbf X}^{(1)}$ is a consequence of Proposition \ref{P:BBesovembedding}.
	
	Fix $n \ge 2$ and assume that the result holds for all $m < n$. We then have, for $0 \le s \le u \le t \le T$,
	\begin{align*}
		\delta(\mbf X^{(n)} - \tilde{\mbf X}^{(n)})_{sut}
		&= \sum_{m=1}^{n-1} \left( \mbf X^{(m)}_{su} \otimes \mbf X^{(n-m)}_{ut} - \tilde{\mbf X}^{(m)}_{su} \otimes \tilde{\mbf X}^{(n-m)}_{ut} \right)\\
		&= \sum_{m=1}^{n-1} \left[ \left( \mbf X^{(m)}_{su} - \tilde{\mbf X}^{(m)}_{su} \right) \otimes \mbf X^{(n-m)}_{ut} + \tilde{\mbf X}^{(m)}_{su} \otimes \left( \mbf X^{(n-m)}_{ut} - \tilde{\mbf X}^{(n-m)}_{ut} \right) \right].
	\end{align*}
	Then both the inductive hypothesis and \eqref{roughHolderBesovnorms} imply that
	\begin{align*}
		&\abs{\delta(\mbf X^{(n)} - \tilde{\mbf X}^{(n)})_{sut}}\\
		&\lesssim \sum_{m=1}^{n-1} \left[ \normm{\tau_{n-m} \mbf X}_{\mbf B^\alpha_{pq}}^{n-m} \sum_{k=1}^m \left( \normm{\tau_{m} \mbf X}_{\mbf B^\alpha_{pq}}\vee \normm{\tau_{m} \tilde{\mbf X}}_{\mbf B^\alpha_{pq}} \right)^{m-k} \nor{\mbf X^{(k)} - \tilde{\mbf X}^{(k)}}{\mbb B^{k\alpha}_{p/k,q} } \right. \\
		&\qquad+ \left. \normm{\tau_m \tilde{\mbf X}}_{\mbf B^\alpha_{pq}}^m \sum_{\ell=1}^{n-m} \left( \normm{\tau_{n-m} \mbf X}_{\mbf B^\alpha_{pq}} \vee \normm{\tau_{n-m} \tilde{\mbf X}}_{\mbf B^\alpha_{pq}} \right)^{n-m-\ell} \nor{\mbf X^{(l)} - \tilde{\mbf X}^{(l)} }{\mbb B^{\ell \alpha}_{p/\ell,q} } \right]\\
		&\qquad \cdot (s-u)^{m(\alpha - 1/p)} (t-u)^{(n-m)(\alpha - 1/p)}\\
		&\lesssim \sum_{m=1}^{n-1} \left( \normm{\tau_n \mbf X}_{\mbf B^\alpha_{pq}} \vee \normm{\tau_n \tilde{\mbf X}}_{\mbf B^\alpha_{pq}} \right)^{n-m} \nor{\mbf X^{(m)} - \tilde{\mbf X}^{(m)} }{\mbb B^{m\alpha}_{p/m,q}} \\
		&\qquad \cdot \left[ (u-s) \vee (t-u) \right]^{\alpha - 1/p} \left[ (u-s) \wedge (t-u) \right]^{(n-1)(\alpha - 1/p)}. 
	\end{align*}
	The result now follows from Proposition \ref{P:BBesovembedding}.
\end{proof}

Inspired by the Campanato characterization of H\"older continuity, we record here the following result for Besov rough paths, which follows from Proposition \ref{P:generalCampanato}.

\begin{corollary}
Assume that $\mbf X: [0,T] \to T^{(N)}(\RR^m)$ satisfies Chen's relation \eqref{Chen} and, for some $C > 0$ and $\beta \in (0,1)$,
	\begin{equation}\label{Campanato}
		\frac{1}{b-a} \int_a^b \abs{ \frac{1}{b-a} \int_a^b \mbf X^{(n)}_{st}ds}dt \le C(b-a)^{\beta n} \quad \text{for all } a < b \text{ and } n=1,..,N.
	\end{equation}
	Then there exists $M = M_\beta > 0$ such that
	\[
		\abs{ \mbf X^{(n)}_{st}}^{1/n} \le MC|t-s|^\beta \quad \text{for all }(s,t) \in [0,T]^2.
	\]
\end{corollary}

Along with Lemma \ref{L:interpolate}, Proposition \ref{P:roughBesovembedding} allows for the following estimates that trade regularity for integrability.

\begin{lemma}\label{L:roughinterpolate}
	Assume \eqref{levelNparameters} and let $\mbf X, \tilde{\mbf X} \in \mbf B^\alpha_{pq}([0,T])$. Then, for all $j,k \in \{1,2,\ldots,N\}$ with $j < k$,
	\[
		\nor{\mbf X^{(k)}}{\mbb B^{j\alpha}_{p/j,q/j}([0,T])} \lesssim_{\alpha,p,q} T^{(k-j)(\alpha - 1/p)} \normm{\tau_k \mbf X}_{\mbf B^\alpha_{pq}([0,T])}^k
	\]
	and
	\begin{align*}
		\nor{\mbf X^{(k)} - \tilde{\mbf X}^{(k)} }{\mbb B^{j\alpha}_{p/j,q/j}([0,T])} \lesssim_{\alpha,p,q} T^{(k-j)(\alpha - 1/p)} \sum_{i=1}^k \left( \normm{\tau_{k} \mbf X}_{\mbf B^\alpha_{pq}([0,T])} \vee  \normm{\tau_{k}\tilde{\mbf X}}_{\mbf B^\alpha_{pq}([0,T])} \right)^{k-i} \\
		\cdot \nor{\mbf X^{(i)} - \tilde{\mbf X}^{(i)} }{\mbb B^{i\alpha}_{p/i,q/i}([0,T])}
	\end{align*}
\end{lemma}

\begin{proof}
	Lemma \ref{L:interpolate} and Proposition \ref{P:roughBesovembedding} immediately yield
	\begin{align*}
		\nor{\mbf X^{(k)}}{\mbb B^{j\alpha}_{p/j,q/j}([0,T])}
		&\lesssim_{\alpha,p,q} T^{(k-j)(\alpha - 1/p)} \nor{\mbf X^{(k)}}{\mbb C^{k(\alpha - 1/p)}}^{\frac{k-j}{k}} \nor{\mbf X^{(k)}}{\mbb B^{k\alpha}_{p/k,q/j}}^{\frac{j}{k}}\\
		&\lesssim_{\alpha,p,q} T^{(k-j)(\alpha - 1/p)} \normm{\tau_k \mbf X}_{\mbf B^\alpha_{pq}([0,T])}^{k-j} \nor{\mbf X^{(k)}}{\mbb B^{k\alpha}_{p/k,q/k}}^{\frac{j}{k}}
		\le T^{(k-j)(\alpha - 1/p)} \normm{\tau_k \mbf X}_{\mbf B^\alpha_{pq}([0,T])}^k.
	\end{align*}
	The estimate for $\mbf X^{(k)} - \tilde{\mbf X}^{(k)}$ follows in the same way, using the second estiamte in Proposition \ref{P:roughBesovembedding}.
\end{proof}

\subsubsection{Lyons extension theorem}
We next present our first application of Besov sewing in the rough path context, which is a version of the Lyons extension theorem on the Besov scale. For the reader interested in solving Besov rough differential equations, the subsequent result can be skipped over at no loss.

\begin{theorem}\label{T:Lyonsextension}
	Assume that $\frac{1}{M+1} < \alpha < 1$, $1/\alpha < p \le \oo$, and $0 < q \le \oo$. Then, for any $N \ge M$, there exists a continuous map $\mathscr E : \mbf B^\alpha_{pq;M} \to \mbf B^\alpha_{pq;N}$ such that, given $\mbf X, \tilde{\mbf X} \in \mbf B^\alpha_{pq,M}$, for all $k = 1,2,\ldots,M$,
	\[
		(\mathscr E\mbf X)^{(k)} = \mbf X^{(k)}, 
	\]
	and, for all $k = M+1,M+2,\ldots, N$,
	\begin{equation}\label{lyonsmap}
		\nor{(\mathscr E\mbf X)^{(k)}}{\mbb B^{k\alpha}_{p/k,q/k}} \lesssim_{M,N,\alpha,p,q} \normm{\mbf X}_{\mbf B^\alpha_{pq;M}}^k
	\end{equation}
	and
	\begin{equation}\label{lyonscty}
		\nor{(\mathscr E\mbf X)^{(k)} - (\mathscr E \tilde{\mbf X})^{(k)} }{\mbb B^{k\alpha}_{p/k,q/k}} \lesssim_{M,N\alpha,p,q} \sum_{j=1}^M \nor{\mbf X^j - \tilde{\mbf X}^j}{\mbb B^{j\alpha}_{p/j,q/j}} \left( \normm{\mbf X}_{\mbf B^\alpha_{pq;M}} \vee \normm{\tilde{\mbf X}}_{\mbf B^\alpha_{pq;M}} \right)^{k-j}.
	\end{equation}
\end{theorem}

\begin{proof}
	It suffices to consider the case where $N = M+1$, and the general statement then follows by induction. For $(s,t) \in \Delta_2(0,T)$, define
	\[
		A_{st} := \sum_{k=1}^M \mbf X^{M-k+1}_{0,s} \otimes \mbf X^{(k)}_{s,t}.
	\]
	Then, for $(s,u,t) \in \Delta_3(0,T)$,
	\begin{align*}
		\delta A_{sut} &= \sum_{k=1}^M \left( - \mbf X^{(M-k+1)}_{su} \otimes \mbf X^{(k)}_{ut} - \delta \mbf X^{(M-k+1)}_{0su} \otimes \mbf X^{(k)}_{ut} + \mbf X^{(M-k+1)}_{0s} \otimes \delta \mbf X^{(k)}_{sut} \right) \\
		&= - \sum_{k=1}^M  \mbf X^{(M-k+1)}_{su} \otimes \mbf X^{(k)}_{ut} - \sum_{k=1}^M \sum_{j=1}^{M-k+1} \mbf X^{(j)}_{0s} \otimes \mbf X^{(M-k+1-j)}_{su}\otimes  \mbf X^{(k)}_{ut}
		+ \sum_{k=1}^M \mbf X^{(M-k+1)}_{0s} \otimes \sum_{j=1}^k \mbf X^{(j)}_{su} \otimes \mbf X^{(k-j)}_{ut} \\
		&=- \sum_{k=1}^M  \mbf X^{(M-k+1)}_{su} \otimes \mbf X^{(k)}_{ut},
	\end{align*}
	and so
	\[
		\nor{\delta A}{\oline{\mbb B}^{(M+1)\alpha}_{p/(M+1), q/(M+1)}} \lesssim_q \sum_{k=1}^M \nor{\mbf X^{(M-k+1)}}{\mbb B^{(M-k+1)\alpha}_{p/(M-k+1),q/(M-k+1)}} \nor{\mbf X^{(k)}}{\mbb B^{k\alpha}_{p/k,q/k}}
		\lesssim \normm{\mbf X}_{\mbf B^\alpha_{pq}}^{M+1}.
	\]
	The estimate \eqref{Rbound} from Theorem \ref{T:Besovsewing} then yields
	\[
		\nor{\mathscr RA}{\oline{\mbb B}^{(M+1)\alpha}_{p/(M+1), q/(M+1)}}  \lesssim_{\alpha,p,q} \normm{\mbf X}_{\mbf B^\alpha_{pq}}^{M+1}.
	\]
	Setting $(\mathscr E \mbf X)^{(M+1)} := \mathscr RA$, it is easy to check that Chen's relations \eqref{Chen} hold. This establishes the existence of the map and the bound \eqref{lyonsmap}. To prove \eqref{lyonscty}, we take $\tilde{\mbf X} \in \mbf B^\alpha_{pq}$ and set
	\[
		\tilde A_{st} := \sum_{k=1}^M \tilde {\mbf X}^{M-k+1}_{0,s} \otimes \tilde {\mbf X}^{(k)}_{s,t}.
	\]
	Then
	\[
		\delta A_{st} - \delta \tilde A_{st} = - \sum_{k=1}^M  \mbf X^{(M-k+1)}_{su} \otimes \mbf X^{(k)}_{ut} +  \sum_{k=1}^M \tilde{ \mbf X}^{(M-k+1)}_{su} \otimes \tilde{\mbf X}^{(k)}_{ut},
	\]
	whence
	\[
		\nor{\delta A - \delta \tilde A}{\oline{\mbb B}^{(M+1)\alpha}_{p/(M+1),q/(M+1)}}
		\lesssim_{p,q} \sum_{k=1}^M \nor{X^{(k)} - \tilde X^{(k)}}{\mbb B^{k\alpha}_{p/k,q/k}} \left( \normm{\mbf X}_{\mbf B^\alpha_{pq;M}} \vee \normm{\tilde{\mbf X}}_{\mbf B^\alpha_{pq;M}} \right)^{M+1-k}.
	\]
	The bound \eqref{lyonscty} then follows from another application of the sewing lemma estimate \eqref{Rbound}.
\end{proof}

\begin{remark}\label{R:lyons}
	Theorem \ref{T:Lyonsextension} is proved in the regime of \eqref{levelNparameters} where $\alpha > \frac{1}{M+1}$. When $\alpha = \frac{1}{M+1}$ and $0 < q \le M+1$, there is the subtlety that the component $(\mathscr E \mbf X)^{(M+1)}$ constructued with Theorem \ref{T:Besovsewingendpoint} belongs, not to $\mbb B^{(M+1)\alpha}_{p/(M+1),q/(M+1)}$, but to $\mbb B^{\omega}_{p/(M+1),q/(M+1)}$ for some modulus $\omega$ that is a bit worse (by, say, a logarithmic correction) than $\tau \mapsto \tau^{(M+1)\alpha}$. One can conceive of a generalization of Definition \ref{D:BesovRP} in which the regularity of the various components is given by more general moduli than the powers $\tau^\alpha$, $\tau^{2\alpha}$, $\tau^{3\alpha}$, etc., for which an extension result in the endpoint case could be proved, but we do not pursue this here.
	
	These remarks are consistent with Theorem \ref{T:YoungBesov}, which can be seen as a generalization of Theorem \ref{T:Lyonsextension} when $M =1$. Indeed, Theorem \ref{T:YoungBesov} implies that, for $X \in B^{1/2}_{p,q}([0,T],\RR^m)$ with $q \le 2 < p$, the well-defined iterated integrals
	\[
		\Delta_2(0,T) \ni (s,t) \mapsto \mbb X_{st} := \int_s^t \delta X_{s,r} \otimes dX_r \in \RR^n \otimes \RR^n
	\]
	belong, not to $\mbb B^1_{p/2,q/2}$, but rather $\mbb B^\omega_{p/2,q/2}$ for some modulus $\omega$ involving a logarithmic correction of $\tau \mapsto \tau$.
\end{remark}

\subsubsection{Multidimensional stochastic processes as Besov rough paths} 

{\bf Brownian and fractional Brownian motion.}  A classical result \cite{MR1277166} asserts that for standard $n$-dimensional Brownian motion $W (\omega) \in B^{1/2}_{p\infty}([0,T])$, almost surely.
Let $\mbf W(\omega)$ be It\^o- or Stratonovich Brownian rough paths over $\RR^n$, see eg. \cite[Ch.3]{friz2020course}.
The following regularity results appear to be new, even in the Brownian rough path case. 

\begin{theorem} \label{thm:BMbesovrough}
Almost surely, $\mbf W (\omega) \in \mbf B^{1/2}_{p\infty}([0,T])$ for every $p< \infty$, and $\mbf W (\omega) \notin \mbf B^{1/2}_{\infty\infty}([0,T])$.
For fractional Brownian rough paths with Hurst parameter $H \in (1/4,1/2]$, the analogous statement holds with $\mbf B^{1/2}_{p\infty}([0,T])$ replaced by $\mbf B^{H}_{p\infty}([0,T])$.
\end{theorem}
\begin{proof} The negative inclusion is immediate from the corresponding (well-known) statement for (fractional) Brownian motion, together with the fact that the rough Besov norm of any rough path over some path $X$ dominates the classical Besov norm of some $X$. We show the first statement, in the Brownian case $H=1/2$. By Lemma \ref{L:discreteBesovnorm}, it suffices to check that, with probability one, $(Y_{np})_{n=0}^\oo$ is bounded in $n$, where
$$
       Y_{np}^p := 2^{np/2} \int_0^1  d (\mbf W_t,\mbf W_{t+2^{-n}})^p dt 
$$
and $d$ is the metric introduced in Remark \ref{R:tensormetric}.
By basic properties of Brownian (rough) paths, and homogeneity of the metric, it is clear that $2^{np/2} d (\mbf W_t,\mbf W_{t+2^{-n}})^p $ has the same law as $ d (\mbf W (0), \mbf W (1))^p$, with (finite) mean $c_p^p$. This is also the mean of $Y_{np}^p$, and one estimates without problem that the variance of $Y_{np}^p$ goes to zero, with rate $2^{-n}$. A Borel-Cantelli argument then shows a.s. convergence $Y_{np}^p \to c_p^p$. This implies the desired boundedness, and the proof is finished. (See \cite{MR2445509} for a similar argument, applied to Banach valued Brownian motion.)
Using basic facts about fractional Brownian rough paths, in particular existence for $H> 1/4$ and their natural scaling properties (see e.g. \cite{friz2020course} in case $H > 1/3$, and  \cite{friz2010multidimensional} for the general case), the above argument extends immediately to the fractional setting. 
\end{proof} 

\begin{remark} The case $H > 1/2$ is not excluded, but (as level-$1$ rough path) is not particularly challenging.
On the other hand, the construction of higher order iterated integrals in Theorem \ref{T:Lyonsextension} with the correct Besov (rough path) regularity, is non-trivial even in case $H>1/2$.
\end{remark}

{\bf Local martingales.}  We first recall the BDG inequality in the Besov scale. For a (continuous) local martingale $(M_t:0 \le t \le T)$ with $(\gamma_0,p_0,q_0)$-Besov regularity, $\gamma_0 - 1/p_0 > 0$, and $r \in [1,\infty)$, one has
\begin{equation} \label{eq:Besov-BDG:M<S}
\norm{ \norm{M}_{B^{\gamma_{0}}_{p_{0},q_{0}}} }_{L^{r}(\Omega)}
\lesssim
\norm{ \norm{SM}_{\mbb B^{\gamma_{0}}_{p_{0},q_{0}}} }_{L^{r}(\Omega)},
\end{equation}
where
\begin{equation} \label{eq:S:cont}
SM_{s,t} := (\langle M \rangle_t - \langle M \rangle_s)^{1/2},
\end{equation}
the square-root of the quadratic varation of $M$, is conveniently measured in $2$-parameter Besov sense. (This is essentially found in \cite{MR1277166}, and also follows from taking $F \equiv 1$ in Appendix~\ref{app:besov}.)
Note that the employed left-hand Besov norm on $[0,T]$ dominates the uniform norm of $M$, so that the classical BDG inequality immediately shows that one can replace $  \lesssim$ above by a two-sided sandwich estimate.

A multidimensional local martingale $M = (M^1,\dotsc,M^n)$ can be enhanced to a rough process via the $2$-parameter process
\[
\mathbb{M}_{s,t} := \int_{u=s}^{t} \delta M_{s,u-} \otimes \dif M_{u},
\]
where the integration is taken in the It\^o sense.
For this enhanced process, an estimate analogous to \eqref{eq:Besov-BDG:M<S} holds with the homogenous Besov rough path norm defined in \eqref{roughquasinorm}.
\begin{theorem}[Rough BDG in Besov scale] \label{thm:Martbesovrough}
Let $r\in (1,\infty)$, $p,q \in (2,\infty)$, and $\alpha \in (1/3,1)$ with $\alpha - 1/p > 0$.
Let $M$ be an $n$-dimensional local martingale, and $\mathbf{M} = (M, \mathbb{M})$ the resulting It\^o local martingale rough path over $\RR^n$.
Then the following BDG type estimate holds,
\begin{equation} \label{eq:Martbesovrough}
\norm{ \normm{\mbf M}_{\mbf B^\alpha_{p,q}} }_{L^{r}(\Omega)}
\lesssim_{\alpha,p,q,r}  
\norm{ \norm{S M}_{\mbb B^{\alpha}_{p,q}} }_{L^{r}(\Omega)},
\end{equation}
where the square-root of the quadratic variation $SM$ is defined as in \eqref{eq:S:cont}.
\end{theorem}
We do not claim that the ranges of exponents in Theorem~\ref{thm:Martbesovrough} are optimal.

As can be seen from analogous results in $p$-variation scale \cite{friz2008burkholder,MR3909973,MR4003122,arxiv:2008.08897}, it is illuminating to consider results such as Theorem~\ref{thm:Martbesovrough} in an anisotropic setting, in which different components of the martingale are measured in different norms.
This is the content of Appendix~\ref{appendix}.
\begin{proof}[Proof of Theorem~\ref{thm:Martbesovrough}]
Apply Theorem~\ref{thm:besov-pprod} pairwise to the components of $M$, and use \eqref{eq:Besov-BDG:M<S} to replace Besov norms of $M$ that appear by $2$-parameter Besov norms of $SM$.
\end{proof}

Finally, note that $\normm{\mbf M}_{\mbf B^\alpha_{p,q}}$ on the left-hand side of \eqref{eq:Martbesovrough} can be strengthened to the homogenous Besov rough path norm of the
canoncial level-$N$ lift of $\mbf M$, which is supplied by Theorem \ref{T:Lyonsextension}.

\subsection{Controlled rough paths: the level-$2$ case} \label{SS:ctrld}

For the remainder of subsections \ref{SS:ctrld} and \ref{SS:BesovRDE}, we consider the level-$2$ case, that is,
\begin{equation}\label{level2parameters}
	\left\{
	\begin{split}
	&\frac{1}{3} \le \alpha < 1, \quad \frac{1}{\alpha} < p \le \oo, \\
	&0 < q \le \oo \quad \text{if } \alpha > \frac{1}{3}, \text{ and}\\
	&0 < q \le 3 \quad \text{if } \alpha = \frac{1}{3}.
	\end{split}
	\right.
\end{equation}
We also follow notation similar to, say, \cite{friz2020course,gubinelli2004controlling}, and write $\mbf X = (\mbf X^{(1)}, \mbf X^{(2)}) = (\delta X, \mbb X)$, with the convention that $X_0 \equiv 0$.

\begin{definition}\label{D:controlledBesovRP}
Assume \eqref{level2parameters} and let $\mbf X \in \mbf B^\alpha_{pq}([0,T],\RR^n)$. Then $(Y,Y'):[0,T] \to \RR^m \times \RR^n \otimes (\RR^m)$ is called an $\mbf X$-controlled rough path over $\RR^m$, and $(Y,Y') \in \mathscr B^{\alpha}_{pq, \mbf X}([0,T],\RR^m)$, if
\begin{equation}\label{Yreg}
	Y \in B^\alpha_{pq}([0,T],\RR^m), \quad Y' \in B^\alpha_{pq}([0,T], \RR^n \otimes \RR^m), \quad \text{and} \quad R^Y := \delta Y - Y' \delta X \in \mbb B^{2\alpha}_{p/2,q/2}([0,T],\RR^m).
\end{equation}
We define
\[
	[(Y,Y')]_{\mathscr B^\alpha_{pq,\mbf X}} := [Y']_{B^\alpha_{pq}} + \nor{R^Y}{\mbb B^{2\alpha}_{p/2,q/2}}.
\]
and, given $\mbf X, \tilde{\mbf X} \in \mbf B^\alpha_{pq}$, $(Y,Y') \in \mathscr B^\alpha_{pq, \mbf X}$, and $(\tilde Y, \tilde Y') \in \mathscr B^\alpha_{pq,\tilde{\mbf X}}$,
\[
	d_{\mbf X, \tilde{\mbf X}, \mathscr B^\alpha_{pq}}( Y,Y'; \tilde Y, \tilde Y') = \mathrm d_{B^\alpha_{pq}}(Y' , \tilde Y') + \mathrm d_{\mbb B^{2\alpha}_{p/2,q/2}} (R^Y , R^{\tilde Y}).
\]
If $\mbf X = \tilde{\mbf X}$, we write $d_{\mbf X, \mathscr B^\alpha_{pq}} := d_{\mbf X, \mbf X, \mathscr B^\alpha_{pq}}$.
\end{definition}

\begin{lemma}\label{L:controllednorms}
	Assume \eqref{level2parameters} and let $\mbf X \in \mbf B^\alpha_{pq}$ and $(Y,Y') \in \mathscr B^\alpha_{pq, \mbf X}$. Then, for all $\beta$ satisfying $\alpha + \frac{1}{p} < \beta \le 2\alpha$,
	\begin{equation}\label{RYHolder}
		\nor{R^Y}{\mbb C^{\beta - 2/p}([0,T])} \lesssim_{\alpha,\beta,p,q} \nor{R^Y}{\mbb B^{\beta}_{p/2,q/2} } + [Y']_{B^\alpha_{pq}} [X]_{B^\alpha_{pq}}
	\end{equation}
	and
	\begin{equation}\label{Ycontrol}
		[Y]_{B^\alpha_{pq}} \lesssim_{\alpha,p,q} |Y'_0|[X]_{B^\alpha_{pq}} + (T^{\beta - \alpha - 1/p} \vee T^{\alpha - 1/p}) \left( [Y']_{B^\alpha_{pq} }[X]_{B^\alpha_{pq}} + \nor{R^Y}{\mbb B^{\beta}_{p/2,q/2}}\right).
	\end{equation}
	If also $\tilde{\mbf X} \in \mbf B^\alpha_{pq}$ and $(\tilde Y, \tilde Y') \in \mathscr B^\alpha_{pq,\tilde{\mbf X}}$, then
	\begin{equation}\label{RYRtildeYHolder}
		\nor{R^Y - R^{\tilde Y}}{\mbb C^{\beta - 2/p}([0,T])} \lesssim_{\alpha,\beta,p,q} \left( \nor{R^Y - R^{\tilde Y}}{\mbb B^{\beta}_{p/2,q/2}} + [Y' - \tilde Y']_{B^\alpha_{pq}}[X]_{B^\alpha_{pq}} + [\tilde Y']_{B^\alpha_{pq}} [X - \tilde X]_{B^\alpha_{pq}} \right)
	\end{equation}
	and
	\begin{equation}\label{Ydistancecontrol}
	\begin{split}
		[Y - \tilde Y]_{B^\alpha_{pq}} 
		&\lesssim_{\alpha,p,q} |Y'_0 - \tilde Y'_0|[X]_{B^\alpha_{pq}} + |\tilde Y'_0|[X - \tilde X]_{B^\alpha_{pq}}\\
		&+ T^{\beta - \alpha - 1/p} \left( [Y' - \tilde Y']_{B^\alpha_{pq} }[X]_{B^\alpha_{pq}} + [\tilde Y']_{B^\alpha_{pq}} [X - \tilde X]_{B^\alpha_{pq}} + \nor{R^Y - R^{\tilde Y}}{\mbb B^{\beta}_{p/2,q/2}}\right).
	\end{split}
	\end{equation}
\end{lemma}

\begin{proof}
	For $(s,u,t) \in \Delta_3(0,T)$, we have $\delta R^Y_{sut} = \delta Y'_{s,u} \delta X_{u,t}$, and so Proposition \ref{P:Besovembedding} gives
	\[
		|\delta R^Y_{sut}| \lesssim_{\alpha,p,q} [Y']_{B^\alpha_{pq}} [X]_{B^\alpha_{pq}} (u-s)^{\alpha - 1/p} (t-u)^{\alpha - 1/p}.
	\]
	Because $R^Y \in \mbb B^{\beta}_{p/2,q/2}$ and $\beta - \frac{2}{p} > \alpha - 1/p > 0$, \eqref{RYHolder} is now a consequence of Proposition \ref{P:BBesovembedding}. The bound \eqref{RYRtildeYHolder} is proved analogously, from the fact that
	\begin{align*}
		\abs{ \delta(R^Y - R^{\tilde Y})_{sut}}
		&\le |\delta Y'_{su} - \delta \tilde Y'_{su}||\delta X_{ut}| + |\delta \tilde Y'_{su}| |\delta X_{ut} - \delta \tilde X_{ut}|\\
		&\lesssim_{\alpha,p,q} \left( [Y' - \tilde Y']_{B^\alpha_{pq}} [X]_{B^\alpha_{pq}} + [\tilde Y']_{B^\alpha_{pq}} [X - \tilde X]_{B^\alpha_{pq}} \right) (u-s)^{\alpha-1/p} (t-u)^{\alpha - 1/p}.
	\end{align*}
	
	Lemma \ref{L:interpolate} and \eqref{RYHolder} now give
	\[
		\nor{R^Y}{\mbb B^\alpha_{p,q}} \lesssim_{\alpha,\beta,p,q} T^{\beta - \alpha - \frac{1}{p}}\nor{R^Y}{\mbb C^{\beta - 1/p}}^{1/2} \nor{R^Y}{\mbb B^\beta_{p/2,q}} 
		\lesssim_{\alpha,\beta,p,q} T^{\beta - \alpha - \frac{1}{p} } \left( \nor{R^Y}{\mbb B^{\beta}_{p/2,q/2} } + [Y']_{B^\alpha_{pq}} [X]_{B^\alpha_{pq}}\right),
	\]
	and so \eqref{Ycontrol} follows from the estimate (using Proposition \ref{P:Besovembedding})
	\[
		[Y]_{B^\alpha_{pq}} \lesssim_{q} \nor{Y'}{\oo} [X]_{B^\alpha_{pq}} + \nor{R^Y}{\mbb B^{\alpha}_{pq}}
		\lesssim_{\alpha,p,q} |Y_0| [X]_{B^\alpha_{pq}} + T^{\alpha - 1/p} [X]_{B^\alpha_{pq}} +\nor{R^Y}{\mbb B^{\alpha}_{pq}}.
	\]
	The bound \eqref{Ydistancecontrol} is obtained through virtually identical arguments, with the use of \eqref{RYRtildeYHolder}.
\end{proof}

For a rough path $\mbf X$, a controlled rough path $(Y,Y')$, and a partition $P = \{0 = \tau_0 < \tau_1 < \cdots < \tau_N = 1\}$ of $[0,1]$, we define
\[
	\mathscr I^P_{st} = \sum_{i=1}^N \left( Y_{s + \tau_{i-1}(t-s)} \delta X_{s + \tau_{i-1}(t-s), s + \tau_i (t-s)}
	+ Y'_{s +\tau_{i-1}(t-s)} \mbb X_{s + \tau_{i-1}(t-s), s + \tau_i (t-s)} \right),
\]
with the aim to construct the integral of $Y$ against $\mbf X$:
\[
	Z_t - Z_s = \int_s^t Y_r d \mbf X_r := \lim_{\norm{P} \to 0} \mathscr I^P_{s,t} \quad \text{for } (s,t) \in \Delta_2(0,T).
\]
Through the remainder of the section, we set
\begin{equation}\label{level2omega}
	\omega(\tau) := 
	\begin{dcases}
		\tau^{3\alpha} & \text{if } \alpha > 1/3 \text{ and}\\
		\tau \ell_{q/3}(\tau) & \text{if } \alpha = 1/3,
	\end{dcases}
\end{equation}
where $\ell_{q/3}$ satisfies \eqref{loglike}.

\begin{theorem}\label{T:Besovrough}
	Assume \eqref{level2parameters}, $\mbf X \in \mbf B^\alpha_{pq}([0,T],\RR^n)$, and $(Y,Y') \in \mathscr B^\alpha_{pq,\mbf X}([0,T],\RR^m \otimes \RR^n)$. Then there exists $(Z,Z') \in \mathscr B^\alpha_{pq,\mbf X}([0,T],\RR^m)$ such that $Z' = Y$,
	\begin{equation}\label{roughint}
		\begin{dcases}
		\lim_{\norm{P} \to 0} \nor{\mathscr I^P - \delta Z}{\mbb B^{\omega}_{p/3,r}} = 0 & \text{if } \alpha > \frac{1}{3}, \text{ for all } r > \frac{1}{3\alpha} \text{ and } r \ge \frac{q}{3}, \text{ and}\\
		\lim_{\norm{P} \to 0} \nor{\mathscr I^P - \delta Z}{\mbb B^{1}_{p/3,\oo}} = 0 & \text{if } \alpha = \frac{1}{3},
		\end{dcases}
	\end{equation}
	\begin{equation}\label{roughintegralremainder}
		\nor{ \delta Z - Y \delta X - Y' \mbb X}{\mbb B^\omega_{p/3,q/3}([0,T])}
		\lesssim_{\alpha,p,q} \nor{R^Y}{\mbb B^{2\alpha}_{p/2,q/2}} [X]_{B^\alpha_{pq}} + [Y']_{B^\alpha_{pq}} \nor{\mbb X}{\mbb B^{2\alpha}_{p/2,q/2}},
	\end{equation}
	\begin{equation}\label{roughintegralendpoint}
		\left\{
		\begin{split}
		&\text{if } \alpha = 1/3, \text{ then }  \delta Z - Y \delta X - Y' \mbb X \in \mbb B^1_{p/3,\oo;\circ}([0,T]) \text{ and}\\
		&\nor{ \delta Z - Y \delta X - Y' \mbb X}{\mbb B^1_{p/3,\oo}([0,T])}
		\lesssim_{\alpha,p,q} \nor{R^Y}{\mbb B^{2\alpha}_{p/2,q/2}} [X]_{B^\alpha_{pq}} + [Y']_{B^\alpha_{pq}} \nor{\mbb X}{\mbb B^{2\alpha}_{p/2,q/2}},
		\end{split}
		\right.
	\end{equation}
	and
	\begin{equation}\label{controlledintegralpathremainder}
		\nor{R^Z}{\mbb B^{2\alpha}_{p/2,q/2}} \lesssim_{\alpha,p,q} |Y'_0| \nor{\mbb X}{\mbb B^{2\alpha}_{p/2,q/2}} +  \frac{\omega(T)}{T^{2\alpha + 1/p}}  [(Y,Y')]_{\mathscr B^\alpha_{pq,\mbf X}} \left( \normm{\mbf X}_{\mbf B^\alpha_{pq}} \vee  \normm{\mbf X}_{\mbf B^\alpha_{pq}}^2\right).
	\end{equation}
\end{theorem}

\begin{proof}
For $(s,t) \in \Delta_2(0,T)$, set $A_{st} := Y_s(X_t - X_s) + Y'_s \mbb X_{st}$. Then, for all $r < s < t$,
\[
	\delta A_{rst} = -R^Y_{rs} (X_t - X_s) + (Y'_r - Y'_s) \mbb X_{st},
\]
and so
\[
	\nor{\delta A}{\oline{\mbb B}^{3\alpha}_{p/3,q/3}} \le \nor{R^Y}{\mbb B^{2\alpha}_{p/2,q/2}} [X]_{B^\alpha_{p,q}} + [Y']_{B^\alpha_{pq}} \nor{\mbb X}{B^{2\alpha}_{p/2,q/2}}.
\]
Lemma \ref{L:roughinterpolate} yields
\begin{align*}
	\nor{A}{\mbb B^\alpha_{pq}} \lesssim_{p,q} \nor{Y}{\oo} [X]_{B^\alpha_{pq}} + \nor{Y'}{\oo} \nor{\mbb X}{\mbb B^\alpha_{pq}}
	\lesssim_{\alpha,p,q}\nor{Y}{\oo} [X]_{B^\alpha_{pq}} +T^{\alpha-1/p} \nor{Y'}{\oo}\normm{\mbf X}_{\mbf B^\alpha_{pq}}^2.
\end{align*}
Finally, we note that, by Proposition \ref{P:roughBesovembedding} and Lemma \ref{L:controllednorms},
\begin{align*}
	|\delta A_{rst}|
	&\le |R^Y_{rs}| |X_t - X_s| + |Y'_r - Y'_s| |\mbb X_{st}|\\
	&\le \left(\nor{R^Y}{\mbb C^{2\alpha - 1/p}} [X]_{C^{\alpha - 1/p}} + [Y']_{C^{\alpha - 1/p}} \nor{\mbb X}{\mbb C^{2(\alpha - 1/p)}} \right)  \left[ (s-r) \vee (t-s) \right]^{2(\alpha-1/p)} \left[ (s-r) \wedge (t-s) \right]^{\alpha - 1/p}\\
	&\lesssim_{\alpha,p,q} M \left[ (s-r) \vee (t-s) \right]^{2(\alpha-1/p)} \left[ (s-r) \wedge (t-s) \right]^{\alpha - 1/p}
\end{align*}
where
\[
	M := \nor{R^Y}{\mbb B^{2\alpha}_{p/2,q/2}}[X]_{B^\alpha_{pq}} + [Y']_{B^\alpha_{pq}} \normm{\mbf X}_{\mbf B^\alpha_{pq}}^2.
\]
The convergence statement \eqref{roughint} and the bounds \eqref{roughintegralremainder} and \eqref{roughintegralendpoint} then follow from Theorems \ref{T:Besovsewing} and \ref{T:Besovsewingendpoint}, (as well as Remark \ref{R:RAsmallT}) because, with the notation of those theorems, 
\[
	\mathscr RA_{st} = \delta Z_{st} - Y_s \delta X_{st} - Y'_s \mbb X_{st} \quad \text{for } (s,t) \in \Delta_2(0,T).
\]
Moreover, because $3\alpha - \frac{3}{p} > \alpha - \frac{1}{p}$, Theorem \ref{T:Besovsewingcontinuous} implies that $Z \in B^\alpha_{pq}$, as well as
\begin{align*}
	\sup_{(s,t) \in \Delta_2(0,T)} \frac{|\mathscr RA_{st}|}{\omega(t-s)(t-s)^{-3/p}}
	&\lesssim_{\alpha,p,q} \nor{\delta A}{\oline{\mbb B}^{3\alpha}_{p/3,q/3}} + M 
	\lesssim_{\alpha,p,q} \nor{R^Y}{\mbb B^{2\alpha}_{p/2,q}} [X]_{B^\alpha_{pq}} + [Y']_{B^\alpha_{pq}}  \normm{\mbf X}_{\mbf B^\alpha_{pq}}^2\\
	&\lesssim [(Y,Y')]_{\mathscr B^\alpha_{pq,\mbf X}} \left( \normm{\mbf X}_{\mbf B^\alpha_{pq}} \vee  \normm{\mbf X}_{\mbf B^\alpha_{pq}}^2\right).
\end{align*}
Lemma \ref{L:interpolatemodulus} (with $\rho(\tau) = \omega(\tau)/\tau^{p/3}$) yields
\[
	\nor{\mathscr RA}{\mbb B^{2\alpha}_{p/2,q/2}} \lesssim_{\alpha,p,q} \frac{\omega(T)}{T^{2\alpha + 1/p}} \nor{\mathscr RA}{\mbb C^{\rho}}^{1/3} \nor{\mathscr RA}{\mbb B^{\omega}_{p/3,q/3}}^{2/3}
	\lesssim_{\alpha,p,q} \frac{\omega(T)}{T^{2\alpha + 1/p}}   [(Y,Y')]_{\mathscr B^\alpha_{pq,\mbf X}} \left( \normm{\mbf X}_{\mbf B^\alpha_{pq}} \vee  \normm{\mbf X}_{\mbf B^\alpha_{pq}}^2\right).
\]
The bound \eqref{controlledintegralpathremainder} now follows from writing $R^Z = \mathscr RA + Y' \mbb X$, and the fact that, by Proposition \ref{P:Besovembedding},
\[
	\nor{Y' \mbb X}{\mbb B^{2\alpha}_{p/2,q/2}} \le \nor{Y'}{\oo} \nor{\mbb X}{\mbb B^{2\alpha}_{p/2,q/2}}
	\lesssim_{\alpha,p,q} |Y'_0| \nor{\mbb X}{\mbb B^{2\alpha}_{p/2,q/2}} + T^{\alpha - 1/p} [Y']_{B^\alpha_{pq}} \nor{\mbb X}{\mbb B^{2\alpha}_{p/2,q/2}}.
\]
\end{proof}

\begin{theorem}\label{T:stability}
	Assume \eqref{level2parameters}, let $\mbf X, \tilde{\mbf X} \in \mbf B^\alpha_{pq}$, $(Y,Y') \in \mathscr B^\alpha_{pq, \mbf X}$, and $(\tilde Y, \tilde Y') \in \mathscr B^\alpha_{pq, \tilde{\mbf X}}$, fix $M > 0$, assume
	\[
		\left( |Y_0| + |Y'_0| + \nor{(Y,Y')}{\mathscr B^\alpha_{pq,\mbf X}} \right) \vee \left( |\tilde Y_0| + |\tilde Y'_0| + \nor{(\tilde Y,\tilde Y')}{\mathscr B^\alpha_{pq,\tilde{\mbf X}}}\right) 
		\vee \normm{\mbf X}_{\mbf B^\alpha_{pq}} \vee \normm{\tilde{\mbf X}}_{\mbf B^\alpha_{pq}}\le M,
	\]
	and define
	\[
		Z_t = \int_0^t Y_s d\mbf X_s \quad \text{and} \quad \tilde Z_t = \int_0^t \tilde Y_s d \tilde{\mbf X}_s.
	\]
	Then, if $\alpha > 1/3$ and $\left( \alpha + \frac{1}{p} \right) \vee (1-\alpha) < \beta \le 2\alpha$,
	\begin{align*}
		d_{\mbf X, \tilde{\mbf X}, \mathscr B^\alpha_{pq}}((Z,Z'), (\tilde Z, \tilde Z'))
		&\lesssim_{M,T} \rho_{\mbf B^\alpha_{pq}}(\mbf X, \tilde{\mbf X}) + |Y_0 - \tilde Y_0| + |Y'_0 - \tilde Y'_0|\\
		&+ T^{\beta - \alpha - 1/p} \left( [Y' - \tilde Y']_{B^\alpha_{pq}} + \nor{R^Y - R^{\tilde Y}}{\mbb B^\beta_{p/2,q/2}} \right),
	\end{align*}
	and, if $\alpha = 1/3$, 
	\begin{align*}
		d_{\mbf X, \tilde{\mbf X}, \mathscr B^{1/3}_{pq}}((Z,Z'), (\tilde Z, \tilde Z'))
		&\lesssim_{p,q,M,T} \rho_{\mbf B^{1/3}_{pq}}(\mbf X, \tilde{\mbf X}) + |Y_0 - \tilde Y_0| + |Y'_0 - \tilde Y'_0|\\
		&+T^{1/3 - 1/p}\ell_{q/3}(T) d_{\mbf X, \tilde{\mbf X}, \mathscr B^{1/3}_{pq} }((Y,Y'), (\tilde Y, \tilde Y')).
	\end{align*}
\end{theorem}

\begin{proof}
	For ease of presentation, define $\beta = 2/3$ whenever $\alpha = 1/3$. Proposition \ref{P:Besovembedding} and Lemma \ref{L:controllednorms} give
	\[
		\nor{Y}{\oo} \vee \nor{\tilde Y}{\oo} \vee \nor{Y'}{\oo} \vee \nor{\tilde Y'}{\oo} \lesssim_{\alpha,p,q,M} 1,
	\]
	\[
		\nor{Y' - \tilde Y'}{\oo} \lesssim_{\alpha,p,q} |Y'_0 - \tilde Y'_0| + T^{\alpha - 1/p} [Y' - \tilde Y']_{B^\alpha_{pq}},
	\]
	and
	\begin{align*}
		\nor{Y - \tilde Y}{\oo} 
		&\lesssim_{\alpha,p,q} |Y_0 - \tilde Y_0| + T^{\alpha - 1/p} [Y - \tilde Y]_{B^\alpha_{pq}}\\
		&\lesssim_{\alpha,p,q,M,T} |Y_0 - \tilde Y_0|  + |Y_0' - \tilde Y_0'| + [X - \tilde X]_{B^\alpha_{pq}} + T^{ \alpha - 1/p} \left( [Y' - \tilde Y']_{B^\alpha_{pq}} + \nor{R^Y - R^{\tilde Y}}{\mbb B^\beta_{p/2,q/2}} \right),
	\end{align*}
	and Lemma \ref{L:roughinterpolate} yields $\nor{\mbb X}{\mbb B^\alpha_{pq}([0,T])} \lesssim_{\alpha,p,q,M,T} 1$ and
	\[
		\nor{\mbb X - \tilde{\mbb X}}{\mbb B^\alpha_{pq}([0,T])} \lesssim_{\alpha,p,q,M,T} \rho_{\mbf B^\alpha_{pq}}(\mbf X, \tilde{\mbf X}).
	\]
	For $(s,t) \in \Delta_2(0,T)$, we set $A_{st} = Y_s (X_t - X_s) + Y'_s \mbb X_{st}$ and $\tilde A_{st} = \tilde Y_s (\tilde X_t - \tilde X_s) + \tilde Y'_s \tilde{\mbb X}_{st}$, and write
	\[
		A_{st} - \tilde A_{st}
		= (Y_s - \tilde Y_s) \delta X_{st} + \tilde Y_s (\delta X_{st} - \delta \tilde X_{st}) + (Y'_s - \tilde Y'_s) \mbb X_{st} + \tilde Y'_s (\mbb X_{st} - \tilde{\mbb X}_{st}).
	\]
	Then
	\begin{align*}
		\nor{A - \tilde A}{\mbb B^\alpha_{pq}}
		&\lesssim_{q} \nor{Y - \tilde Y}{\oo} [X]_{B^\alpha_{pq}} + \nor{\tilde Y}{\oo} [X - \tilde X]_{B^\alpha_{pq}} + \nor{Y' - \tilde Y'}{\oo} \nor{\mbb X}{\mbb B^\alpha_{pq}} + \nor{\tilde Y'}{\oo} \nor{\mbb X - \tilde{\mbb X}}{\mbb B^\alpha_{pq}} \\
		&\lesssim_{\alpha,p,q,M,T} |Y_0 - \tilde Y_0| + |Y'_0 - \tilde Y'_0| + \rho_{\mbf B^\alpha_{pq}}(\mbf X, \tilde{\mbf X}) + T^{\alpha - 1/p} \left( [Y' - \tilde Y']_{B^\alpha_{pq}} + \nor{R^Y - R^{\tilde Y}}{\mbb B^\beta_{p/2,q/2} } \right).
	\end{align*}
	For $(r,s,t) \in \Delta_3(0,T)$,
	\begin{align*}
		\delta (A - \tilde A)_{rst} &= -R^Y_{rs}(X_t - X_s) + R^{\tilde Y}_{rs}(\tilde X_t - \tilde X_s) + (Y'_r - Y'_s)\mbb X_{st} - (\tilde Y'_r - \tilde Y'_s) \tilde{\mbb X}_{st},
	\end{align*}
	so
	\[
		\nor{\delta A - \delta \tilde A}{\oline{\mbb B}^{\alpha + \beta}_{p/3,q/3} }
		\lesssim_{\alpha,p,q,M} [X - \tilde X]_{B^\alpha_{pq}} + \nor{\mbb X - \tilde{\mbb X}}{\mbb B^{2\alpha}_{p/2,q/2}} + [Y' - \tilde Y']_{B^\alpha_{pq}} + \nor{R^Y - R^{\tilde Y}}{\mbb B^\beta_{p/2,q/2}},
	\]
	and, by Lemma \ref{L:controllednorms},
	\begin{align*}
		\abs{ \delta (A - \tilde A)_{rst} }
		&\lesssim_{M,\alpha,p,q} \left( [X - \tilde X]_{B^\alpha_{pq}} + [Y' - \tilde Y']_{B^\alpha_{pq}} + \nor{R^Y - R^{\tilde Y}}{\mbb B^\beta_{p/2,q/2}} \right)(s-r)^{\beta - 2/p} (t-s)^{\alpha - 1/p} \\
		&+ \left( [Y' - \tilde Y']_{B^\alpha_{pq}} + \nor{\mbb X - \tilde{\mbb X} }{\mbb B^{2\alpha}_{p/2,q/2}} + [X - \tilde X]_{B^\alpha_{pq}} \right)(s-r)^{\alpha - 1/p} (t-s)^{2(\alpha - 1/p)}.
	\end{align*}
	When $\alpha > 1/3$, we have $\alpha + \beta - \frac{3}{p} > \alpha - 1/p$, and thus, by Theorems \ref{T:Besovsewing} and \ref{T:Besovsewingcontinuous},
	\begin{align*}
		\nor{\mathscr R(A - \tilde A)}{\mbb B^{\alpha + \beta}_{p/3,q/3}}
		&+ \nor{\mathscr RA - \mathscr R\tilde A}{\mbb C^{\alpha + \beta - 3/p}} 
		\lesssim_{\alpha,p,q,M} \rho_{\mbf B^\alpha_{pq} }(\mbf X, \tilde{\mbf X}) + [Y' - \tilde Y']_{B^\alpha_{pq}} + \nor{R^Y - R^{\tilde Y}}{\mbb B^\beta_{p/2,q/2}},
	\end{align*}
	and then Lemma \ref{L:interpolate} gives
	\begin{align*}
		\nor{\mathscr R(A - \tilde A)}{\mbb B^{2\alpha}_{p/2,q/2}([0,T])} 
		&\lesssim_{\alpha,p,q}
		T^{\beta - \alpha - 1/p} \nor{\mathscr R(A - \tilde A)}{\mbb B^{\alpha + \beta}_{p/3,q/3}}^{2/3} \nor{\mathscr RA - \mathscr R\tilde A}{\mbb C^{\alpha + \beta - 3/p}}^{1/3}\\
		&\lesssim_{\alpha,p,q,M,T}  \rho_{\mbf B^\alpha_{pq} }(\mbf X, \tilde{\mbf X}) + T^{\beta - \alpha - 1/p} \left( [Y' - \tilde Y']_{B^\alpha_{pq}} + \nor{R^Y - R^{\tilde Y}}{\mbb B^\beta_{p/2,q/2}} \right).
	\end{align*}
	When $\alpha = 1/3$, Theorems \ref{T:Besovsewingendpoint} and \ref{T:Besovsewingcontinuous} yield
	\begin{align*}
		\nor{\mathscr R(A - \tilde A)}{\mbb B^{\omega}_{p/3,q/3}}
		&+ \sup_{(s,t) \in \Delta_2(0,T)} \frac{ |(\mathscr RA - \mathscr R\tilde A)_{st}|}{\omega(t-s)(t-s)^{-3/p}} \\
		&\lesssim_{p,q,M} [X - \tilde X]_{B^{1/3}_{pq}} + \nor{\mbb X - \tilde{\mbb X}}{\mbb B^{2/3}_{p/2,q/2}} + [Y' - \tilde Y']_{B^{1/3}_{pq}} + \nor{R^Y - R^{\tilde Y}}{\mbb B^{2/3}_{p/2,q/2}}
	\end{align*}
	so that Lemma \ref{L:interpolatemodulus} gives
	\begin{align*}
		\nor{\mathscr R(A - \tilde A)}{\mbb B^{2/3}_{p/2,q/2}([0,T])} 
		\lesssim_{p,q,M,T} \rho_{\mbf B^{1/3}_{pq} }(\mbf X, \tilde{\mbf X}) + T^{\frac{1}{3} - \frac{1}{p}} \ell_{q/3}(T) d_{\mbf X, \tilde{\mbf X}, \mathscr B^{1/3}_{pq}}((Y,Y'), (\tilde Y, \tilde Y')).
	\end{align*}
	The desired estimate for $d_{\mbf X, \tilde{\mbf X},\mathscr B^\alpha_{pq}}((Z,Z'), (\tilde Z, \tilde Z'))$ now follows from the fact that $Z' = Y$ and $\tilde Z' = \tilde Y$, which, by Lemma \ref{L:controllednorms}, gives
	\[
		[Z' - \tilde Z']_{B^\alpha_{pq}} = [Y - \tilde Y]_{B^\alpha_{pq}} \lesssim_{M,T,\alpha,p,q} |Y'_0 - \tilde Y'_0| + [X - \tilde X]_{B^\alpha_{pq}} + T^{\beta - \alpha - 1/p} \left( [Y' - \tilde Y']_{B^\alpha_{pq}} + \nor{R^Y - R^{\tilde Y}}{\mbb B^\beta_{p/2,q/2}} \right),
	\]
	as well as the equality $R^Z - R^{\tilde Z} = Y' \mbb X - \tilde Y' \tilde{\mbb X} + \mathscr R(A - \tilde A)$ and the estimate obtained above for $\mathscr R(A - \tilde A)$ in $\mbb B^{2\alpha}_{p/2,q/2}$.
\end{proof}

\begin{proposition}\label{P:composecontrolled}
	Assume \eqref{level2parameters} and let $\mbf X \in \mbf B^\alpha_{pq}([0,T],\RR^n)$, $Y \in \mathscr B^\alpha_{pq,\mbf X}([0,T],\RR^m)$, and $f \in C^2(\RR^m)$. Then $(f(Y), f(Y)') \in \mathscr B^\alpha_{pq,\mbf X}$, where $f(Y)' = Df(Y)Y'$, and
	\[
		[(f(Y), f(Y)')]_{\mathscr B^\alpha_{pq,\mbf X}}
		\lesssim_{\alpha,p,q,T} \nor{f}{C^2}\left(1 + [X]_{B^\alpha_{pq}}\right) \left[\left( |Y'_0| + [(Y,Y')]_{\mathscr B^\alpha_{pq,\mbf X}} \right) \vee \left( |Y'_0| + [(Y,Y')]_{\mathscr B^\alpha_{pq,\mbf X}} \right)^2 \right].
	\]
	If $\tilde{\mbf X} \in \mbf B^\alpha_{pq}$, $\tilde Y \in \mathscr B^\alpha_{pq,\tilde{\mbf X}}$, $\delta \in (0,1]$, $f \in C^{2,\delta}$, and
	\[
		\left( |Y'_0| + [(Y,Y')]_{\mathscr B^\alpha_{pq,\mbf X}} \right)\vee \left( |\tilde Y'_0| + [(\tilde Y,\tilde Y')]_{\mathscr B^\alpha_{pq,\tilde{\mbf X}}} \right)\le M \quad \text{and} \quad [X]_{B^\alpha_{pq}} \vee [\tilde X]_{B^\alpha_{pq}} \le M,
	\]
	then
	\begin{align*}
		[f(Y)' - f(\tilde Y)']_{B^\alpha_{pq}} \lesssim_{\alpha,p,q,M,T} 
		\nor{f}{C^2} \left( [X - \tilde X]_{B^\alpha_{pq}} +  |Y'_0 - \tilde Y'_0| + d_{\mbf X, \tilde{\mbf X}, \mathscr B^\alpha_{pq}}(Y,Y',\tilde Y, \tilde Y')\right)
	\end{align*}
	and
	\begin{align*}
		\nor{R^{f(Y)} - R^{f(\tilde Y)} }{\mbb B^{(1+\delta)\alpha}_{p/2,q/2}}
		\lesssim_{\alpha,p,q,\delta,M,T} \nor{f}{C^{2,\delta}}\left( |Y_0 - \tilde Y_0| + |Y'_0 - \tilde Y'_0| + d_{\mbf X, \tilde{\mbf X}, \mathscr B^\alpha_{pq}}(Y,Y',\tilde Y, \tilde Y') + [X - \tilde X]_{B^\alpha_{pq}}\right).
	\end{align*}
\end{proposition}

\begin{proof}
	Set $Z := f(Y)$ and $Z' = Df(Y) Y'$. Then
	\begin{equation}\label{fofYremainder}
		\begin{split}
		R^Z_{st} &:= \delta Z_{st} - Z'_s \delta X_{st}\\
		&= \int_0^1 Df(\tau Y_t + (1-\tau) Y_s)d\tau R^Y_{st} 
		+ \int_0^1 \left[ Df(\tau Y_t + (1-\tau)Y_s) - Df(Y_s) \right] d\tau Y'_s (X_t - X_s) \\
		&= \int_0^1 Df(\tau Y_t + (1-\tau) Y_s)d\tau R^Y_{st} 
		+ \int_0^1(1-\tau) D^2f(\tau Y_t + (1-\tau)Y_s) d\tau (Y_t - Y_s) Y'_s (X_t - X_s).
		\end{split}
	\end{equation}
	For $\tau \in [0,T]$, we have
	\[
		\omega_p(Df(Y)Y', \tau) \le \nor{Df}{\oo} \omega_p(Y',\tau) + \nor{Y'}{\oo} \nor{D^2 f}{\oo} \omega_p(Y,\tau),
	\]
	from which it follows from Lemma \ref{L:controllednorms} that
	\begin{align*}
		[Df(Y)Y']_{B^\alpha_{pq}} &\lesssim_q \nor{Df}{\oo} [Y']_{B^\alpha_{pq}} + \nor{D^2f}{\oo} \nor{Y'}{\oo} [Y]_{B^\alpha_{pq}}\\
		&\lesssim_{\alpha,p,q,T} \nor{f}{C^2} \left( 1 + [X]_{B^\alpha_{pq}}\right)\left[\left( |Y'_0| + [(Y,Y')]_{\mathscr B^\alpha_{pq,\mbf X}} \right) \vee \left( |Y'_0| + [(Y,Y')]_{\mathscr B^\alpha_{pq,\mbf X}} \right)^2 \right]
	\end{align*}
	From the last line of \eqref{fofYremainder}, we obtain, using H\"older's inequality and Lemma \ref{L:controllednorms},
	\begin{align*}
		\nor{R^Z}{\mbb B^{2\alpha}_{p/2,q/2}}
		&\lesssim_q \nor{Df}{\oo} \nor{R^Y}{\mbb B^{2\alpha}_{p/2,q/2}} + \nor{D^2f}{\oo} \nor{Y'}{\oo} [Y]_{B^\alpha_{pq}} [X]_{B^\alpha_{pq}} \\
		&\lesssim_{\alpha,p,q,T} \nor{f}{C^2}\left(1 + [X]_{B^\alpha_{pq}}\right)\left[ \left(|Y'_0| + [(Y,Y')]_{\mathscr B^\alpha_{pq,\mbf X}} \right) \vee \left( |Y'_0| + [(Y,Y')]_{\mathscr B^\alpha_{pq,\mbf X}}\right)^2\right].
	\end{align*}
	We now set $\tilde Z := f(\tilde Y)$ and $\tilde Z' := Df(\tilde Y) \tilde Y'$, and note that, by Lemma \ref{L:controllednorms},
	\[
		[Y - \tilde Y]_{B^\alpha_{pq}} \lesssim_{\alpha,p,q,M,T} [X - \tilde X]_{B^\alpha_{pq}} + |Y'_0 - \tilde Y'_0| + d_{\mbf X, \tilde{\mbf X}, \mathscr B^\alpha_{pq}}(Y,Y',\tilde Y, \tilde Y'),
	\]
	so that
	\begin{align*}
		[f(Y)' - f(\tilde Y)']_{B^\alpha_{pq}}
		&= [Df(Y)Y' - Df(\tilde Y) \tilde Y']_{B^\alpha_{pq}}\\
		&\lesssim_q \nor{Df}{\oo} [Y' - \tilde Y']_{B^\alpha_{pq}} + [Df(Y) - Df(\tilde Y)]_{B^\alpha_{pq}}\nor{Y'}{\oo} \\
		&\lesssim_{\alpha,p,q,M} \nor{Df}{\oo} [Y' - \tilde Y']_{B^\alpha_{pq}} + \nor{D^2 f}{\oo}[Y - \tilde Y]_{B^\alpha_{pq}} \\
		&\lesssim_{\alpha,p,q,M} \nor{f}{C^2} [X - \tilde X]_{B^\alpha_{pq}} + \nor{f}{C^2}\left( |Y'_0 - \tilde Y'_0| + d_{\mbf X, \tilde{\mbf X}, \mathscr B^\alpha_{pq}}(Y,Y',\tilde Y, \tilde Y')\right).
		\end{align*}
	Next, defining $R^{\tilde Z}$ as in \eqref{fofYremainder}, with $Y$ and $X$ replaced with $\tilde Y$ and $\tilde X$, we write $R^Z - R^{\tilde Z} = \I + \II + \III$, where
	\[
		\I_{st} := \int_0^1 Df(\tau Y_t + (1-\tau)Y_s)d\tau R^Y_{st} - \int_0^1 Df(\tau \tilde Y_t + (1-\tau)\tilde Y_s)d\tau R^{\tilde Y}_{st},
	\]
	\begin{align*}
		\mathrm{II}_{st}
		&:= \int_0^1 \left[ Df(\tau \tilde Y_t + (1-\tau)\tilde Y_s) - Df(\tilde Y_s) \right] d\tau \left( Y'_s \delta X_{st} - \tilde Y'_s \delta \tilde X_{st} \right)\\
		&= \int_0^1 (1-\tau) D^2 f(\tau \tilde Y_t + (1-\tau)\tilde Y_s) d\tau \delta \tilde Y_{st} \left( Y'_s \delta X_{st} - \tilde Y'_s \delta \tilde X_{st} \right),
	\end{align*}
	and
	\begin{align*}
		\mathrm{III}_{st}
		&:= \left[ \int_0^1 \int_0^1 D^2 f(\sigma [\tau Y_t + (1-\tau) Y_s] + (1-\sigma) [\tau \tilde Y_t + (1-\tau) \tilde Y_s] ) d\sigma \left( \tau (Y_t - \tilde Y_t) + (1-\tau) (Y_s - \tilde Y_s)\right)d\tau \right. \\
		&-  \left. \int_0^1 D^2 f(\sigma Y_s + (1-\sigma) \tilde Y_s) d \sigma (Y_s - \tilde Y_s) \right] Y'_s \delta X_{st}.
	\end{align*}
	Lemma \ref{L:controllednorms} immediately yields
	\begin{align*}
		\nor{\I}{\mbb B^{2\alpha}_{p/2,q/2}} 
		&\lesssim_q \nor{Df}{\oo} \nor{R^Y - R^{\tilde Y}}{\mbb B^{2\alpha}_{p/2,q/2}} + \nor{D^2f}{\oo} \nor{Y - \tilde Y}{\oo}\nor{R^Y}{\mbb B^{2\alpha}_{p/2,q/2}} \\
		&\lesssim_{\alpha,p,q,M,T} \nor{f}{C^2} \left( |Y_0 - \tilde Y_0| + |Y'_0 - \tilde Y'_0| + d_{\mbf X, \tilde{\mbf X}, \mathscr B^\alpha_{pq}}((Y,Y'), (\tilde Y, \tilde Y'))\right) + \nor{D^2f}{\oo} [X - \tilde X]_{B^\alpha_{pq}}
	\end{align*}
	and
	\[
		\nor{\mathrm{II}}{\mbb B^{2\alpha}_{p/2,q/2}}
		\lesssim_{\alpha,p,q,M,T} \nor{D^2f}{\oo} \left( [Y' - \tilde Y']_{B^\alpha_{pq}} + [X - \tilde X]_{B^\alpha_{pq}}\right).
	\]
	We next write $\mathrm{III} := (\mathrm{III}^a + \mathrm{III}^b)Y' \delta X$, where, for $(s,t) \in \Delta_2(0,T)$,
	\[
		\mathrm{III}^a_{st} := \int_0^1\left( \int_0^1 D^2 f\left(\sigma [\tau Y_t + (1-\tau) Y_s] + (1-\sigma) [\tau \tilde Y_t + (1-\tau) \tilde Y_s]  \right)d\tau - D^2 f(\sigma Y_s + (1-\sigma) \tilde Y_s) \right)d\sigma (Y_s - \tilde Y_s)
	\]
	and
	\[
		\mathrm{III}^b_{st} := \int_0^1 \tau \int_0^1 D^2 f\left(\sigma [\tau Y_t + (1-\tau) Y_s] + (1-\sigma) [\tau \tilde Y_t + (1-\tau) \tilde Y_s] \right) d\sigma d\tau \left( \delta Y_{st} - \delta \tilde Y_{st}\right).
	\]
	We then have
	\[
		\abs{ \mathrm{III}^a_{st}} \lesssim [D^2 f]_{C^\delta} \left( |\delta Y_{st}|^\delta + |\delta \tilde Y_{st}|^\delta\right) \nor{Y - \tilde Y}{\oo}
	\quad \text{and} \quad
		\abs{ \mathrm{III}^b_{st}} \lesssim \nor{D^2f}{\oo} \abs{ \delta Y_{st} - \delta \tilde Y_{st}},
	\]
	and so, by H\"older's inequality,
	\[
		\Omega_{p/2}(\mathrm{III},\tau) \lesssim \nor{Y'}{\oo} \omega_p(X,\tau) \left[ [D^2 f]_{C^\delta} \left( \omega_p(Y,\tau)^\delta + \omega_p(\tilde Y, \tau)^\delta\right) \nor{Y - \tilde Y}{\oo} + \nor{D^2f}{\oo} \omega_p(Y - \tilde Y,\tau) \right].
	\]
	Thus, by Lemma \ref{L:controllednorms},
	\begin{align*}
		\nor{\mathrm{III}}{\mbb B^{(1+\delta)\alpha}_{p/2,q/2}}
		&\lesssim \nor{Y'}{\oo} [X]_{B^\alpha_{pq}} \left[ [D^2 f]_{C^\delta} \left( [Y]_{B^\alpha_{pq}}^\delta + [\tilde Y]_{B^\alpha_{pq}}^\delta\right) \nor{Y - \tilde Y}{\oo} + \nor{D^2f}{\oo} [Y - \tilde Y]_{B^\alpha_{pq}}\right]\\
		&\lesssim_{\alpha,p,q,M,T} \nor{f}{C^{2,\delta}} \left( |Y_0 - \tilde Y_0| + [Y - \tilde Y]_{B^\alpha_{pq}} \right) \\
		&\lesssim_{\alpha,p,q,M,T} \nor{f}{C^{2,\delta}}\left( |Y_0 - \tilde Y_0| + |Y'_0 - \tilde Y'_0| + d_{\mbf X, \tilde{\mbf X}, \mathscr B^\alpha_{pq}}(Y,Y',\tilde Y, \tilde Y') + [X - \tilde X]_{B^\alpha_{pq}}\right).
	\end{align*}
	Combining all estimates and using the fact that $\mbb B^{2\alpha}_{p/2,q/2} \subset \mbb B^{(1+\delta)\alpha}_{p/2,q/2}$ gives the result.
\end{proof}

\subsection{Rough differential equations with Besov signals: the level-$2$ case} \label{SS:BesovRDE}

We now consider the the initial value problem
\begin{equation}\label{E:RDE}
	dY_t = f(Y_t)\cdot d \mbf X_t \quad \text{in } [0,T] \quad \text{and} \quad Y_0 = y,
\end{equation}
where $y \in \RR^m$, $\mbf X \in \mbf B^\alpha_{pq}$ with $\alpha$, $p$, and $q$ satisfying the level-$2$ conditions \eqref{level2parameters}, and the nonlinearity satisfies
\begin{equation}\label{level2f}
	\begin{dcases}
	f \in C^{2,\delta}(\RR^m;\RR^m \otimes \RR^n) & \text{if } \alpha > 1/3, \text{ where } (2+\delta)\alpha > 1 \text{ and } \delta \alpha > 1/p,\\
	f \in C^3(\RR^n;\RR^m \otimes \RR^n) & \text{if } \alpha = \frac{1}{3}.
	\end{dcases}
\end{equation}
More precisely, we seek a unique $\mbf X$-controlled path $(Y,Y') \in \mathscr B^\alpha_{pq,\mbf X}([0,T],\RR^m)$ satisfying the integral relation
\begin{equation}\label{E:RDEintegral}
	Y_t = y + \int_0^t f(Y_s) \cdot d \mbf X_s.
\end{equation}

\begin{theorem}\label{T:RDE}
	Assume \eqref{level2parameters}, and fix $\mbf X \in \mbf B^\alpha_{pq}([0,T]; \RR^n)$ and $f$ satisfying \eqref{level2f}. Then, for every fixed $y \in \RR^m$, there exists a unique solution $(Y,Y') \in \mathscr B^\alpha_{pq,\mbf X}([0,T];\RR^m)$ of \eqref{E:RDE}. Moreover, there exists a constant $M$ depending only on $\alpha$, $p$, $q$, $T$, $\nor{f}{C^2}$, and $\normm{\mbf X}_{\mbf B^\alpha_{pq}}$ such that $[(Y,Y')]_{\mathscr B^\alpha_{pq,\mbf X}} \le M$.
\end{theorem}

\begin{proof}
	Define $\mathscr X := \left\{ (Y,Y') \in \mathscr B^\alpha_{pq,\mbf X}([0,T],\RR^m) : Y_0 = y , \; Y'_0 = f(y) \right\}$ and, for $(Y,Y') \in \mathscr X$,
	\[
		\mathscr T(Y,Y') := \left(  y + \int_0^\cdot f(Y_s) \cdot d\mbf X_s, f(Y) \right).
	\]
	Theorem \ref{T:Besovrough} and Proposition \ref{P:composecontrolled} imply that $\mathscr T$ is well-defined from $\mathscr X$ to $\mathscr X$. As per Remark \ref{R:Besovsup}, it is easy to see that $\mathscr X$ is a complete metric space under the metric $d_{\mbf X, \mathscr B^\alpha_{pq}}$.
	
	{\it Step 1.} Define $(\oline{Y} , \oline{Y}') \in \mathscr X$ by $\oline{Y}_t := y + f(y)X_t$ and $\oline{Y}'_t := f(y)$ for $t \in [0,T]$, and, for $M > 0$, define
	\[
		\mathscr X_M := \left\{ (Y,Y') \in \mathscr X : d_{\mathscr B^\alpha_{pq},\mbf X}((Y,Y'), (\oline{Y},\oline{Y}')) \le M \right\},
	\]
	which is again a complete metric space. We first show that, for $M > 0$ sufficiently large and $0 < T_0 \le 1$ sufficiently small, depending only on $\alpha$, $p$, $q$, $\nor{f}{C^2}$, and $\normm{\mbf X}_{\mbf B^\alpha_{pq}}$, $\mathscr T$ maps $\mathscr X_M$ into $\mathscr X_M$.
	
	Assume that $(Y,Y') \in \mathscr X_M$ and set $(Z,Z') := \mathscr T(Y,Y') \in \mathscr X$. By Lemmas \ref{L:Besovcomposition} and \ref{L:controllednorms},
	\begin{align*}
		[Z' - f(y)]_{B^\alpha_{pq}} 
		&= [f(Y)]_{B^\alpha_{pq}} \le \nor{Df}{\oo} [Y]_{B^\alpha_{pq}}\\
		&\lesssim_{\alpha,p,q} \nor{Df}{\oo} \left( |f(Y_0)| [X]_{B^\alpha_{pq}} + T^{\alpha - 1/p} \left( [Y']_{B^\alpha_{pq}}[X]_{B^\alpha_{pq}} + \nor{R^Y}{\mbb B^{2\alpha}_{p/2,q/2}} \right) \right) \\
		&\le \nor{Df}{\oo} \nor{f}{\oo} [X]_{B^\alpha_{pq}} + (1 + [X]_{B^\alpha_{pq}} ) M T^{\alpha - 1/p},
	\end{align*}
	and Theorem \ref{T:Besovrough} gives (recall the definition of $\omega$ from \eqref{level2omega})
	\begin{align*}
		\nor{R^Z}{\mbb B^{2\alpha}_{p/2,q/2}}
		&\lesssim_{\alpha,p,q} \nor{f}{\oo} \nor{\mbb X}{\mbb B^{2\alpha}_{p/2,q/2}} + [ (f(Y), f(Y)') ]_{\mathscr B^\alpha_{pq,\mbf X} } \left( \normm{\mbf X}_{\mbf B^\alpha_{pq}} \vee \normm{\mbf X}_{\mbf B^\alpha_{pq}}^2\right) \frac{\omega(T)}{T^{2\alpha + 1/p}} .
	\end{align*}
	By Proposition \ref{P:composecontrolled}, we find that
	\[
		[ (f(Y), f(Y)')]_{\mathscr B^\alpha_{pq,\mbf X}} \lesssim_{\alpha,p,q,T} \nor{f}{C^2}\left(1 + [X]_{B^\alpha_{pq}}\right)\left[ \left( |Y'_0| + [(Y,Y')]_{\mathscr B^\alpha_{pq,\mbf X}}\right) \vee \left( |Y'_0| + [(Y,Y')]_{\mathscr B^\alpha_{pq,\mbf X}}\right)^2\right].
	\]
	Combining these estimates gives, for some constant $C_1 = C_1\left(\alpha,p,q,\nor{f}{C^2}, \normm{\mbf X}_{\mbf B^\alpha_{pq}}\right) > 0$,
	\begin{equation}\label{aprioribound}
		[Z' - f(y)]_{B^\alpha_{pq}} + \nor{R^Z}{\mbb B^{2\alpha}_{p/2,q}}
		\le C_1\left( 1+ \frac{\omega(T_0)}{T_0^{2\alpha + 1/p}} (M \vee M^2)\right).
	\end{equation}
	We then set $M := 2C_1$, in which case \eqref{aprioribound} becomes
	\[
		[Z' - f(y)]_{B^\alpha_{pq}} + \nor{R^Z}{\mbb B^{2\alpha}_{p/2,q/2}}
		\le \frac{M}{2} + C_1(1 \vee 2C_1)\frac{\omega(T_0)}{T_0^{2\alpha + 1/p}} M.
	\]
	We then conclude by choosing $T_0 > 0$ sufficiently small that
	\[
		\frac{\omega(T_0)}{T_0^{2\alpha + 1/p}} \le \frac{1}{2C_1(1 \vee 2C_1)}.
	\]
	
	{\it Step 2.} We next show that, shrinking $T_0$ if necessary, $\mathscr T$ is a contraction on $\mathscr X_M$. Let $(Y,Y'), (\tilde Y,\tilde Y') \in \mathscr X_M$, and set $(Z,Z') = \mathscr T(Y,Y')$ and $(\tilde Z, \tilde Z') = \mathscr T(\tilde Y, \tilde Y')$. If $\alpha > 1/3$, then $\beta := (1+\delta)\alpha$ satisfies $\alpha + \beta > 1$ and $\beta > \alpha + 1/p$, and so, by Theorem \ref{T:stability}, 
	\[
		d_{\mbf X, \mathscr B^\alpha_{pq}}(Z,Z', \tilde Z, \tilde Z') \lesssim_{\alpha,\delta,p,q,\nor{f}{C^2}, \normm{\mbf X}_{\mbf B^\alpha_{pq}}} T_0^{\delta \alpha - 1/p}\left( [f(Y)' - f(\tilde Y)']_{B^\alpha_{pq}} + \nor{R^{f(Y)} - R^{f(\tilde Y)} }{\mbb B^{(1+\delta)\alpha}_{p/2,q/2}} \right).
	\]
	Proposition \ref{P:composecontrolled} then yields
	\[
		d_{\mbf X, \mathscr B^\alpha_{pq}}(Z,Z', \tilde Z, \tilde Z') \lesssim_{\alpha,\delta,p,q,\nor{f}{C^{2,\delta}}, \normm{\mbf X}_{\mbf B^\alpha_{pq}}}  T_0^{\delta \alpha - 1/p} d_{\mbf X, \mathscr B^\alpha_{pq}}(Y,Y', \tilde Y, \tilde Y'),
	\]
	and we conclude upon shrinking $T_0$ as needed. A similar argument holds for when $\alpha = 1/3$, in which case Theorem \ref{T:stability} and Proposition \ref{P:composecontrolled} instead give
	\[
		d_{\mbf X, \mathscr B^{1/3}_{pq}}(Z,Z', \tilde Z, \tilde Z') \lesssim_{p,q,\nor{f}{C^{3}}, \normm{\mbf X}_{\mbf B^{1/3}_{pq}}} T_0^{\alpha - 1/p}\ell_{q/3}(T_0) d_{\mbf X, \mathscr B^{1/3}_{pq}}(Y,Y', \tilde Y, \tilde Y').
	\]
	The fixed-point construction can then be iterated to build the unique solution on all of $[0,T]$.
\end{proof}

The notion of controlled-rough path solution of \eqref{E:RDE} can equivalently be characterized in the Davie sense. More precisely, the increment $Y_t - Y_s$, minus an appropriate level-$2$ expansion, belongs to a Besov space of appropriately higher-order regularity. 

\begin{proposition}\label{P:Davies}
	Assume \eqref{level2parameters}, $y \in \RR^m$, $f$ satisfies \eqref{level2f}, and $Y \in B^\alpha_{pq}([0,T],\RR^m)$. Then $(Y,Y')$ belongs to $\mathscr B^\alpha_{pq,\mbf X}$ and is a solution of \eqref{E:RDE} if and only if $Y' = f(Y)$ and 
	\[
		\mathscr D_{st} := Y_t - Y_s - f(Y_s)(X_t - X_s) - Df(Y_s)f(Y_s) \mbb X_{st} \quad \text{for } (s,t) \in \Delta_2(0,T)
	\]
	satisfies
	\begin{equation}\label{Davies}
		\mathscr D \in
		\begin{dcases}
		\mbb B^{3\alpha}_{p/3,q/3}([0,T],\RR^m) & \text{if } \alpha > \frac{1}{3} \text{ and}\\
		\mbb B^{1}_{p/3,\oo,\circ}([0,T], \RR^m) & \text{if } \alpha = \frac{1}{3}.
		\end{dcases}
	\end{equation}	
\end{proposition}

\begin{proof}
	Given a solution of \eqref{E:RDE} in the sense of rough integrals, the conclusion is an immediate consequence of \eqref{roughintegralremainder} and \eqref{roughintegralendpoint} from Theorem \ref{T:Besovrough}. Conversely, assume that $(Y,Y') = (Y, f(Y))$ and $\mathscr D$ satisfies \eqref{Davies}. Lemma \ref{L:Besovcomposition} gives $f(Y) \in B^\alpha_{pq}([0,T],\RR^m \otimes \RR^n)$. We then set
	\[
		R^Y_{st} := Y_t - Y_s - f(Y_s) (X_t - X_s) = Df(Y_s) f(Y_s) \mbb X_{st} + \mathscr D_{st} \quad \text{for } (s,t) \in \Delta_2(0,T).
	\]
	Because $\alpha > 1/p$, Propositions \ref{P:Besovembedding} and \ref{P:roughBesovembedding} imply that $\mathscr D \in \mbb C^{\alpha - 1/p}([0,T],\RR^m)$, and so, by Lemma \ref{L:interpolate}, $\mathscr D \in \mbb B^{2\alpha}_{p/2,q/2}$. The properties of $f$ and $Y$ give $(Df(Y_s) f(Y_s) \mbb X_{st})_{(s,t) \in \Delta_2(0,T)} \in \mbb B^{2\alpha}_{p/2,q/2}$, and we conclude that $R^Y \in B^{2\alpha}_{p/2,q/2}$, and, therefore, $(Y,Y') \in \mathscr B^\alpha_{pq,\mbf X}$.
	
	As in Theorem \ref{T:Besovrough}, we define 
	\[
		\tilde Y_t := y + \int_0^t f(Y_s) d \mbf X_s,
	\]
	and \eqref{roughintegralremainder} and \eqref{roughintegralendpoint} immediately give
	\[
		\delta \tilde Y - f(Y)\delta X- Df(Y)f(Y) \mbb X \in
		\begin{dcases}
		\mbb B^{3\alpha}_{p/3,q/3}([0,T],\RR^m) & \text{if } \alpha > \frac{1}{3} \text{ and}\\
		\mbb B^{1}_{p/3,\oo,\circ}([0,T], \RR^m) & \text{if } \alpha = \frac{1}{3}.
		\end{dcases}
	\]
	 It follows that 
	 \[
	 	Y - \tilde Y
		\in
		\begin{dcases}
		B^{3\alpha}_{p/3,q/3}([0,T],\RR^m) & \text{if } \alpha > \frac{1}{3} \text{ and}\\
		B^{1}_{p/3,\oo,\circ}([0,T], \RR^m) & \text{if } \alpha = \frac{1}{3},
		\end{dcases}
	\]
	and so Lemma \ref{L:gammalarger1} and the fact that $Y_0 = \tilde Y_0 = y$ give $Y = \tilde Y$, as desired.
\end{proof}

\begin{remark}\label{R:DaviesHolder}
	In the proof of Proposition \ref{P:Davies}, the $(\alpha -1/p)$-H\"older regularity of the Davie's remainder $\mathscr D$ was used in order to regain integrability (from $p/3$ to $p/2$), but the exact value of the H\"older exponent was not important. After the fact, we actually see that $\mathscr D \in \mbb C^{\rho}$, where $\rho(\tau) = \omega(\tau) \tau^{-3/p}$, which is a consequence of \eqref{RHolder} and \eqref{Rctsendpoint} from Theorem \ref{T:Besovsewingcontinuous}. 
\end{remark}

We finish this section by proving that the It\^o-Lyons map is locally Lipschitz continuous in the data.

\begin{theorem}\label{T:IL}
	Assume \eqref{level2parameters}, fix $M > 0$. and let $y,\tilde y \in \RR^m$, $f^1,f^2$ satisfying \eqref{level2f}, and $\mbf X^1,\mbf X^2 \in \mbf B^\alpha_{pq}$ be such that
	\[
		\bigvee_{i=1}^2 \nor{f^i}{C^{2,\delta}} \vee \normm{\mbf X^i}_{\mbf B^\alpha_{pq}} \le M,
	\]
	with $C^{2,\delta}$ replaced by $C^3$ when $\alpha = 1/3$. Then, for $i = 1,2$, the solutions $(Y^i,f(Y^i)) \in \mathscr B^\alpha_{pq,\mbf X^i}$ of \eqref{E:RDE} corresponding to respectively $y^i$, $f^i$, and $\mbf X^i$ satisfy 
	\[
		d_{\mbf X, \tilde{\mbf X}, \mathscr B^\alpha_{pq}}\left( (Y^1,f(Y^1)), (Y^2, f(Y^2))\right)
		\lesssim_{\alpha,p,q,M,T} \left( |y^1 - y^2| + \nor{f^1 - f^2}{C^{2,\delta} }  +  \rho_{\mbf B^\alpha_{pq}}(\mbf X, \tilde{\mbf X}) \right).
	\]
\end{theorem}

\begin{proof}
	By Theorem \ref{T:RDE}, we have $[(Y,Y')]_{\mathscr B^\alpha_{pq,\mbf X}}\vee [(\tilde Y, \tilde Y')]_{\mathscr B^\alpha_{pq,\tilde{\mbf X}}} \lesssim_{\alpha,p,q,M,T} 1$. Define now
	\[
		(W^i,(W^i)') := (f^i(Y^i), Df^i(Y^i)(Y^i)') \quad \text{for } i = 1,2 \quad \text{and} \quad (\tilde W, \tilde W') := (f^1(Y^2), Df^1(Y^2) (Y^2)').
	\]
	Theorem \ref{T:stability} now gives, for some non-increasing function $\sigma: [0,\oo) \to [0,\oo)$ satisfying $\lim_{\tau \to 0^+} \sigma(\tau) = 0$,
	\begin{align*}
		d_{\mbf X^1, \mbf X^2, \mathscr B^\alpha_{pq}}&(Y^1, f(Y^1), Y^2, f(Y^2))\\
		&\lesssim \rho_{\mbf B^\alpha_{pq}}(\mbf X^1,\mbf X^2) + |f^1(y^1) - f^2(y^2)| + |Df^1(y^1)f^1(y^1) - Df(y^2)f(y^2)|\\
		&+ \sigma (T)\left( [(W^1)' - (W^2)']_{B^\alpha_{pq}} + \nor{R^{W^1} - R^{W^2} }{\mbb B^{(1+\delta)\alpha}_{p/2,q}} \right) \\
		&\lesssim \rho_{\mbf B^\alpha_{pq}}(\mbf X^1,\mbf X^2) + |y^1 - y^2| \\
		&+\sigma(T)\left( [(W^1)' - \tilde W']_{B^\alpha_{pq}} +[\tilde W' - (W^2)']_{B^\alpha_{pq}} + \nor{R^{W^1} - R^{\tilde W}}{\mbb B^{(1+\delta)\alpha}_{p/2,q/2} } + \nor{R^{\tilde W} - R^{W^2} }{\mbb B^{(1+\delta)\alpha}_{p/2,q/2}} \right),
	\end{align*}
	with $\delta = 1$ when $\alpha = 1/3$. Proposition \ref{P:composecontrolled} yields
	\begin{align*}
		[(W^1)' - \tilde W']_{B^\alpha_{pq}}
		&= [f^1(Y^1)' - f^1(Y^2)']_{B^\alpha_{pq}}\\
		&\lesssim [X^1 - X^2]_{B^\alpha_{pq}} + |y^1-y^2| + d_{\mbf X, \tilde{\mbf X}, \mathscr B^\alpha_{pq}}(Y^1,f^1(Y^1), Y^2, f^2(Y^2)),
	\end{align*}
	\begin{align*}
		\nor{R^{W^1} - R^{\tilde W}}{\mbb B^{(1+\delta)\alpha}_{p/2,q/2}}
		&= \nor{R^{f^1(Y^1)} - R^{f^1(Y^2)}}{\mbb B^{(1+\delta)\alpha}_{p/2,q/2}}\\
		&\lesssim |y^1 - y^2| + [X^1 - X^2]_{B^\alpha_{pq}} + d_{\mbf X, \tilde{\mbf X}, \mathscr B^\alpha_{pq}}(Y^1,f^1(Y^1),Y^2,f^2(Y^2)),
	\end{align*}
	and
	\[
		[\tilde W' - (W^2)']_{B^\alpha_{pq}} + \nor{R^{\tilde W} - R^{W^2}}{\mbb B^{(1+\delta)\alpha}_{p/2,q/2}}
		= [(f^1-f^2)(Y^2)]_{B^\alpha_{pq}} + \nor{R^{(f^1 - f^2)(Y^2)}}{\mbb B^{(1+\delta)\alpha}_{p/2,q/2}}
		\lesssim \nor{f^1 - f^2}{C^2}.
	\]
	Combining all terms, we conclude that
	\begin{align*}
		d_{\mbf X^1, \mbf X^2, \mathscr B^\alpha_{pq}}(Y^1, f(Y^1), Y^2, f(Y^2))
		&\lesssim |y^1 - y^2| + \nor{f^1 - f^2}{C^2} + \rho_{\mbf B^\alpha_{pq}}(\mbf X^1,\mbf X^2)\\
		&+ \sigma(T_0) d_{\mbf X^1, \mbf X^2,\mathscr B^\alpha_{pq}}(Y^1, f^1(Y^1), Y^2, f^2(Y^2)).
	\end{align*}
	It follows that, if $T_0$ is sufficiently small, then, upon rearranging terms, we have the desired result on $[0,T_0]$. The estimate can be extended iteratively to the rest of $[0,T]$ as before.
\end{proof}

\subsection{Beyond level-$2$ Besov rough paths} \label{sec:beyond}

When $N \ge 3$, solving nonlinear differential equations driven by level-$N$ rough paths inevitably leads to iterated integrals with branching, as, for example,
\begin{equation}\label{branchnobranch}
        \int ( \delta X^i ) ( \delta X^j )  d X^k,  \qquad \text{vs. the non-branching }  \iiint d X^i d X^j d X^k,
\end{equation}
which are not meant to be well-defined, but rather are part of the augmented information supplied by some rough path. The second, non-branching term in \eqref{branchnobranch} is precisely contained in the $(\RR^n)^{\otimes 3}$-valued third level tier of $\mbf X$, as introduced in Definition \ref{D:BesovRP}.
On the other hand, 
%
%
%
%
%
 this is not the case for the first, branching, information. The situation can be resolved by imposing a chain-rule, in this specific example,
 \[
 	( \delta X^i ) ( \delta X^j ) = \iint d X^i d X^j  + \iint d X^j d X^i,
\]
which leads to the notion of {\em geometric} rough path. Alternatively, one can work with {\em branched} rough paths, where the state space $T^{(N)}_1(\RR^n)$ is replaced by a (truncated) Hopf algebra of trees that allows to encode the full branching information. A complete theory of branched rough paths and differential equations in the H\"older setting is found in \cite{gubinelli2010ramification,hairer2015geometric}, and in a c\`adl\`ag $p$-variation setting in \cite{friz2018differential}. Although the branched setting is at first sight more involved (one needs to introduce the Connes--Kreimer and Grossmann-Larson Hopf algebras), the absence of algebraic constraints ultimately leads to simplifications, notably when it comes to establishing the stability of controlled rough paths under composition with regular functions. Perhaps for this reason, a detailed discussion of level-$N$ geometric rough differential equations in the controlled rough paths perspective is fairly recent; see for instance \cite{BDFT} and \cite{BoeGen}. (For a development of geometric rough paths theory without controlled structures, following Davie's direct approach, we refer to \cite{friz2010multidimensional}.) 

For the sake of brevity, we discuss in what follows only the geometric case. 
Adaptions to a branched Besov rough paths setting are then straightforward (and easier in a sense).\footnote{As in \cite{hairer2015geometric} one could also go on to embed the branched case into a suitably extended geometric setting, though this will not (easily) lead to optimal estimates.} As in Definition \ref{D:BesovRP}, throughout this sub-section, we fix $\alpha$, $p$, and $q$ satisfying \eqref{levelNparameters}.

\begin{definition}\label{D:BesovRPgeo}
A level-$N$ Besov rough path $\mbf X \in \mbf B^\alpha_{pq}([0,T])$, in the sense of Definition \ref{D:BesovRP}, is called {\em geometric} if, for all $(s,t) \in \Delta_2(0,T)$, 
$$
           \mbf X_{s,t} \in G^{(N)}(\RR^n)    \subset T^{(N)}_1(\RR^n),
$$
where $G^{(N)}(\RR^n)$ is the free step-$N$ nilpotent group, a.k.a (truncated) character group of the shuffle Hopf-algebra. (This precisely encodes the chain rule.)
\end{definition}

\begin{definition}\label{D:levelNcontrolled}
	A path
	\[
		\mbf Y = \left(\mbf Y^{(0)}, \mbf Y^{(1)}, \ldots, \mbf Y^{(N-1)} \right): [0,T] \to \bigoplus_{k=0}^{N-1} \mcl L((\RR^n)^{\otimes k}; \RR^m)
	\]
	is called a \textit{controlled Besov rough path} over $\RR^m$ with respect to ${\bf X}$, and we write $\mbf \in \mathscr B^\alpha_{pq,\mbf X}([0,T])$, if the ``remainders'' defined by, for $0\le i\le N-1$,
	\[
		\Delta_2(0,T) \ni (s,t) \mapsto R_{s,t}^{i} :=
		\begin{cases}
			\mbf Y_{t}^{(i)} - \mbf Y_{s}^{(i)} - \sum_{j=1}^{N-1-i}\mbf Y_{s}^{(i+j)}\mbf X_{s,t}^{(j)}, & \text{if }0\le i\le N-2,\\
			\mbf Y_{t}^{(N-1)} - \mbf Y_{s}^{(N-1)}, & \text{if }i=N-1,
		\end{cases}
	\]
	satisfy
	\[
		[{\mbf Y}]_{\mathscr B^\alpha_{pq,\mbf X}} := \sum_{i=0}^{N-1} \nor{R^i}{\mbb B^{ (N-i) \alpha}_{{p}/{(N-i)},q/(N-i)}} < \infty.
	\]
	For geometric $\mbf X, \tilde{\mbf X} \in \mbf B^\alpha_{pq}$, $\mbf Y \in \mathscr B^\alpha_{pq,\mbf X}$ and $\tilde{\mbf Y} \in \mathscr B^\alpha_{pq,\tilde{\mbf X}}$, if $(\tilde R^i)_{i=0}^{N-1}$ denote the corresponding remainder terms for $\tilde{\mbf Y}$, we define
	\[
		d_{\mathscr B^\alpha_{pq}, \mbf X, \tilde{\mbf X}} \left( \mbf Y, \tilde{\mbf Y} \right) := \sum_{i=0}^{N-1} \nor{R^i - \tilde R^i}{\mbb B^{ (N-i) \alpha}_{{p}/{(N-i)},q/(N-i)}}.
	\]
\end{definition}

\begin{remark}\label{R:levelNlevel2}
	Definition \ref{D:levelNcontrolled} is consistent with Definition \ref{D:controlledBesovRP} when $N = 2$, and, if $(Y,Y')$ is a level-$2$ controlled rough path in the sense of Definition \ref{D:controlledBesovRP}, then, with the notation of Definition \ref{D:levelNcontrolled}, $Y = \mbf Y^{(0)}$, $Y' = \mbf Y^{(1)}$, $\delta Y = R^1$, and $R^Y = R^0$.
\end{remark}

We now outline the steps required to generalize the results from the level-$2$ case.

First, just as in Theorem \ref{T:Besovrough}, one checks that
\[
	\Delta_2(0,T) \ni (s,t) \mapsto A_{st} := \mbf Y^{(0)}_s \mbf X^{(1)}_{st} + ... + \mbf Y^{(N-1)}_s \mbf X^{(N)}_{st}
\]
satisfies the assumptions of the Besov sewing results, Theorems \ref{T:Besovsewing}, \ref{T:Besovsewingendpoint}, and \ref{T:Besovsewingcontinuous}. This gives the rough integral of $\mbf Y$ against $\mbf X$, and the estimates from the sewing theorems yield control of the difference between the increments of the integral and $A$ in $\mbb B^{(N+1)\alpha}_{p/(N+1), q/(N+1)}$, as in \eqref{roughintegralremainder}, or with $\tau^{(N+1)\alpha}$ modified by a logarithmic correction as in \eqref{roughintegralendpoint} when $\alpha = 1/(N+1)$. By proving a result similar to Lemma \ref{L:controllednorms}, and with repeated use of Lemma \ref{L:roughinterpolate}, one then shows that the rough integral is a controlled Besov rough path in the sense of Definition \ref{D:levelNcontrolled}, which we may denote by $\int {\mbf Y}d \mbf X$.  Yet another application of the sewing lemma shows that the Besov rough integral is a bounded linear map on the space of controlled Besov rough paths, and locally Lipschitz continuous as a function of the integrating Besov rough paths; cf. Theorem \ref{T:stability}. (This step does not rely on the geometric property, and one would proceed identically in a branched setting, still using Theorems \ref{T:Besovsewing}, \ref{T:Besovsewingendpoint}, and \ref{T:Besovsewingcontinuous}, but now with $A$ of a different form.)
 
Next, the stability of controlled Besov rough paths under composition with a regular map $f$ must be established. The correct image $f({\mbf Y})$, as a {\em higher order controlled rough path}, involves $N-1$ derivatives of $f$. In this (geometric) context, the precise form can be found in \cite[Theorem 2.11]{BDFT}, or in formula (4.2) of \cite{BoeGen}. Checking the required controlled Besov regularity then relies on the same analytic ideas as were presented in Proposition \ref{P:composecontrolled}. 


Finally, using the above results, a Picard fixed point argument yields a unique $\mbf X$-controlled path ${\mbf Y} \in \mathscr B^\alpha_{pq,\mbf X}$ satisfying the integral relation
\begin{equation}\label{E:RDEintegral2}
	Y_t = y + \int_0^t f(Y_s) \cdot d \mbf X_s,
\end{equation}
with
\begin{equation}\label{levelNf}
	\left\{
	\begin{split}
	&f \in C^{N,\delta}, \text{ where } \delta \in (0,1] \text{ satisfies}\\
	&(N+\delta)\alpha > 1 \quad \text{and} \quad \delta \alpha > \frac{1}{p} \quad \text{if } \alpha > \frac{1}{N+1} \text{ and}\\
	&\delta = 1 \quad \text{if } \alpha = \frac{1}{N+1} \text{ and } q \le N+1.
	\end{split}
	\right.
\end{equation}

\begin{theorem}\label{T:RDElevelN}
	For $N \in \NN$, assume \eqref{levelNparameters}, and fix a geometric rough path $\mbf X \in \mbf B^\alpha_{pq}$ and $f$ satisfying 	\eqref{levelNf}. Then, for every fixed $y \in \RR^m$, there exists a unique solution $\mbf Y \in \mathscr B^\alpha_{pq,\mbf X}([0,T],\RR^m)$ of \eqref{E:RDEintegral2}. Moreover, there exists a constant $M$ depending only on $N$, $\alpha$, $p$, $q$, $T$, $\nor{f}{C^N}$, and $\normm{\mbf X}_{\mbf B^\alpha_{pq}}$ such that $[\mbf Y]_{\mathscr B^\alpha_{pq,\mbf X}} \le M$. Finally, the solution map is locally Lipschitz continuous in $(y,f,\mbf X) \in \RR^m \times C^{N,\delta} \times \mbf B^\alpha_{pq}$. 
\end{theorem}

\begin{remark}
	Both here and in the general level-$2$ case, under a corresponding local condition on $f$, existence of a unique solution may only hold on $[0,T^*)$ for some explosion time $T^*$.
\end{remark}

\begin{remark}
	The solution is equivalently characterized by a Davie expansion. More precisely, as in Proposition \ref{P:Davies}, an $\mbf X$-controlled rough path $\mbf Y$ is a solution of \eqref{E:RDEintegral2} if and only if the difference of $\delta Y$ with an appropriate expansion belongs to $\mbb B^{(N+1)\alpha}_{p/(N+1), q/(N+1)}$, if $\alpha > \frac{1}{N+1}$, or $\mbb B^{1}_{p/(N+1),\oo;\circ}$ if $\alpha = \frac{1}{N+1}$.
\end{remark}

\appendix

\section{by P. Zorin-Kranich: Pathwise Besov estimate for stochastic integrals}  \label{appendix} 
As a proof of concept, we show here a pathwise Besov space estimate for It\^{o} integrals in the spirit of \cite{arxiv:2008.08897}.
In order to keep things simple, we stay at the level of complexity of the rough path BDG inequality \cite{MR3909973,MR4003122}.

\subsection{Vector-valued BDG inequality}
In this section, we quickly recall how vector-valued inequalities can be deduced from weighted inequalities.
The basic tool for that purpose is the Rubio de Francia extrapolation theorem, the sharp form of which on $\RR^{n}$ was proved in \cite{MR2754896}.
The probabilisitic version below is taken from \cite[Theorem 8.1]{MR3667592}.

A \emph{weight} on a probability space is an integrable function with values in $(0,\infty)$.
To a weight $w$ on a filtered probability space, for any $p\in (1,\infty)$, is associated the \emph{$A_{p}$ characteristic}
\[
Q_{p}(w) = [w]_{A_{p}} :=
\sup_{\tau \text{ stopping time}} \norm{ \EE(w|\mcl F_{\tau}) \EE(w^{1/(1-p)}|\mcl F_{\tau})^{p-1} }_{L^{\infty}}
\in [1,\infty].
\]
A weight is said to be an $A_{p}$ weight if its $A_{p}$ characteristic is finite.
We will mostly use the theory of weights as a black box, with the exception of the facts that $[w]_{A_{2}} = [w^{-1}]_{A_{2}}$ for any weight $w$, and $[1]_{A_{p}} = 1$ for any $p\in (1,\infty)$.

The extrapolation theorem for random variables reads as follows.
\begin{theorem}[Rubio de Francia extrapolation {\cite[Theorem 8.1]{MR3667592}}]
\label{thm:extrapolation}
Let $r\in (1,\infty)$ and $N_{r} : (0,\infty) \to (0,\infty)$ an increasing function.
Then, for every $p \in (1,\infty)$, there exists an increasing function $N_{p} : (0,\infty) \to (0,\infty)$ such that the following holds.

Let $X,Y$ be positive random variables defined on the same filtered probability space $\Omega$.
If for every weight $w \in A_{r}$ on $\Omega$, we have
\begin{equation}
\label{eq:RdF-weighted-hyp}
\norm{X}_{L^{r}(w)} := \EE( \abs{X}^{r} w )^{1/r} \leq N_{r}([w]_{A_{r}}) \norm{Y}_{L^{r}(w)},
\end{equation}
then, for every weight $w \in A_{p}$ on $\Omega$, we have
\[
\norm{X}_{L^{p}(w)} \leq N_{p}([w]_{A_{p}}) \norm{Y}_{L^{p}(w)}.
\]
\end{theorem}
In order to discourage an overly optimistic interpretation of the notation, we note that the function $N_{p}$ above depends (in a known, explicit way) on $r,p,N_{r}$.
In particular, if we apply this theorem twice to go from $L^{r}$ to $L^{p}$, and then back to $L^{r}$, we will get a worse (larger) function.

Suppose that we have measurable families of random vaiables $X^{(\tau,s)}$ and $Y^{(\tau,s)}$ such that \eqref{eq:RdF-weighted-hyp} holds uniformly in $\tau,s$, which vary over some further measure space.
Applying Theorem~\ref{thm:extrapolation}, we obtain a similar inequality with $p \in (1,\infty)$ in place of $r$.
By Fubini's theorem, it follows that
\[
\norm{ \norm{X^{(\tau,s)}}_{L^{p}_{s}} }_{L^{p}(w)}
\leq N_{p}([w]_{A_{p}}) \norm{ \norm{Y^{(\tau,s)} }_{L^{p}_{s}}}_{L^{p}(w)},
\]
uniformly in $\tau$.
Applying Theorem~\ref{thm:extrapolation} again with some $q\in (1,\infty)$, we obtain
\[
\norm{ \norm{X^{(\tau,s)}}_{L^{p}_{s}} }_{L^{q}(w)}
\leq N_{q}([w]_{A_{q}}) \norm{ \norm{Y^{(\tau,s)} }_{L^{p}_{s}}}_{L^{q}(w)}.
\]
By Fubini's theorem, it follows that
\[
\norm{ \norm{ \norm{X^{(\tau,s)}}_{L^{p}_{s}} }_{L^{q}_{\tau}} }_{L^{q}(w)}
\leq N_{q}([w]_{A_{q}}) \norm{ \norm{ \norm{Y^{(\tau,s)} }_{L^{p}_{s}} }_{L^{q}_{\tau}} }_{L^{q}(w)}.
\]
Applying Theorem~\ref{thm:extrapolation} one more with some $r\in (1,\infty)$, we obtain
\begin{equation}
\label{eq:extrapolation-LpLqLr}
\norm{ \norm{ \norm{X^{(\tau,s)}}_{L^{p}_{s}} }_{L^{q}_{\tau}} }_{L^{r}(w)}
\leq N_{r}([w]_{A_{q}}) \norm{ \norm{ \norm{Y^{(\tau,s)} }_{L^{p}_{s}} }_{L^{q}_{\tau}} }_{L^{r}(w)}.
\end{equation}
One can iterate this indefinitely, obtaining vector-valued inequalities with ever more nested norms, but we will stop here and specialize to $w=1$.

Now we state the weighted martingale inequalities that we will use.
The sharp $A_{2}$ weighted martingale maximal inequality \cite{MR3916937} says that, for any real-valued martingale $f$ and any weight $w\in A_{2}$, we have
\begin{equation}
\label{eq:mart-max-A2}
\norm{Mf}_{L^{2}(w)} \lesssim [w]_{A_{2}} \norm{f}_{L^{2}(w)}.
\end{equation}

The endpoint weighted versions of the BDG inequalities were proved by Osekowski in \cite{MR3567926,MR3688518}.
In particular, in \cite{MR3688518} it was proved that $\EE Mf \cdot w \lesssim \EE Sf \cdot Mw$ for any real-valued martingale $f$ and any weight $w$, where $M$ denotes the martingale maximal and $S$ the martingale square function.
It is well-known that such estimates imply $A_{p}$ weighted estimates of the form \eqref{eq:RdF-weighted-hyp}.
Indeed, by the result of \cite{MR3688518}, for any real-valued martingale $f$ and weights $w,u$, we have
\[
\EE Mf \cdot u \cdot w
\lesssim
\EE Sf \cdot M(uw)
\leq
(\EE (Sf)^{2} w)^{1/2} (\EE (M(uw))^{2} w^{-1})^{1/2}.
\]
By \eqref{eq:mart-max-A2}, we have
\[
(\EE (M(uw))^{2} w^{-1})^{1/2}
\lesssim
[w^{-1}]_{A_{2}} (\EE (uw)^{2} w^{-1})^{1/2}
=
[w]_{A_{2}} (\EE u^{2} w)^{1/2}.
\]
By duality, it follows that
\[
L^{2}_{w} Mf \lesssim [w]_{A_{2}} L^{2}_{w} Sf.
\]
By extrapolation \eqref{eq:extrapolation-LpLqLr}, we obtain the following vector-valued BDG inequality.
\begin{theorem}[Vector-valued BDG]
\label{thm:vv-BDG}
Let $p,q,r \in (1,\infty)$ and let $X^{(\tau,s)}$ be a measurable family of real-valued martingales defined on the same probability space $\Omega$.
Then, we have
\[
\norm{ \norm{ \norm{M X^{(\tau,s)}}_{L^{p}_{s}} }_{L^{q}_{\tau}} }_{L^{r}(\Omega)}
\lesssim_{p,q,r}
\norm{ \norm{ \norm{S Y^{(\tau,s)} }_{L^{p}_{s}} }_{L^{q}_{\tau}} }_{L^{r}(\Omega)}.
\]
\end{theorem}
It is likely possible to allow the exponents to be $1$, but this would require a different proof.

\subsection{Martingale paraproduct estimate} \label{app:besov}

For a discrete time adapted process $(F_{s,t})$ and a discrete time martingale $(g_{n})$, the martingale paraproduct is defined by
\begin{equation}
\label{eq:paraprod}
\Pi_{s,t}(F,g) := \sum_{s < j \leq t} F_{s,j-1} dg_{j}
=
\sum_{s \leq j < t} F_{s,j} (g_{j+1}-g_{j}).
\end{equation}

Using Theorem~\ref{thm:vv-BDG} in place of \cite[Lemma 2.4]{arxiv:2008.08897}, we easily obtain the following analog of \cite[Proposition 2.5]{arxiv:2008.08897}.
\begin{proposition}
\label{prop:vv-pprod}
Let $0 < r,r_{1} \leq \infty$, $1 < r_{0},p,q,p_{0},q_{0} < \infty$, $1 < p_{1},q_{1} \leq \infty$.
Let $\gamma_{0},\gamma_{1}\in\RR$.
Assume
\[
1/p = 1/p_{0} + 1/p_{1},
\quad
1/q = 1/q_{0} + 1/q_{1},
\quad
1/r = 1/r_{0} + 1/r_{1}.
\]
Then, for any measurable families of discrete time martingales $g^{(x,y)}$, two-parameter adapted process $F^{(x,y)}$, and stopping times $\tau'_{x,y}\leq \tau_{x,y}$ on a filtered probability space $\Omega$, we have
\begin{equation}
\label{eq:vv-pprod}
\begin{split}
\MoveEqLeft
\norm{ L^{q}_{x} x^{-\gamma_{0}-\gamma_{1}} L^{p}_{y} \sup_{\tau'_{x,y} < t \leq \tau_{x,y}} \abs{\Pi(F^{(x,y)},g^{(x,y)})_{\tau'_{x,y},t}} }_{L^{r}(\Omega)}
\\ &\lesssim
\norm{ L^{q_{1}}_{x} x^{-\gamma_{1}} L^{p_{1}}_{y} \sup_{\tau'_{x,y} < t \leq \tau_{x,y}} \abs{F^{(x,y)}_{\tau'_{x,y},t}} }_{L^{r_{1}}(\Omega)}
\norm{ L^{q_{0}}_{x} x^{-\gamma_{0}} L^{p_{0}}_{y} Sg^{(x,y)}_{\tau'_{x,y},\tau_{x,y}} }_{L^{r_{0}}(\Omega)}.
\end{split}
\end{equation}
\end{proposition}

Proposition~\ref{prop:vv-pprod} readily extends to continuous time processes upon replacing $\Pi_{s,t}(F,g)$ by the It\^{o} integral $\int_{s}^{t} F_{s,u-} \dif g_{u}$, if $F_{s,\cdot}$ is \cadlag{}.
Indeed, as explained in \cite[Section 4]{arxiv:2008.08897}, Proposition~\ref{prop:vv-pprod}, together with a stopping time construction, provides a construction of the It\^{o} integral that comes directly with the claimed estimate.

Recall Definition~\ref{D:Besovtwoparameter}:
\[
\norm{A}_{\mbb B_{pq}^{\gamma}} := \Bigl[ \int_{0}^{T} \bigl(\tau^{-\gamma} \sup_{0<h\leq\tau} \norm{A_{\cdot,\cdot+h}}_{L^{p}[0,T-h]}\bigr)^{q} \frac{\dif\tau}{\tau} \Bigr]^{1/q}.
\]

For a martingale $g$ and an adapted two-parameter process $F$ on a time interval $[0,T]$, denote
\[
A_{s,t} := \int_{u=s}^{t} F_{s,u-} \dif g_{u}.
\]
In order to simplify notation, we extend $g_{t}=g_{T}$ and $F_{s,t}=F_{s,T}$ for $t>T$.

Let $\gamma=\gamma_{0}+\gamma_{1}$ and let $p,q,r$, etc.\ be as in Proposition~\ref{prop:vv-pprod}.
By the continuous time version of Proposition~\ref{prop:vv-pprod}, we obtain
\begin{align}
\notag
\norm{ \norm{ A }_{\mbb B^{\gamma}_{p,q}} }_{L^{r}(\Omega)}
&\leq
\norm{ \left[ \int_{0}^{T} \left( \tau^{-\gamma} \norm{\sup_{0<h\leq\tau} \abs{A_{\cdot,\cdot+h}}}_{L^{p}[0,T]}\right)^{q} \frac{\dif\tau}{\tau} \right]^{1/q} }_{L^{r}(\Omega)}
\\ &\lesssim \label{eq:vv-pprod:cont:F}
\norm{ \left[ \int_{0}^{T} \left(\tau^{-\gamma_{1}} \norm{\sup_{0<h\leq\tau} \abs{F_{\cdot,\cdot+h}} }_{L^{p_{1}}[0,T]}\right)^{q_{1}} \frac{\dif\tau}{\tau} \right]^{1/q_{1}} }_{L^{r_{1}}(\Omega)}
\\ &\quad \cdot \label{eq:vv-pprod:cont:g}
\norm{ \left[ \int_{0}^{T} \left(\tau^{-\gamma_{0}} \norm{Sg_{\cdot,\cdot+\tau}}_{L^{p_{0}}[0,T]}\right)^{q_{0}} \frac{\dif\tau}{\tau} \right]^{1/q_{0}} }_{L^{r_{0}}(\Omega)}.
\end{align}
Setting $F=1$, $\gamma_{1}=0$, and $p_{1}=q_{1}=r_{1}=\infty$, this recovers the non-endpoint Besov norm BDG inequality \eqref{eq:Besov-BDG:M<S}; we refer to \cite[Thm.\ 1]{MR1277166} for the endpoints.
When $g$ is the Brownian motion, the norm \eqref{eq:vv-pprod:cont:g} is finite for any $\gamma_{0}<1/2$ and $q_{0} \in (1,\infty)$.

Now we can show the anisotropic version of Theorem~\ref{thm:Martbesovrough}.
\begin{theorem} \label{thm:besov-pprod}
Let $p,q,r,p_{0},q_{0},r_{0},p_{1},q_{1},r_{1}$ be as in Proposition~\ref{prop:vv-pprod}.
Let $\gamma_{0},\gamma_{1} \in [0,\infty)$ with $\gamma_{1}>1/p_{1}$ and $\gamma=\gamma_{0}+\gamma_{1}$.
Let $f$ be a \cadlag{} adapted process, $g$ a \cadlag{} martingale, and
\[
A_{s,t} := \int_{u=s}^{t} \delta f_{s,u-} \dif g_{u}.
\]
Then, with $S$ defined as in \eqref{eq:S:cont}, we have
\begin{equation} \label{eq:Besov-paraproduct-est}
\norm{ \norm{ A }_{\mbb B^{\gamma}_{p,q}} }_{L^{r}(\Omega)}
\lesssim
\norm{ 
\norm{f}_{B^{\gamma_{1}}_{p_{1},q_{1}}} 
}_{L^{r_{1}}(\Omega)}
\times 
\norm{ 
\norm{S g}_{\mbb B^{\gamma_{0}}_{p_{0},q_{0}}} 
}_{L^{r_{0}}(\Omega)}.
\end{equation}
\end{theorem}
If $f$ is also a martingale, then, by \eqref{eq:Besov-BDG:M<S}, $f$ can be replaced by $Sf$ on the right-hand side of \eqref{eq:Besov-paraproduct-est}.
It is instructive to note that, in this case, the product $(s,t) \mapsto \delta f_{s,t}  \delta g_{s,t} $, related to the ``paraproduct'' $A$ by It\^o's product rule, has the same $2$-parameter regularity.

\begin{proof}[Proof of Theorem~\ref{thm:besov-pprod}]
Specializing the previous discussion to $F=\delta f$, it suffices to estimate \eqref{eq:vv-pprod:cont:F}.
Assuming that $\gamma_{1}>1/p_{1}$, we will show that
\begin{equation}
\label{eq:Besov-sup-inside}
\left[ \int_{0}^{T} \left(\tau^{-\gamma_{1}} \norm{\sup_{0<h\leq\tau} \abs{\delta f_{\cdot,\cdot+h}} }_{L^{p_{1}}[0,T]}\right)^{q_{1}} \frac{\dif\tau}{\tau} \right]^{1/q_{1}}
\lesssim
\norm{f}_{B^{\gamma_{1}}_{p_{1},q_{1}}}.
\end{equation}
Indeed, the hypothesis ensures that $f$ is (a.e.\ equal to) a $(\gamma_{1}-1/p_{1})$-H\"older function (see e.g.\ Proposition~\ref{P:Besovembedding}).
Therefore, the supremum can be replaced by the supremum over $h\in \tau \cdot ((0,1] \cap \ZZ[1/2])$, where $\ZZ[1/2]$ is the set of rational numbers with denomenator that is a power of $2$.
By the monotone convergence theorem, it suffices to consider $h\in \tau \cdot D_{N}$, where $D_{N}=\{1,\dotsc,2^{N}\}/2^{N}$, as long as we obtain a bound independent of $N$.
Note that
\begin{align*}
\sup_{h\in \tau D_{N+1}} \abs{\delta f_{s,s+h}}
&=
\sup_{h\in (\tau/2) D_{N}  \cup ((\tau/2) D_{N} + \tau/2) } \abs{\delta f_{s,s+h}}
\\ &\leq
\sup_{h\in (\tau/2) D_{N}} (\abs{\delta f_{s,s+h}}^{p_{1}} + \abs{\delta f_{s,s+\tau/2+h}}^{p_{1}})^{1/p_{1}}
\\ &\leq
\sup_{h \in (\tau/2) D_{N}} (\abs{\delta f_{s,s+h}}^{p_{1}} + \abs{\delta f_{s+\tau/2,s+\tau/2+h}}^{p_{1}})^{1/p_{1}} + \abs{\delta f_{s,s+\tau/2}}.
\end{align*}
Taking $L^{p_{1}}$ norm in $s$, we obtain
\[
\norm{\sup_{h \in \tau D_{N+1}} \abs{\delta f_{\cdot,\cdot+h}} }_{L^{p_{1}}}
\leq
2^{1/p_{1}} \norm{\sup_{h \in (\tau/2)D_{N}} \abs{\delta f_{\cdot,\cdot+h}} }_{L^{p_{1}}}
+
\norm{\delta f_{\cdot,\cdot+\tau/2} }_{L^{p_{1}}}.
\]
Iterating this, we obtain a uniform estimate in $N$ provided that $\gamma_{1}>1/p_{1}$.
This shows \eqref{eq:Besov-sup-inside}, which in turn implies \eqref{eq:Besov-paraproduct-est}.
\end{proof}

\bibliography{besovrough}{}
\bibliographystyle{alpha}

\end{document}